\documentclass[a4paper, ]{amsart}
\usepackage{mathrsfs,amssymb}
\usepackage{multicol}\multicolsep=0pt
\usepackage{pstricks,pst-node}

\usepackage[sort]{cite}

\def\crulefill{\leavevmode\leaders\hrule height 1pt\hfill\kern 0pt}
\long\def\QUERY#1{%
\leavevmode\newline%
\noindent$\star\star\star$\thinspace\textsf{Comment/Query}\crulefill\newline%
   \space #1\newline\hbox to 120mm{\crulefill}$\star\star\star$\newline
}

\numberwithin{equation}{section} \theoremstyle{definition}
\newtheorem{Defn}[equation]{Definition}%[section]
\theoremstyle{plain}
\newtheorem{Prop}[equation]{Proposition}
\newtheorem{Theorem}[equation]{Theorem}
\newtheorem{Assumption}[equation]{Assumption}
\newtheorem{Lemma}[equation]{Lemma}
\newtheorem{Cor}[equation]{Corollary}

\def\Case#1{\medskip\noindent\textbf{Case #1}:\leavevmode\newline}

\makeatletter
\def\enumerate{\begingroup\ifnum\@enumdepth>3\@toodeep\else
      \advance\@enumdepth\@ne
      \edef\@enumctr{enum\romannumeral\the\@enumdepth}%
      \topsep\z@\parskip\z@
      \list{\csname label\@enumctr\endcsname}
        {\@nmbrlisttrue\let\@listctr\@enumctr
         \parsep\z@\itemsep\z@\topsep\z@
         \setcounter{\@enumctr}{0}
         \def\makelabel##1{\hss\llap{\rm ##1}}
       }\fi}

\makeatother

\let\epsilon=\varepsilon
\def\({\big(}
\def\){\big)}

\def\N{\mathbb N}
\def\R{\mathbb R}

\def\Z{\mathbb Z}

\def\0{\underline{0}}
\def\bu{\mathbf u}

\def\Bcal{\mathfrak B}
\def\Dcal{\mathcal D}
\def\Ef{{\mathcal E}}

\def\H{\mathscr H}
\let\proj=\varepsilon
\def\Nrf{\mathbb N_r^{f,n}}
\def\Sym{\mathfrak S}
\def\W{\mathscr W}

\def\Wlam{\W_{r,n}^{\gdom (f, \lambda)}}

\def\simk{\overset{k}\sim}
\def\simkk{\overset{k+1}\sim}
\def\siml{\overset{l}\sim}
\DeclareMathOperator{\End}{End} 
\DeclareMathOperator{\Rad}{Rad} \DeclareMathOperator*{\Res}{Res}

\let\gdom\rhd
\let\gedom\unrhd

\def\m{\mathfrak m}
\def\floor#1{\lfloor\tfrac#1\rfloor}
\def\UPD{\mathscr{T}^{ud}}
\def\Std{\mathscr{T}^{s}}

\def\bfs{\mathbf s}
\def\bft{\mathbf t}
\def\bfu{\mathbf u}
\def\bfv{\mathbf v}

\def\ts{\tilde\s}
\def\t{\mathfrak t}
\def\u{\mathfrak u}
\def\v{\mathfrak v}

\def\ts{\mathbf t}

{\catcode`\|=\active
  \gdef\set#1{\mathinner{\lbrace\,{\mathcode`\|"8000%
                                   \let|\midvert #1}\,\rbrace}}
  \gdef\seT#1{\mathinner{\Big\lbrace\,{\mathcode`\|"8000%
                                   \let|\midverT #1}\,\Big\rbrace}}
}
\def\midvert{\egroup\mid\bgroup}
\def\midverT{\egroup\,\Big|\,\bgroup}

\def\Set[#1]#2|#3|{\Big\{\ #2\ \Big| \
           \vcenter{\hsize #1mm\centering #3}\Big\}}

\def\map#1#2{\,{:}\,#1\!\longrightarrow\!#2}

\let\epsilon=\varepsilon
\def\({\big(}
\def\){\big)}

\def\N{\mathbb N}
\def\R{\mathbb R}

\def\Z{\mathbb Z}

\def\0{\underline{0}}
\def\bu{\mathbf u}

\def\Bcal{\mathfrak B}
\def\Dcal{\mathcal D}
\def\Ef{{\mathcal E}}

\def\H{\mathscr H}
\let\proj=\varepsilon
\def\Nrf{\mathbb N_r^{f,n}}

\def\Sym{\mathfrak S}
\def\W{\mathscr B}

\def\simk{\overset{k}\sim}
\def\simkk{\overset{k+1}\sim}

\let\gdom\rhd
\let\gedom\unrhd

\def\m{\mathfrak m}
\def\floor#1{\lfloor\tfrac#1\rfloor}
\def\UPD{\mathscr{T}^{ud}}
\def\Std{\mathscr{T}^{std}}

\def\s{\mathfrak s}
\def\ts{\tilde\s}
\def\t{\mathfrak t}
\def\u{\mathfrak u}
\def\v{\mathfrak v}

{\catcode`\|=\active
  \gdef\set#1{\mathinner{\lbrace\,{\mathcode`\|"8000%
                                   \let|\midvert #1}\,\rbrace}}
  \gdef\seT#1{\mathinner{\Big\lbrace\,{\mathcode`\|"8000%
                                   \let|\midverT #1}\,\Big\rbrace}}
}
\def\midvert{\egroup\mid\bgroup}
\def\midverT{\egroup\,\Big|\,\bgroup}

\def\Set[#1]#2|#3|{\Big\{\ #2\ \Big| \
           \vcenter{\hsize #1mm\centering #3}\Big\}}

\def\map#1#2{\,{:}\,#1\!\longrightarrow\!#2}

\title{The representations of cyclotomic BMW algebras}

\author{Hebing Rui}
\address{H. Rui, Department of Mathematics, East China Normal University, %
         200062 Shanghai, P.R. China.}
\email{hbrui@math.ecnu.edu.cn}

\author{Jie Xu }
\address{J. Xu, Department of Mathematics, East China Normal University, %
         200062 Shanghai, P.R. China.}
\email{52060601009@student.ecnu.edu.cn}

\begin{document}

\begin{abstract} In this paper, we  prove that the cyclotomic BMW algebras $\W_{2p+1, n}$  are cellular
in the sense of \cite{GL}. We also classify the irreducible
$\W_{2p+1, n}$-modules over a field.
\end{abstract}
 \sloppy \maketitle
%%%%%%%%%%%%%%%%%%%%%%%%%%%%%%%%%%%%%%%%%%%%%%%%%%%%%%%%%%%%%%%%%%%%%%%%%%
\sloppy \maketitle

\section{Introduction}

Let $r, n$ be two positive integers.
Haering--Oldenburg~\cite{HO:cycBMW} introduced a class of finite
dimensional algebras $\W_{r, n}$ called \textsf{cyclotomic BMW
algebras} in order to study the link invariants. Such algebras  are
associative algebras over a commutative ring $R$,  which are the
cyclotomic quotients of affine BMW algebras in \cite{HO:cycBMW} and
\cite{GoodmanHauschild}.

Motivated by Ariki, Mathas and Rui's work on
 cyclotomic Nazarov--Wenzl algebras~\cite{AMR},  we began to
study  $\W_{r, n}$ in August of 2004.  By using the method  given in
section~3  in \cite{AMR}, we constructed all possible irreducible
modules for $\W_{r, 2}$.\footnote{ This is the main result of
~\cite{Xu}.} A by-product is the definition of $\bu$-admissible
conditions given in Definition~\ref{add1} for certain parameters in
the ground ring $R$. We remark that there are two papers \cite{GH1}
and \cite{WY} in Arxiv which were posted at the end of  2006. In
those papers, there are  some  results for $\bu$-admissible
conditions.

In this paper,  we  will construct a class of rational functions,
which are similar to those  in \cite{AMR} for cyclotomic
Nazarov--Wenzl algebras. Such rational functions can be used to
construct the seminormal representations for $\W_{r, n}$ under
certain  conditions given in Lemma~\ref{generic u} and
Assumption~\ref{be}. By Wedderburn--Artin Theorem  on the semisimple
finite dimensional algebras, we  prove that the rank of  $\W_{r, n}$
is no less than $r^n (2n-1)!!$ if $\W_{r, n}$ is free. In order to
prove that $\W_{r, n}$ is free over $R$ with rank $r^n (2n-1)!!$, we
have to find a subset $S$ of $\W_{r, n}$ whose cardinality is $r^n
(2n-1)!!$ such that  $\W_{r, n}$ is generated by $S$  as $R$-module.
 If so, then $\W_{r, n}$ is free
over $R$ with rank $r^n(2n-1)!!$ as required.

In order to construct such a subset $S$,  we construct a class of
quotient modules of $\W_{r, n}$. Such modules can be used to
construct a filtration of two--sided ideals of $\W_{r, n}$.
Therefore, we can lift the set of generators for all such quotient
modules to get a set of generators for $R$-module $\W_{r, n}$.
Together with our previous results on the seminormal representations
for $\W_{r, n}$, we  prove that $\W_{r, n}$ is a cellular algebra in
the sense of \cite{GL}. In particular, $\W_{r, n}$ is free over $R$
with rank $r^n(2n-1)!!$ as required. For some technique reasons, we
have to assume that $r$ is odd when we construct the quotient
modules for $\W_{r, n}$. Using the results on the  representation
theory of cellular algebras given in \cite{GL}, we find the
relationship between the irreducible modules for Ariki--Koike
algebras and $\W_{r, n}$. This will enable us to classify the
irreducible $\W_{r, n}$--modules over a field.

We organize the paper as follows. In section~2, we recall the
definition of $\W_{r, n}$ over a commutative ring. In section~3, we
give all possible irreducible modules for $\W_{r, 2}$. We  also give
the definition of $\bu$-admissible conditions. Under this assumption
together with two conditions in Lemma~\ref{generic u} and
Assumption~\ref{be}, we construct the seminormal forms for $\W_{r,
n}$ in section~4. In section~5, we construct the quotient modules
for $\W_{r, n}$ for all odd positive integers $r$. At the end of
section~5, we  prove that such a $\W_{r, n}$ is cellular in the
sense of \cite{GL}.  Finally, we classify the irreducible $\W_{r,
n}$--modules for odd $r$  in section~6.

When we wrote the paper, we received Dr. Shona Yu's Ph. D thesis
entitled ``The cyclotomic Birman-Murakami-Wenzl Algebras"~
\cite{Yu1}. In her thesis, Yu has  proved that $\W_{r, n}$ is
cellular algebra for all $r$ by using different method.  However, we
could not understand why she had assumed that $\omega_0$, one of
parameters appeared in the definition of $\W_{r, n}$,  is invertible
when she constructed a subset of $\W_{r, n}$ which generates $\W_{r,
n}$ as $R$-module.

\textbf{Acknowledgment:} The first author was supported in part by
NSFC  and NCET. He also thanks Professor Ariki and Professor Mathas
for their collaboration in \cite{AMR}.

\section{Cyclotomic BMW algebras}
Throughout the paper,  we fix two positive integers $r, n$ and a
commutative ring $R$ with multiplicative identity~$1_R$.

\begin{Defn}
\label{Waff relations} Suppose that  $R$ is a commutative ring which
contains $q^{\pm 1}, u_1^{\pm 1}, \dots, u_r^{\pm 1},  \varrho^{\pm
1}, \delta^{\pm 1}$ with $\delta=q-q^{-1}$. Fix
$\Omega=\set{\omega_a|a\in \mathbb Z}\subseteq R$ such that
$$\omega_0=1-\delta^{-1}(\varrho-\varrho^{-1}).$$ The
\textsf{cyclotomic  BMW algebras} $\W_{r, n}$ is the unital
associative $R$--algebra  generated  by $\set{T_i,E_i, X_j^{\pm
1}|1\le i<n \text{ and }1\le j\le n } $ subject to the following
relations:

\begin{enumerate}
\item $X_i X_{i}^{-1}=X_{i}^{-1}X_i=1$ for $1\le i\le n$.
    \item (Kauffman skein relation )
$1=T_i^2-\delta T_i +\delta \varrho E_i$, for $1\le i<n$.
    \item (braid relations)
\begin{enumerate}
\item $T_iT_j=T_jT_i$ if $|i-j|>1$,
\item $T_iT_{i+1}T_i=T_{i+1}T_iT_{i+1}$, for $1\le i<n-1$,
\item $T_iX_j=X_jT_i$ if $j\ne i,i+1$.
\end{enumerate}
    \item (Idempotent relations)
$E_i^2=\omega_0E_i$, for $1\le i<n$.
    \item (Commutation relations)  $X_iX_j=X_jX_i$, for $1\le i,j\le n$.

\item (Skein relations)
\begin{enumerate}\item      $T_iX_i-X_{i+1}T_i=\delta X_{i+1}
(E_i-1)$,   for $1\le i<n$,
    \item     $X_iT_i-T_iX_{i+1}=\delta (E_i-1) X_{i+1}$, for $1\le i<n$.
    \end{enumerate}
    \item (Unwrapping relations)
        $E_1X_1^aE_1=\omega_a E_1$, for $a\in \mathbb Z$.
    \item (Tangle relations)
\begin{enumerate}
\item $E_iT_i=\varrho E_i=T_iE_i$, for $1\le i\le n-1$,

\item $E_{i+1}E_i=E_{i+1}T_iT_{i+1}=T_iT_{i+1} E_i$,for $1\le i\le n-2$.
\end{enumerate}

    \item (Untwisting relations)\begin{enumerate} \item
        $E_{i+1}E_iE_{i+1}=E_{i+1}$  for $1\le i\le n-2$,
        \item $E_iE_{i+1}E_i=E_i$, for $1\le i\le
        n-2$.\end{enumerate}
    \item (Anti--symmetry relations) $E_iX_iX_{i+1}=E_i=X_iX_{i+1}E_i$, for $1\le i<n$.
\item (Cyclotomic relation) $(X_1-u_1)(X_1-u_2)\cdots(X_1-u_r)=0$
\end{enumerate}
\end{Defn}

By Definition~\ref{Waff relations}(b)(f)(h), we have $X_i=T_{i-1}
X_{i-1}T_{i-1}$ for $2\le i\le n$. Therefore, $\W_{r, n}$ can also
be generated by $E_i, T_i, X_1$ for $1\le i\le n-1$. We will not
need this fact. What we will need is $X_i=T_{i-1} X_{i-1}T_{i-1}$
later on.

$\W_{r, n}$ was introduced by Haering--Oldenburg in order to study
the link invariants. It was re-defined by Goodman and Hauschild as
the quotient algebra of affine BMW algebra in
\cite{GoodmanHauschild}. Further,  Goodman and Hauschild
\cite{GoodmanHauschild} have proved that  affine BMW algebras are
free over $R$ by showing that they are isomorphic to the affine
Kauffman tangle algebras. We do not recall such  result since we do
not need it when we discuss the $\W_{r, n}$ later on.

The main purpose of this paper is to develop the representation
theory of $\W_{r, n}$ by using the method in \cite{AMR}.

\begin{Lemma} \label{antiinvolution}
There is  a unique R-linear anti-isomorphism $\ast: \W_{r,n}
\longrightarrow \W_{r,n}$ such that $T_i^\ast=T_i, E_i^\ast=E_i$ and
$X_j^\ast=X_j$ for all positive integers $ i<n$ and $j\leq
n$.\end{Lemma}

\begin{proof} By checking the defining relations for $\W_{r, n}$, $\ast$ is an $R$-linear anti-automorphism
if $E_iT_{i+1}T_i=E_iE_{i+1}=T_{i+1}T_iE_{i+1}$ for $1\le i\le n-2$.
In fact, By Definition~\ref{Waff relations}(h)(ii),
$E_{i+1}E_i=E_{i+1}T_iT_{i+1}$. So, $E_i=E_iE_{i+1}T_iT_{i+1}$ and
$E_iT_{i+1}^{-1}=E_iE_{i+1}T_i$. By Definition~\ref{Waff
relations}(b) and (h)(i), $E_i(T_{i+1}-\delta+\delta
E_{i+1})=E_iE_{i+1}(T_i^{-1}+\delta-\delta E_i)$. By
Definition~\ref{Waff relations}(i)(ii) again,
$E_iT_{i+1}=E_iE_{i+1}T_i^{-1}$.  In other words,
$E_iT_{i+1}T_i=E_iE_{i+1}$.  Similarly, we  can prove
$E_iE_{i+1}=T_{i+1}T_iE_{i+1}$.
\end{proof}

\begin {Lemma}\label{relations2} Given   positive integers $k\le n-1$ and $a$.
\begin{itemize}
\item [(1)] $T_kX_k^a=X_{k+1}^aT_k+\sum_{i=1}^a\delta
X_{k+1}^i(E_k-1)X_k^{a-i}$, \item [(2)]
$T_k^{-1}X_k^a=X_{k+1}^aT_k^{-1}+\sum_{i=1}^a\delta
X_{k+1}^{a-i}(E_k-1)X_k^i$,\item [(3)] $E_kX_k^a T_k=\varrho
E_kX_k^{-a}+\delta \sum_{i=1}^a E_kX_k^{a-i}E_kX_k^{-i}-\delta
\sum_{i=1}^a E_kX_k^{a-2i}$.
  \item [(4)] $T_kX_k^{-a}=X_{k+1}^{-a}T_k-\sum_{i=1}^{a}\delta
X_{k+1}^{-a+i}(E_k-1)X_k^{-i}$,
\item[(5)] $T_k^{-1}X_k^{-a}=X_{k+1}^{-a}T_k^{-1}-\sum_{i=1}^{a}\delta
X_{k+1}^{-i}(E_k-1)X_k^{-a+i}$,\item[(6)] $E_kX_k^{-a} T_k=\varrho
E_kX_k^{a}-\delta \sum_{i=1}^{a} E_kX_k^{-i}E_kX_k^{a-i}+\delta
\sum_{i=1}^{a} E_kX_k^{a-2i}$.\end{itemize}
\end{Lemma}

\begin{proof} (1) can be proved by induction on $a$. (2) follows
from (1) and \ref{Waff relations}(b). Applying anti-automorphism
$\ast$ on (1), we get the formula for $X_k^a T_k$. Multiplying $E_k$
on such a formula, we get (3). (4) follows from (1). (5) follows
from (4) and \ref{Waff relations}(b). (6) follows from (4).
\end{proof}

Acting  $E_1$ on  both sides of $X_1^b \prod_{i=1}^r(X_1-u_i)=0$ for
$b\in \mathbb Z$,  and Lemma~\ref{relations2}(1) for $k=1$,
respectively, we have
\begin{equation} \label{adm1}\begin{cases} & \sum_{s=0}^{r} (-1)^{r-s} \sigma_{r-s} (\bu)
\omega_{s+b} E_1=0\\ & \omega_a E_1=\omega_{-a}E_1+\sum_{i=1}^a
\varrho^{-1}\delta (\omega_{a-i} \omega_{-i}-\omega_{a-2i})E_1,\\
\end{cases}
\end{equation}
 where $\sigma_i(\bu)$ is the $i$-th basic symmetric
polynomial in $u_1, u_2, \cdots, u_r$.

If we assume that  $E_1$ is not a torsion element, then
\begin{equation}\label{ad1}
\begin{cases} & \sum_{s=0}^{r} (-1)^{r-s} \sigma_{r-s} (\bu)
\omega_{s+b}=0, \forall\  b\in \mathbb Z, \\
& \omega_a=\omega_{-a}+\sum_{i=1}^a \varrho^{-1}\delta (\omega_{a-i}
\omega_{-i}-\omega_{a-2i}).\\
\end{cases}
\end{equation}

\begin{Defn}\label{admissble}
The parameters $\omega_a\in R, a\in \mathbb Z$  are called
\textsf{admissible} if they satisfy (\ref{ad1}).
\end{Defn}

In \cite{AK}, Ariki and Koike introduced  $\H_{r,n}(\bu):=\H_{r,
n}$, \textsf{the cyclotomic Hecke algebra} of type $G(r, 1, n)$ or
the  \textsf{Ariki-Koike algebra}. By definition, it  is the unital
associative $R$-algebra generated by $y_1, \dots, y_n $ and $g_1,
g_2, \dots, g_{n-1}$ subject to the following relations:
\begin{enumerate}
\item $(g_i-q)(g_i+q^{-1})=0$, if $1\le i\le n-1$,
\item $g_ig_j=g_jg_i$, if $|i-j|>1$,
\item $g_ig_{i+1}g_i=g_{i+1}g_ig_{i+1}$, for $1\leq i<n-1$,
\item $g_iy_j=y_jg_i$, if $j\neq i,i+1$,
\item $y_iy_j=y_jy_i$, for $1\leq i,j\leq n$,
\item $y_{i+1}=g_iy_ig_i$,
for $1\leq i\leq n-1$,
\item $(y_1-u_1)(y_1-u_2)\dots(y_1-u_r)=0$.
\end{enumerate}
In this paper, since we are assuming that $u_1, \dots, u_r$ are
invertible in $R$, $y_i$'s are invertible in $\H_{r, n}$.

Let $\langle E_1\rangle$ be the two-sided ideal of $\W_{r, n}$
generated by $E_1$.  It is not difficult to see that there is an
epimorphism
\begin{equation}\label{eta} \eta: \H_{r,n}(\bu)\longrightarrow
\W_{r,n}/\langle E_1\rangle
\end{equation}  determined by:
$\eta(g_i)=T_i + \langle E_1\rangle$ and $\eta(y_j)= X_j + \langle
E_1\rangle$ for positive integers $i<n$ and $j\leq n$.

So, any $\W_{r,n}$-module, which is annihilated by $E_1$, is an
$\H_{r,n}(\bu)$-module. We will use  this fact frequently in the
next section.

\section{$\bu$-admissible conditions}
In this section,  unless otherwise stated, we always  assume that
$R=\mathbb Q(u_1, u_2, \cdots,
 u_r, q)$, where $q, u_1, u_2, \cdots, u_r$
 are algebraically independent over $\mathbb Q$. Let $\varrho^{\pm 1} \in R$, $\delta=q-q^{-1}$ and
 $\Omega=\set{\omega_a\mid a\in \mathbb Z }\subseteq R$ such that
$$\omega_0=1-\delta^{-1}(\varrho-\varrho^{-1}).$$

The main purpose of this section is to construct all possible
irreducible representations of $\W_{r, 2}$ over $R$. We find a set
of conditions on the parameters $\varrho$ and  $\{\omega_a\mid a\in
\mathbb Z\}$, called \textsf{ $\bu$-admissible conditions}, such
that the dimension of the corresponding $\W_{r, 2}$ is  $3r^2$.
These conditions, which are similar to those for the cyclotomic
Nazarov-Wenzl algebras in \cite{AMR},  are exactly what we need for
general $n$.\footnote{The results given in this section are the main
results of \cite{Xu}. Similar results can be found in \cite{WY}. We
remark the method we are using  here is almost the same as that used
in \cite{AMR}.}

\begin{Prop}
\label{E=0}Suppose that $M$ is an irreducible
$\W_{r,2}$--module such that $E_1M=0$. Then one of the following
results holds.
\begin{enumerate}
\item $M=Rm$ is one dimensional and the action of $\W_{r,2}$
        is determined by
$$T_1 m=\epsilon m,\quad E_1m=0, \quad X_1m=u_im,\quad\text{and}\quad
               X_2m=\epsilon^2 u_im,$$
where $\epsilon\in \{q, -q^{-1}\}$ and $1\le i\le r$. In particular,
up to isomorphism, there are exact $2r$ such representations.
\item $M$ is two dimensional and the action of $\W_{r,2}$ is given by
\begin{tabular}{lll}
$T_1\mapsto\frac{u_{j}}{u_{j} - u_{i}} \left(\begin{array}{ll}
\delta &q - u_{i}q^{-1}u_{j}^{-1} \\  q^{-1} -
{qu_{i}}{u_{j}^{-1}}& - \delta u_{i}u_{j}^{-1} \end{array}\right)$\\
$E_1\mapsto\left(\begin{array}{ll} 0&0\\0&0\end{array}\right)$\\
$X_{1}\mapsto\left(\begin{array}{ll}
u_{i}&0\\0&u_{j}\end{array}\right)$\\
$X_{2}\mapsto\left(\begin{array}{ll}
u_{j}&0\\0&u_{i}\end{array}\right)$\\
\end{tabular}

\noindent  where $1\le i, j\le  r$ with $i\neq j$. In particular, up
to isomorphism there are exact $\binom r2$ such representations.
\end{enumerate}
\end{Prop}

\begin{proof} It follows from (\ref{eta}) that  any irreducible $\W_{r, 2}$-module $M$
 has to be an irreducible $\H_{r, 2}$-module if $E_1 M=0$. By the results for $\H_{r, 2}$ in \cite{AK},
 $M$ has to be one of the modules given in (a) and
(b). By direct computation, both (a) and (b) do define the $\W_{r,
2}$-modules with trivial action of $E_1$ on them.
\end{proof}

\begin{Prop} \label{Ene0}
Suppose $\omega_0\ne0$.   There is  a unique irreducible $\W_{r,
2}$--module $M$ such that $E_1 M\ne0$. Moreover,  $\dim_R M\le r$.
If $d=\dim_R M$, then there exists a basis $\{m_1, m_2, \dots,m_d\}$
of~$M$ and scalars $\{v_1,\dots,v_d\}\subseteq\{u_1,\dots,u_r\}$,
with $v_i\ne v_j$ whenever~$i\ne j$, such that for $1\le i\le d$ the
following hold:
\begin{enumerate}
    \item $X_1 m_i=v_im_i$ and $X_2 m_i=v_i^{-1} m_i$,
    \item $E_1 m_i= \gamma_i (m_1+\dots+m_d)$
    \item $\displaystyle T_1 m_i= \frac{\delta (\gamma_i-1)}{v_i^2-1}m_i
        +\sum_{j\ne i}\frac{\delta\gamma_i}{v_iv_j-1}m_j$,
\end{enumerate}\noindent%
where   $\omega_{a} = \sum\limits_{j=1}^d v_{j}^{a}\gamma_{j}$,
$a\in \mathbb Z$, and

\begin{itemize} \item [(1)] $\gamma_{i} = (\gamma_d(v_i) + \delta^{-1}\varrho(v_{i}^{2}
- 1) \prod \limits_{j\neq i}v_{j})\prod\limits_{j\neq
i}\frac{v_{i}v_{j} - 1}{v_{i} - v_{j}}$, where $\gamma_d(z)=1$ for
$2\nmid d$ and $-z$, otherwise.
\item [(2)]$\varrho^{- 1}= \alpha \prod_{l=1}^d v_{l}$ where $\alpha\in \{1, -1\}$ if $2\nmid d$ and $\alpha\in \{q^{-1}, -q\}$, otherwise.
\item[(3)]
 $\omega_{0}=
\delta^{-1}\varrho(\prod\limits_{l=1}^d v_{l}^2 - 1) +
1-\frac{(-1)^d +1}{2}\prod_{i=1}^d v_i$.
\end{itemize}

Conversely,  if  $\gamma_{j}, \varrho^{-1}$,  $\omega_{a} =
\sum\limits_{j = 1}^d v_{j}^{a}\gamma_{j}$, $a\in \mathbb Z$, are
defined as above, then (a)-(c) define an irreducible
$\W_{r,2}$-module with $E_1 M\neq 0$.
\end{Prop}

\begin{proof} The result can be proved by arguments similar to those
for the  cyclotomic Nazarov-Wenzl algebras in  \cite{AMR}. We
include a proof here.

By direct computation, one can verify that the $ R$-module $N$
generated by $ \{X_{1}^{\alpha}X_{2}^{\beta},
X_{1}^{\alpha}X_{2}^{\beta}T_1, X_{1}^{\alpha}E_1X_{1}^{\beta}| 0
\leq \alpha,\beta \leq r - 1\}$ is a right $\W_{r, 2}$-module. Since
$1\in N$, $N=\W_{r, 2}$. In particular, $\W_{r, 2}$ is of finite
dimension. So is any irreducible $\W_{r, 2}$-module.

Suppose that $M$ is an irreducible $\W_{r,2}$--module such that
$E_1M\ne0$ and  $d=\dim_R M$. We first show that (a)--(c) hold.

Since $u_1,\dots,u_r$ are pairwise distinct, we can fix a basis
$\{m_1,\dots,m_d\}$ of $M$ consisting of eigenvectors for $X_1$.
Write $X_1m_i=v_i m_i$, for some $v_i\in \{u_1, \cdots, u_r\}$. We
remark that we will  prove that $v_i\neq v_j$ whenever $i\neq j$.

Since we are assuming that $\omega_0\neq 0$,  we define
$f=\frac1{\omega_0}E_1$. Then $f^2=f$, $E_1f=fE_1=\omega_0 f$, and
$f\neq 0$. Also, we have  $fM\ne0$ since $E_1M\ne0$.

Fix an element $0\ne m\in fM$. Then $E_1m=\omega_0 m$ and
$T_1m=\varrho m$ (since $T_1E_1=\varrho E_1$).  As $X_1X_2E_1=E_1$,
$X_1X_2 E_1m =\omega_0m$. Therefore,  $X_1X_2m=m$.

We consider the  $R$-module $M'$ generated by
$\{m,X_1m,\dots,X_1^{d-1}m\}$. Since
$E_1X_1^km=E_1X_1^kfm=\frac1{\omega_0}E_1X_1^kE_1m
         =\frac{\omega_k}{\omega_0}E_1m
         =\omega_k m, $ $M'$ is stable under the action of $E_1$.
By Lemma~\ref{relations2}(3) for $k=1$ and
Lemma~\ref{antiinvolution},
\begin{align*}
T_1X_{1}^{a}m &= \frac{1}{\omega_{0}}T_1X_{1}^a E_1m\\
&=\frac{1}{\omega_{0}}(\varrho X_{1}^{ - a} + \sum\limits_{i =
1}^a \delta (\omega_{a - i}X_{1}^{ - i} - X_{1}^{a - 2i}))E_1m\\
&=\varrho X_{1}^{ - a}m + \sum\limits_{i = 1}^a \delta (\omega_{a -
i}X_{1}^{ - i}m - X_{1}^{a - 2i}m)
\end{align*}
Note that $(X_1-v_1) (X_1-v_2)\cdots (X_1-v_d)=0$ on $M$. Each
$X_1^a$ for $a\in \mathbb Z$ can be written as a linear combination
of $X_1^b$ with $0\le b\le d-1$.  Therefore, $M'$ is closed under
the action of $T_1$. Note that $X_2=T_1X_1T_1$. So, $M'$ is closed
under the action of $X_2$. We have $M'=M$ since $M$ is irreducible.
In particular, by $\dim M=d$,  $\{m, X_1m, \cdots, X_1^{d-1} m\}$ is
a basis of $M$. Since $E_1 X_1^k m=\omega_k m$, $E_1M=Rm$.

Write $m=\sum_{i=1}^dr_im_i$, for some $r_i\in R$. Suppose that
$r_i=0$ for some $i$. Then
$$\prod_{\substack{1\le j\le d\\j\ne i}}(X_1-v_j)\cdot m=0.$$
This contradicts the linear independence of
$\{m,X_1m,\dots,X_1^{d-1}m\}$; hence, $r_i\ne0$ for $i=1,\dots,d$.
By rescaling the $m_i$ if necessary, we do  assume that
\begin{equation}\label{m}m=m_1+m_2+\dots+m_d\end{equation}in the following.

By the previous arguments, all of the eigenvalues
$\{v_1,\dots,v_d\}$ of $X_1$ must be distinct. It is easy to verify
that $X_1X_2$ is central in $\W_{r,2}$,  $X_1 (X_1X_2) m_i=(X_1X_2)
X_1m_i$.  So $X_1X_2$ acts on $m_i$ as a scalar since $u_i$'s are
algebraic independent. On the other hand, since $X_1X_2 m=m$, by
(\ref{m}),  $X_1X_2$ acts as $1$ on $M$. Therefore, $X_2m_i=X_1^{-1}
m_i=v_i^{-1} m_i$, for $i=1,\dots,d$, proving (a).

 Since $u_1,\dots,u_r$
are algebraically independent,   $v_i^2-1$ and $v_iv_j-1$, for $i\ne
j$, are invertible in $R$. So the formula in part (c) makes sense.

As $E_1M=Rm$, we can define elements $\gamma_i\in R$ by
$$
E_1m_i=\gamma_im=\gamma_i(m_1+\dots+m_d),\qquad\text{for }
i=1,\dots,d.
$$
Write $T_1 m_i=\sum_{j=1}^d c_{ji}m_j$. By Definition~\ref{Waff
relations}(f) for $i=1$, we have
$$T_1X_{2}m_{i} - X_{1}T_1m_i = \delta X_{2}m_{i} - \delta
E_1X_{2}m_{i}.$$ So,
 $$\sum\limits_{j = 1}^d
v_{i}^{ - 1}c_{ji} m_{j} - \sum\limits_{j = 1}^d c_{ji} v_{j}m_{j} =
\delta v_{i}^{ - 1}m_{i} - \delta v_{i}^{ - 1}\gamma_{i}m.$$
Comparing the coefficients of $m_j$ on both sides of the above
equation, we have \begin{equation}\label{cji} c_{ji}
=\frac{\delta(\gamma_{i} - \delta_{ij})}{v_{i}v_{j} - 1},
\end{equation} where $\delta_{ij}$ is the Kronecker function.
Therefore,
$$T_1 m_{i} = \frac{\delta(\gamma_{i} - 1)}{v_{i}^{2} - 1}m_{i} +
\sum\limits_{1\leq j\leq d \atop j\neq i}
\frac{\delta\gamma_{i}}{v_{i}v_{j} - 1}m_{j}.$$ This proves (c).

Since  $\varrho E_1 = T_1E_1$, we have, for  any positive integer $
i\le d $:
$$\begin{aligned} \varrho\gamma_{i}\sum_{j = 1}^d m_{j} &= \varrho E_1m_{i} = T_1E_1m_{i}
= T_1(\gamma_{i}\sum_{j = 1}^d m_{j})\\
& = \gamma_{i}\sum\limits_{j = 1}^d (\frac{\delta(\gamma_{j} -
1)}{v_{j}^{2} - 1}m_{j} + \sum\limits_{1\leq k\leq d \atop k\neq
j} \frac{\delta\gamma_{j}}{v_{j}v_{k} - 1}m_{k})\\
&= \gamma_{i}(\sum\limits_{j = 1}^d \frac{\delta(\gamma_{j} -
1)}{v_{j}^{2} - 1}m_{j} + \sum\limits_{j = 1}^d \sum\limits_{1\leq
k\leq d \atop k\neq j} \frac{\delta\gamma_{k}}{v_{j}v_{k} -
1}m_{j})\\
\end{aligned}$$

Since $E_1M\neq 0$, there is at least a non-zero $\gamma_{i}$.
Therefore, $$\varrho \sum\limits_{j = 1}^d m_{j} = \sum\limits_{j =
1}^d \frac{\delta(\gamma_{j} - 1)}{v_{j}^{2} - 1}m_{j} +
\sum\limits_{j = 1}^d \sum\limits_{1\leq k\leq d \atop k\neq j}
\frac{\delta\gamma_{k}}{v_{j}v_{k} - 1}m_{j}.$$ Comparing the
coefficients of  $m_{j}$ on both sides of the equality, we have
\begin{equation}\label {uniquesolution}
\sum_{k = 1}^d \frac{\gamma_{k}}{v_{j}v_{k} - 1} = \delta^{ -
1}\varrho + \frac{1}{v_{j}^{2} - 1}, j = 1,2,\cdots,d.\end{equation}

The system of linear equations (\ref{uniquesolution}) on $\gamma_k$
has a unique solution since $\det A_d\neq 0$ where $A_d$ is
 the coefficient matrix such that the $(i, j)$-th entry of
 $A_d$ is  $(v_{i}v_{j} - 1)^{-1}$. In fact, we have
\begin{equation}\label{det} \det A_d= \prod\limits_{1 \leq k< j\leq d}(v_{k} -
v_{j})^{2}{\prod\limits_{1 \leq k, j\leq d}(v_{k}v_{j} -
 1)^{-1}}.
 \end{equation}

 In order to verify (\ref{det}), we first observe that
$\prod_{1 \leq k, j\leq d}(v_{k}v_{j} - 1)\cdot \det A_d$ is a
symmetric polynomial in $v_{1},\dots,v_{d}$. So, it is divided by
$v_{k} - v_{j}, k\neq j $. Therefore, there is a $f_d
(v_{1},\cdots,v_{d})\in \mathbb Q[v_1, \dots, v_d]$ such that
\begin{equation}\label{35}\prod_{1 \leq k, j\leq d}(v_{k}v_{j}
- 1)\cdot \det A_d = f_{d}(v_{1},\cdots,v_{d})\prod_{1 \leq k< j\leq
d}(v_{k} - v_{j})^{2}.\end{equation}

 Comparing the coefficients of $v_1$ on both sides of (\ref{35}),
we know that the highest degree of $v_1$ in $f_d(v_1, v_2, \cdots,
v_d)$ is zero. Since $f_d (v_1, v_2, \cdots, v_d) $ is a symmetric
polynomial in $v_1, v_2, \cdots, v_d$, $f_d(v_1, v_2, \cdots,
v_d)=c_d\in \mathbb Q$.

In order to determine $c_d$, we set $v_d=0$ and get
$$\det A_d|_{v_{d} = 0} = -\det A_{d-1}
\prod_{i=1}^{d-1} v_{i}^{2} $$

Using (\ref{35}) to rewrite the above equation, we have
$c_d=c_{d-1}$. By induction assumption,  $c_d=c_j$, $1\le j\le d-1$.
An easy computation shows that $c_1=1$. Thus $c_d=1$, proving
(\ref{det}).

We have proved that the system of linear equations given in
(\ref{uniquesolution}) has a unique solution.  We define
 $$f(z) =  \frac{\gamma_d(z)}{(z^{2} - 1)(v_{j}z -
1)}\prod_{l= 1}^d \frac{v_{l}z - 1}{z - v_{l}} +
\frac{\delta^{-1}\varrho}{z(v_{j}z - 1)} \prod_{l = 1}^d
\frac{v_{l}(v_{l}z - 1)}{z - v_{l}},
$$
where $\gamma_d(z)$ is defined in (1) of Proposition~\ref{Ene0}.  By
residue theorem for complete non-singular curves for $f(z)$, we have

$$ \sum_{k = 1}^d \frac{1}{v_{j}v_{k} - 1}(\gamma_d(v_k) +
\delta^{-1}\varrho(v_{k}^{2} - 1) \prod_{l\neq k}v_{l})\prod_{l\neq
k}\frac{v_{k}v_{l} - 1}{v_{k} - v_{l}} = \delta^{ - 1}\varrho +
\frac{1}{v_{j}^{2} - 1}.$$  The above equalities shows that
$\gamma_k$'s given in (1)  in Proposition~\ref{Ene0} satisfy
(\ref{uniquesolution}). We remark that the  left (resp. right) side
of  the above equality can be interpreted as $\sum_{k = 1}^d Res_{z
= v_{k}}f(z)dz$ (resp. $-\sum\limits_{v\in I} Res_{z=v} f(z)dz$ and
$I=\{\infty, \pm 1, 0\}$). This completes the proof of (b).

Now, we prove the formula about $\omega_a$, $a\in \mathbb Z$.
Since $E_1m=\omega_0m$ and $m=\sum_{i=1}^d m_i$,
$\omega_0=\sum_{i=1}^d \gamma_i$. Similarly, we have that
$\omega_a=\sum_{j=1}^d v_j^a \gamma_j$. In order to  compute
$\omega_0$, we  define
$$g(z) =
\frac{\gamma_d(z) }{z^{2} - 1}\prod\limits_{l = 1}^d \frac{v_{l}z -
1}{z - v_{l}} + \frac{\delta^{-1}\varrho}{z} \prod\limits_{l = 1}^d
\frac{v_{l}(v_{l}z - 1)}{z - v_{l}}.$$ Then $\omega_{0} = \sum_{i =
1}^d \gamma_{i} = \sum_{i = 1}^d Res_{z = v_{i}}g(z)dz$. By residue
theorem for complete non-singular curves for  $g(z)$, $\omega_0= -
\sum_{v\in I}Res_{z = v}f(z)dz$ and $I=\{\infty, \pm1, 0\}$. By
direct computation,
$$
\omega_0= \delta^{-1}\varrho(\prod\limits_{l=1}^d v_{l}^2 - 1) + 1
-\frac{(-1)^d+1} 2\prod_{l = 1}^d v_{l}.
$$
By  solving the equation $(\omega_0 - 1)\delta=\varrho^{-1} -
\varrho$, we get $\rho$ as required.

We next show that $M$ is uniquely determined, up to isomorphism.
Suppose that $\W_{r,2}$ has another irreducible module of dimension
$d'$ upon which $E_1$ acts non--trivially. Then, by the previous
arguments,
$$ \omega_0= \delta^{-1}\varrho'(\prod\limits_{l=1}^{d'}
w_{l}^2 - 1) + 1 -\frac{(-1)^{d'}+1} 2\prod_{l = 1}^{d'} w_{l}.
$$ for some $\set{w_1,\dots,w_{d'}}\subseteq\{u_1,\dots,u_r
\}$. Since we are assuming that $u_1,\dots,u_r, q$ are algebraically
independent,  $d'=d$, $\varrho=\varrho'$ and $w_i=v_{(i)\sigma}$,
for some $\sigma\in\Sym_d$ and $1\le i\le d$.  By (a)--(c), $M\cong
M'$ as required.

Finally, it remains to verify that (a)--(c) do  define a
representation of $\W_{r,2}(\bu)$ whenever
$\omega_a=\sum_{i=1}^dv_i^a\gamma_i$, for $a\in \mathbb Z$ and
$\gamma_i$, $\varrho$ as above. In other word, we need verify the
defining relations for $\W_{r, 2}$. More explicitly, we need verify
(a), (b), (d)--(h) and (j)-(k) in Definition~\ref{Waff relations}.

In fact, it is easy to verify  (a),(e),(j),(k).  Note that (d) and
(g) can be verified easily by using the formula
$\omega_a=\sum_{j=1}^d v_j^d \gamma_j$. (f) follows from (\ref{cji})
and  (h) follows from
$$\sum\limits_{k = 1}^d \frac{\gamma_{k}}{v_{j}v_{k} - 1} = \delta^{
- 1}\varrho + \frac{1}{v_{j}^{2} - 1}, j = 1,2,\cdots,d.$$ Finally,
we verify (b) in Definition~\ref{Waff relations}.

We have already proved that  Definition~\ref{Waff relations}(f)
holds on  $M$. Thus, the following equalities hold in $\End_R(M)$:
$$\begin{cases} & T_1^{2}X_{2} - T_1X_{1}T_1 =  \delta
T_1X_{2} - \varrho\delta E_1X_{2}\\
 & X_{2}T_1^{2} - T_1X_{1}T_1 = \delta
X_{2}T_1 - \varrho\delta X_{2}E_1\\
\end{cases}
$$
Therefore, $ (T_1^{2} - \delta T_1 + \varrho\delta E_1)X_{2} =
X_{2}(T_1^{2} - \delta T_1 + \varrho\delta E_1)$.  Since
$v_{i}^{-1}\neq v_j^{-1}$ whenever $i\neq j$, and  $X_2 m_i=v_i^{-1}
m_i$, $T_1^{2} - \delta T_1 + \varrho\delta E_1$  acts on
$\{m_{1},\cdots,m_{d}\}$ diagonally. So, for each positive integer
$i\le d$, there is a $c_i\in R$ such that $$(T_1^{2} - \delta T_1 +
\varrho\delta E_1)m_{i} = c_{i}m_{i}.$$ So, {\small \begin{align*}
&(T_1^{2} - \delta T_1 + \varrho \delta E_1)m_{i}\\
= &  T_1( \frac{\delta(\gamma_{i} - 1)}{v_{i}^{2} - 1}m_{i} +
\sum\limits_{1\leq j\leq d\atop j\neq
i}\frac{\delta\gamma_{i}}{v_{i}v_{j} - 1}m_{j})\\ & - \delta(
\frac{\delta(\gamma_{i} - 1)}{v_{i}^{2} - 1}m_{i}  +
\sum\limits_{1\leq j\leq d\atop j\neq
i}\frac{\delta\gamma_{i}}{v_{i}v_{j} - 1}m_{j}) +
\varrho\delta\gamma_{i}\sum\limits_{j = 1}^d m_{j}\\
 = &  \frac{\delta(\gamma_{i} - 1)}{v_{i}^{2} - 1}
(\frac{\delta(\gamma_{i} - 1)}{v_{i}^{2} - 1}m_{i} +
\sum\limits_{1\leq j\leq d\atop j\neq
i}\frac{\delta\gamma_{i}}{v_{i}v_{j} - 1}m_{j}) +\sum\limits_{1\leq
j\leq d\atop j\neq i}\frac{\delta\gamma_{i}}{v_{i}v_{j} -
1}(\frac{\delta(\gamma_{j}
- 1)}{v_{j}^{2} - 1}m_{j}\\
& + \sum\limits_{1\leq k\leq d\atop k\neq
j}\frac{\delta\gamma_{j}}{v_{j}v_{k} - 1}m_{k}) - \delta(
\frac{\delta(\gamma_{i} - 1)}{v_{i}^{2} - 1}m_{i} +
\sum\limits_{1\leq j\leq d\atop j\neq
i}\frac{\delta\gamma_{i}}{v_{i}v_{j} - 1}m_{j}) +
\varrho\delta\gamma_{i}\sum\limits_{j = 1}^d m_{j}\\
= & c_{i}m_{i}.
\end{align*}}
Comparing the coefficient of $m_i$ on both sides of the above
equality, we have $$c_{i} = \delta^{2}\gamma_{i}\sum_{j = 1}^d
\frac{\gamma_{j}}{(v_{i}v_{j} - 1)^2} + \delta^{2}\frac{v_{i}^{2} -
(1 + v_{i}^2)\gamma_{i}}{(v_{i}^{2} - 1)^2} +
\varrho\delta\gamma_{i}.$$

 Suppose $$h(z) =
\frac{\gamma_d(z) }{(z^{2} - 1)(v_{i}z - 1)^2}\prod\limits_{l = 1}^d
\frac{v_{l}z - 1}{z - v_{l}} + \frac{\delta^{-1}\varrho}{z(v_{i}z -
1)^2} \prod\limits_{l = 1}^d \frac{v_{l}(v_{l}z - 1)}{z - v_{l}}.$$
We use the residue theorem for complete non-singular curves for
$h(z)$ and the  equalities for $\gamma_i$, $1\le i\le d$, in
Proposition~\ref{Ene0} to compute $c_i$. We  discuss the case
$2\nmid d$ and leave the other to the reader.

By direct computation, we have $$Res_{z=v}h(z)=\begin{cases}
-\frac 12 (v_i-1)^{-2}, & \text{if $v=1$,}\\
-\frac 12 (v_i+1)^{-2}, & \text{if $v=-1$,}\\ \delta^{-1}\varrho, &
\text{if
$v=0$,}\\
\end{cases}
$$
and
$$\begin{cases} & Res_{z=v_i^{-1}}h(z)  =\frac{v_i^2}{(1-v_i^2)^2}\prod_{l\neq i}\frac{v_i-v_l}{v_iv_l-1}
                      +\delta^{-1}\varrho \frac{v_i^2}{1-v_i^2}\prod_{l\neq i}v_l\prod_{l\neq i}\frac{v_i-v_l}{v_iv_l-1}, \\
& Res_{z=v_k}h(z) =\frac{1}{(1-v_iv_k)^2}\prod_{l\neq
k}\frac{v_kv_l-1}{v_k-v_l}
                      +\delta^{-1}\varrho \frac{v_k^2-1}{(1-v_iv_k)^2}\prod_{l\neq k}v_l\prod_{l\neq k}\frac{v_kv_l-1}{v_k-v_l}.\\
                      \end{cases} $$
Using (1)--(2) of Proposition~\ref{Ene0} to rewrite the above
equalities yields
$$\begin{cases}  Res_{z=v_i^{-1}}h(z)  &=\frac{v_i^2}{\gamma_i(1-v_i^2)^2}-\frac1{\gamma_i\delta^{2}} \\
Res_{z=v_k}h(z) &=\frac{\gamma_k}{(1-v_iv_k)^2}.\\
                      \end{cases} $$

By the residue theorem for complete non-singular curves for $h(z)$,
we have
$$Res_{z\in I}h(z)+\sum_{k=1}^d Res_{z=v_k}h(z)=0$$
where $I=\{\pm 1, 0, v_i^{-1}\}$.   Rewriting the above equality
yields
$$\delta^{-1}\varrho= \frac{v_i^2+1}{(v_i^2-1)^2}-
\frac{v_i^2}{\gamma_i(1-v_i^2)^2}(1-\frac{\delta^{-2}(v_i^2-1)^2
}{v_i^2})-\sum_{k=1}^d \frac{\gamma_k}{(1-v_iv_k)^2}.$$ So,
$$c_i=\delta\varrho\gamma_i-\frac{v_i^2+1}{(v_i^2-1)^2}\delta^2\gamma_i+
\frac{v_i^2}{(1-v_i^2)^2}\delta^2+\delta^2\gamma_i\sum_{k=1}^d
\frac{\gamma_k}{(1-v_iv_k)^2}=1.$$ This shows that
Definition~\ref{Waff relations}(b) holds on $M$.\end{proof}

\begin{Theorem}\label{free2} Suppose   $R=\mathbb Q(u_1, u_2, \cdots, u_r,
q)$ where $u_1, \dots, u_r, q$ are algebraically independent over
$\mathbb Q$. If $\varrho$ and $\omega_a, a\in \mathbb Z$,  are given
as in Proposition~\ref{Ene0} for $d=r$,  then $\W_{r, 2}$ is
semisimple over $R$. Moreover, $S= \{X_{1}^{\alpha}X_{2}^{\beta},
X_{1}^{\alpha}X_{2}^{\beta}T_1, X_{1}^{\alpha}E_1X_{1}^{\beta}| 0
\leq \alpha,\beta \leq r - 1\}$ is an $R$--basis  of $\W_{r, 2}$.
\end{Theorem}
\begin{proof} By direct computation, one can verify that  the $
R$-submodule $M$ of $\W_{r, 2}$  generated by $S$  is a right
$\W_{r, 2}$-module. Since $1\in S$,  $M=\W_{r, 2}$.  In particular,
 $\dim \W_{r, 2}\le 3 r^2$.

Under the assumption, we have, by Proposition~\ref{Ene0}, that there
is a unique  irreducible $\W_{r, 2}$-module $M$ with $E_1 M\neq 0$.
By Wedderburn-Artin Theorem for semisimple finite dimensional
algebras together with Proposition~\ref{E=0}, we have $$3r^2 = 2r +
4 \frac{r(r-1)}{2} + r^2= dim\W_{r,2}/Rad\W_{r,2}\leq dim {\W_{r,2}}
\leq 3r^2,
$$
where $\Rad \W_{r, 2}$ is the Jacobson radical of $\W_{r, 2}$. Thus
all inequalities given  above are equalities. In particular, $\W_{r,
2}$ is semisimple and  $S$ is an $R$--basis.
\end{proof}

\begin{Defn}
Suppose   $\mathbf{x} = (x_{1},x_{2},\cdots,x_{r})$ are variables.
For any non-negative integer $a$, Define $Q_{a}(\mathbf{x}),
Q'_{a}(\mathbf{x})$  such that
\begin{equation}\label{Q}\begin{cases} & \prod_{i = 1}^r \frac{y -
x_{i}}{x_{i}y - 1} = \sum_{a=0}^\infty Q_{a}(\mathbf{x})y^{a}, \\
& \prod_{i = 1}^r \frac{x_{i}y - 1}{y - x_{i}} = \sum_{a=
0}^\infty Q'_{a}(\mathbf{x})y^{a}.\\
\end{cases}
\end{equation}
\end{Defn}

We set $Q_a(\mathbf {x})=Q_a'(\mathbf {x})=0$ if $a<0$. For each
non-negative integer $a$, it is not difficult to verify that
$Q_{a}(\mathbf{x})$ (resp. $Q'_{a}(\mathbf{x})$)   is a symmetric
polynomial in variables $x_1, x_2, \cdots, x_r$ (resp. $x_1^{-1},
\dots, x_r^{-1}$). Further,
$$Q'_{a}(\mathbf{x})=Q_a(x_1^{-1}, \dots, x_r^{-1}).$$

\begin{Lemma} Suppose $R$ is an integral domain which contains the
identity $1_R$, and the units $u_{1}, u_{2}  ,\cdots,u_{r}, q,
(q-q^{-1})$ such that $u_iu_j^{\pm 1}\neq 1$. Let $F$ be the
quotient field of $R$. For $a\in \mathbb Z$, define
$$\omega_{a} =
\sum_{i = 1}^r (\gamma_d(u_i) + \delta^{-1}\varrho(u_{i}^{2} - 1)
\prod_{j\neq i}u_{j})u_{i}^{a}\prod_{j\neq i}\frac{u_{i}u_{j} -
1}{u_{i} - u_{j}}$$ where $\varrho$ is defined in (2) of
Proposition~\ref{Ene0} for $d=r$.
 Suppose that  $a \geq 0$.
 Then
\begin{equation}\label{oa}\omega_{a} =
\begin{cases} A +
\sum\limits_{k = 0}^{a - 1} \frac{1 + ( - 1)^{k}}{2}Q_{a - 1 -
k}(\mathbf u) - \delta^{-1}\varrho
\delta_{a0}, &\text{if $2\nmid r$},\\
A - \sum\limits_{k = 0}^a \frac{1 + ( - 1)^{k}}{2}Q_{a - k}(\mathbf
u) - \delta^{-1}\varrho \delta_{a0},
&\text{if $2\mid r$}, \\
\end{cases}
\end{equation}
and for $a>0$,
\begin{equation}\label{ob}\omega_{-a} =
\begin{cases} B +
\sum_{k = 0}^{a - 1} \frac{1 + ( - 1)^{k}}{2}Q'_{a - 1 -
k}(\mathbf u), &\text{if $2\nmid r$},\\
B - \sum_{k = 0}^{a-2} \frac{1 + ( - 1)^{k}}{2}Q'_{a - 2 -
k}(\mathbf u),
&\text{if $2\mid r$}, \\
\end{cases}
\end{equation}
where $$\begin{cases}  A &= \frac{1 + ( - 1)^{a}}{2} +
\delta^{-1}\varrho Q_{a}(\mathbf u)\prod_{i=1}^d  u_{i},\\
B& =\frac{1 + ( - 1)^{a}}{2} - \delta^{-1}\varrho  Q'_{a}(\mathbf u)
\prod_{i=1}^d u_{i}.\\
\end{cases}$$ In particular, $\omega_a\in R$ for all $a\in \mathbb Z$.
\end{Lemma}

\begin{proof} The result for $a=0$ follows from  Proposition~\ref{Ene0}.
Suppose $a$ is a positive integer.  We define
$$f(z) =
\frac{\gamma_d(z)   z^{a}}{z^{2} - 1}\prod\limits_{l = 1}^r
\frac{u_{l}z - 1}{z - u_{l}} + \delta^{-1}\varrho z^{a - 1}
\prod\limits_{l = 1}^r \frac{u_{l}(u_{l}z - 1)}{z - u_{l}}.
$$
Then  $\omega_{a} = \sum_{i = 1}^r Res_{z = u_{i}}f(z)dz$. By
residue theorem for complete non-singular curves for $f(z)$,
\begin{equation}\label{omm}
\omega_a= - \sum_{v\in \{\infty, \pm 1\}}Res_{z =
v}f(z)dz.\end{equation}

When $r$ is odd, $Res_{z = 1}f(z)dz = - \frac{1}{2}$ and $ Res_{z =
- 1}f(z)dz = - \frac{( - 1)^{a}}{2}$.  On the other hand, $-
z^{-2}f({z^{-1}}) =z^{-a} \varphi_{1}(z)+{z^{-a - 1}} \varphi_2(z)$,
where $$\begin{cases} & \varphi_{1}(z) = (z^{2} -
1)^{-1}\prod_{l = 1}^r \frac{z - u_{l}}{u_{l}z - 1}\\
&  \varphi_{2}(z) = - \delta^{-1}\varrho
\prod_{l = 1}^r \frac{(z - u_{l})u_l}{u_{l}z - 1}.\\
\end{cases} $$  Also, we have  $(\varphi_{2}(z))^{(a)}|_{z = 0}=- a!
\delta^{-1}\varrho \prod_{l = 1}^r u_{l}Q_{a}(\mathbf u)$ and
\begin{align*}
(\varphi_{1}(z))^{(a - 1)}|_{z = 0} &=\sum_{k = 0}^{a - 1}
\binom{a-1}{k}(\frac{1}{z^{2} - 1})^{(k)}\Big(\prod_{l = 1}^r
\frac{z - u_{l}}{u_{l}z - 1}\Big)^{(a - 1 - k)}|_{z = 0}\\
&=\sum_{k = 0}^{a - 1}\binom{a-1}{k}  (- 1)k!\frac{1 + (-
1)^{k}}{2}(a - 1 -
k)!Q_{a - 1 - k}(\textbf{u})\\
&= -(a-1)! \sum_{k = 0}^{a - 1} \frac{1 + ( - 1)^{k}}{2}Q_{a - 1 -
k}(\mathbf u).
\end{align*}
Note that $Res_{z = \infty}f(z)dz = Res_{z = 0}f_{1}(z)dz +Res_{z =
0}f_{2}(z)dz$ where $f_1(z)=z^{-a} \varphi_1(z)$ and
$f_2(z)=z^{-a-1} \varphi_2(z)$. So,
$$Res_{z = \infty}f(z)dz = - \sum\limits_{k = 0}^{a - 1} \frac{1 + (
- 1)^{k}}{2}Q_{a - 1 - k}(\mathbf u) - \delta^{-1}\varrho \prod_{l =
1}^r u_{l}Q_{a}(\mathbf u),
$$
 and (\ref{oa}) follows immediately from (\ref{omm}).

When $r$ is even, $\gamma_d(z)=-z$. When we compute $Res_{z=\infty}
f(z)dz$, we need compute $(\varphi_{1}(z))^{(a)}|_{z = 0}$.
Therefore, we need use $a$ instead of $a-1$ in (\ref{oa}) for odd
$r$. This implies the result for (\ref{oa}) in the case $2\mid r$.

One can verify (\ref{ob})
 by similar arguments as
above. Since $Q_{a}(\mathbf{u})$ (resp.   $Q'_{a}(\mathbf{u})$) are
polynomials in variables $u_1, u_2, \ldots, u_r$ (resp. $u_1^{-1},
\dots, u_r^{-1}$),
 $\omega_{a}\in R$  for all $a\in \mathbb Z $.
 \end{proof}

\begin{Defn}\label{add1} Suppose  $R$ is a commutative ring which contains identity
$1_R$, and the units $q, (q-q^{-1})$,  $u_i, 1\le i\le r$. Let
$\Omega = \{\omega_{a}\mid a\in \mathbb Z\}\subseteq R$. Then
$\Omega\cup\{\varrho\}$ is called  \textsf{$\mathbf{u}$-admissible}
if $\omega_{a}$ satisfy (\ref{oa}), (\ref{ob}) for $a\in \mathbb Z$
and $\varrho$ is given in Proposition~\ref{Ene0} for $d=r$.
\end{Defn}

By Theorem~\ref{free2}, $\W_{r, 2}$ is a free over $R$ with
dimension $3r^2$ if $R=\mathbb Q(u_1, \u_2, \cdots, u_r, q)$, and
$\Omega\cup \{\varrho\}$ is $\bfu$-admissible. We will show that
$\W_{r, n}$ is free over a commutative ring with rank $r^n (2n-1)
!!$ if $\Omega\cup\{\varrho\}$ is $\bfu$-admissible and $2\nmid r$.

 Motivated by Nazarov's work on Brauer algebras in \cite{Nazarov:brauer}, we
 define two generating functions $$\begin{cases} & \widetilde {w}_{1, +}(y) =
\sum_{a= 0}^{\infty} \omega_{a}y^{ - a}\\ & \widetilde{w}_{1,
-}(y) = \sum\limits_{a=1}^{\infty} \omega_{-a}y^{ - a}.\\
\end{cases}$$

\begin{Lemma} \label{uadmissequi}Suppose $y$ is an indeterminant.
Then $\Omega\cup\{\varrho\}$ is $\mathbf{u}$-admissible if and only
if $\varrho$ is given in Proposition~\ref{Ene0} for $d=r$ and
$$\begin{cases} \widetilde{w}_{1, +}(y) &=
\frac{y^{2}}{y^{2} - 1} - \delta^{-1}\varrho +
(\delta^{-1}\varrho\prod\limits_{l = 1}^r u_l +
\frac{y\gamma_d(y)}{y^{2} - 1})\prod_{l = 1}^r u_l\prod_{l = 1}^r
\frac{y - u_{l}^{-1}}{y - u_{l}},\\
\widetilde{w}_{1, -}(y) &= \frac{1}{y^{2} - 1} + \delta^{-1}\varrho
- \frac{1}{\prod_{l = 1}^r u_l}(\delta^{-1}\varrho\prod\limits_{l =
1}^r u_l - \frac{y}{\gamma_d(y)(y^{2} - 1)})\prod_{l = 1}^r \frac{y
- u_{l}}{y - u_{l}^{-1}}.\\
\end{cases}
$$
\end{Lemma}

\begin{proof}
When $r$ is odd, we have $\sum_{a=0}^{\infty} \frac{1 + (-
1)^{a}}{2}y^{- a} = \frac{y^{2}}{y^{2} - 1}$. By (\ref{Q}),
$$\varrho\prod_{l = 1}^r u_l\sum_{a= 0}^{\infty}
\delta^{-1}Q_{a}(\mathbf u)y^{- a} = \delta^{-1}\varrho\prod_{l =
1}^r u_l^2\prod_{l = 1}^r \frac{y - u_{l}^{-1}}{y - u_{l}}, $$ and
$$\begin{aligned} \sum_{a=0}^{\infty} \sum_{k = 0}^{a - 1} \frac{1 + ( -
1)^{k}}{2}Q_{a - 1 - k}(\mathbf u)y^{- a}& = (y^{-1}+y^{-3}+\cdots )
\sum_{a=0}^{\infty} Q_a(\bfu) y^{-a}\\
&= \frac{y}{y^{2} - 1}\prod_{l = 1}^r u_l\prod_{l = 1}^r \frac{y -
u_{l}^{-1}}{y - u_{l}}
\end{aligned}
$$
Taking the sum of the above equalities yields  the formula for
$\widetilde{w}_{1, +1}(y)$ in the case $r$ is odd. Similarly, one
can verify the formulae in other cases.\end{proof}

\begin{Cor} $\Omega\cup\{ \rho\}$ is $\bu$-admissible if and only if  \begin{enumerate}
\item $\varrho$ is defined by Proposition~\ref{Ene0} for $d=r$,
\item    $\omega_a$ for $0\le a\le r-1$ is determined by
$\tilde w_{1, +}(y)$, \item  $\Omega$ is admissible.\end{enumerate}
\end{Cor}
\begin{proof} ``$\Longrightarrow$"
Let $c=\delta^{-1}\varrho u_1\cdots u_r$. By
Lemma~\ref{uadmissequi}, we have {\small $$ (\widetilde{w}_{1,
+}(y)-\frac{y^{2}}{y^{2} - 1} +
\delta^{-1}\varrho)(\widetilde{w}_{1, -}(y)-\frac{1}{y^{2} - 1} -
\delta^{-1}\varrho)
=(\frac{y\gamma_d(y)}{1-y^{2}}-c)(\frac{y\gamma_d(y)^{-1}}{(1-y^{2})}+c).
$$}

For any positive integers $r$ and all $\varrho$ in
Proposition~\ref{Ene0} for $d=r$, we have
$$(\frac{y\gamma_d(y)}{1-y^{2}}-c)(\frac{y\gamma_d(y)^{-1}}{(1-y^{2})}+c)=\frac{y^2}{(1-y^2)^2}-\delta^{-2}.
$$
So, \begin{equation}\label{abcd}(\widetilde{w}_{1,
+}(y)-\frac{y^{2}}{y^{2} - 1} +
\delta^{-1}\varrho)(\widetilde{w}_{1, -}(y)-\frac{1}{y^{2} - 1} -
\delta^{-1}\varrho)=\frac{y^2}{(1-y^2)^2}-\delta^{-2}.\end{equation}

Multiplying $(1-y^2)^2$ on both sides of (\ref{abcd}) yields an
equality given by  Laurent polynomials in indeterminate $y$.
Comparing the coefficients of $y^i, i\le 4$ on both sides of such
equality  yields the second equality in (\ref{ad1}). If $\omega_a=
\sum_{j=1}^r u_j^a \gamma_j$ where $\gamma_j$'s are
 given in Proposition~\ref{Ene0} for $d=r$, then
$$\begin{aligned} \sum_{a=0}^{r} (-1)^{r-a} \sigma_{r-a} (\bu)
\omega_{a+b} &= \sum_{a=0}^{r}
(-1)^{r-a} \sigma_{r-a} (\bu) \sum_{j=1}^r u_j^{a+b} \gamma_j\\
&= \sum_{j=1}^r (\sum_{a=0}^r (-1)^{r-a} \sigma_{r-a} (\bu) u_j^a)
u_j^b\gamma_j\\
&=0.\\
\end{aligned}$$
So, $\Omega$ is admissible.

``$\Longleftarrow$" Conversely, if $\Omega$ is admissible, then
$\omega_a$ for all $a\in \mathbb Z$ are determined by $\omega_0,
\omega_1, \dots, \omega_{r-1}$, uniquely. This implies the
result.\end{proof}

\section{The seminormal representations of $\W_{r,n}(\bu)$}

% for the next two sections we change these so that we do not
% have too many s's.
%\def\s{\mathfrak t}
\def\ts{\tilde\s}

In this section, unless otherwise stated,  we always keep the
following assumptions:

\begin{Assumption} (1) $R$ is a field which contains non zero elements $q, $ and
$u_i, 1\le i\le r$ with $o(q^2)>2n$. (2) The parameters $(u_1,u_2,
\dots, u_r)$ are generic in the sense of Definition~\ref{W-generic}.
(3) The root conditions~\ref{be} hold. (4) $\Omega\cup\{\varrho\}$
is $\bu$-admissible.\end{Assumption}

 The main purpose of this section is to construct the seminormal
representations for $\W_{r, n}$ over $R$. We remark that the method
we are using is the same as that for cyclotomic Nazarov-Wenzl
algebras in \cite{AMR}. We start by  recalling some combinatorics.

A  \textsf{partition} of $m$ is a sequence of non--negative integers
$\lambda=(\lambda_1,\lambda_2,\dots)$ such that $\lambda_i\ge
\lambda_{i+1}$ for all positive integers $i$ and
$|\lambda|:=\lambda_1+\lambda_2+\cdots=m$. Similarly, an
\textsf{$r$-partition} of $m$ is an ordered $r$--tuple
$\lambda=(\lambda^{(1)},\dots,\lambda^{(r)})$ of partitions
$\lambda^{(s)}$ such that
$|\lambda|:=|\lambda^{(1)}|+\dots+|\lambda^{(r)}|=m$. Let
$\Lambda_r^+(m)$ be the set of all $r$-partitions of $m$.

 If $\lambda$ and $\mu$ are two $r$-partitions we say that $\mu$ is
obtained from $\lambda$ by \textsf{adding} a box if there exists a
pair $(i,s)$ such that $\mu^{(s)}_i=\lambda^{(s)}_i+1$ and
$\mu^{(t)}_j=\lambda^{(t)}_j$ for $(j,t)\ne(i,s)$. In this situation
we will also say that $\lambda$ is obtained from $\mu$ by
\textsf{removing} a box and we write $\lambda\subset\mu$ and
$\mu\setminus\lambda=(s, i,\lambda^{(s)}_i+1)$. We will also say
that the triple $(s, i,\lambda^{(s)}_i+1)$ is an \textsf{addable}
(resp. \textsf{removable}) node of $\lambda$ (resp. $\mu$) which is
in the $i$-th row, $\lambda_i^{(s)}+1$-column of $s$-th component of
$\lambda$ (resp. $\mu$).

 Fix an integer $m$ with $0\le m\le\floor{n2}$. Let
 $\lambda\in \Lambda_{r}^+(n-2m)$. It has been defined in \cite{AMR} that
 an \textsf{$n$--updown
 $\lambda$--tableau}, or more simply an updown
$\lambda$--tableau, is a sequence $\t=(\t_0, \t_1,\t_2,\dots,\t_n)$
of $r$-partitions where $\t_n=\lambda$ and the $r$-partition $\t_i$
is obtained from $\t_{i-1}$ by either \textit{adding} or
\textit{removing} a box, for $i=1,\dots,n$.  When $i=0$, we always
assume that $\t_i=\emptyset$. Let $\UPD_n(\lambda)$ be the set of
$n$--updown $\lambda$--tableaux.

There is an equivalence relation $\simk$ on $\UPD_n(\lambda)$, which
has  been defined in \cite{AMR}.  Suppose $\s, \t\in
\UPD_n(\lambda)$. Then  $\t\simk \s$ if $\t_j=\s_j$
    whenever $1\le j\le n$ and $j\neq k$, for $\s,\t\in\UPD_n(\lambda)$.
The following result has been proved in  \cite{AMR}.

\begin{Lemma}\label{simk} Suppose $s\in \UPD_n(\lambda)$ with
$\s_{k-1}=\s_{k+1}$. Then there is a bijection between the set of
all addable and removable nodes of $\s_{k-1}$ and the set of $\t\in
\UPD_n(\lambda)$ with $\t\simk \s$.
\end{Lemma}

Suppose that $\t\in \UPD_n(\lambda)$ where $\lambda\in
\Lambda_r^+(n-2f)$ and $0\le f\le \lfloor \frac n2\rfloor $. For
each positive integer $k\le n$,  either $\t_k\subset\t_{k-1}$ or
$\t_{k-1}\subset\t_k$. We define
\begin{equation} c_\t(k)=\begin{cases} \label{content}
     u_s q^{2(j-i)} ,  &\text{if }\t_k{\setminus}\t_{k-1}=(i,j,s),\\
     u_s^{-1} q^{-2(j-i)}, &\text{if }\t_{k-1}{\setminus}\t_{k}=(i,j,s).
\end{cases}\end{equation}
We call $c_\t(k)$ the \textsf{content} of $k$ in $\t$. Let
$\alpha=(i,j,s)$.   We also define
\begin{equation}\label{content1} c_\lambda (\alpha)=\begin{cases} u_s
q^{2(j-i)}, & \text{if $\alpha$ is an addable node of $\lambda$, }\\
u_s^{-1} q^{-2(j-i)}, & \text{if $\alpha$ is a removable node of
$\lambda$.}\\
\end{cases}
\end{equation}
 We write $c(\alpha)$ instead of $c_\lambda(\alpha)$ if
there is no confusion.

 The following condition, which is a counterpart of
the generic condition for cyclotomic Nazarov-Wenzl algebras in
\cite{AMR}, guarantees the existence of the seminormal
representations for $\W_{r, n}$.

\begin{Defn}\label{W-generic}
    The parameters $\bu=(u_1,\dots,u_r)$ are \textbf{generic} for
    $\W_{r,n}$ if whenever there exists $d\in\Z$ such that either
    $u_i u_j^{\pm 1}=q^{2d}1_R$ and $i\ne j$,
    or $u_i=\pm q^{d}\cdot1_R$ then $|d|\ge 2n$.
\end{Defn}

\begin{Lemma}\label{generic u}
    Suppose that the parameters $\bu$ are generic for $\W_{r,n}$
    and that $o(q^2)>2n$. Let $\s,\t\in\UPD_n(\lambda)$ where $\lambda\in \Lambda_r^+(n-2f)$ and $0\le f\le\floor{n2}$. Then
\begin{enumerate}
\item $\s=\t$ if and only if $c_\s(k)=c_\t(k)$, for all positive integers $k\le n$.
\item $c_\s(k)\neq c_\s(k+1)$, for all positive integers  $ k<n$.
\item if $\s_{k-1}=\s_{k+1}$ then $c_\s(k)\neq  c_\t(k)^{\pm 1} $, whenever
    $\t\simk\s$ and $\t\ne\s$,
\item $c_\t(k)\neq \pm q^{\pm 1}$ for all positive integers $k\le n$.
\end{enumerate}
\end{Lemma}

\begin{proof} (a)-(c) can be proved by the arguments similar to
those in \cite{AMR}.  The key point is that the  assumptions imply
that the contents of the addable and removable nodes in $\lambda$
are distinct so an up-down $\lambda$--tableau $\s$ is uniquely
determined by the sequence of contents $c_\s(k)$, for $k=1,\dots,n$.
(d)  can be verified by direct computation.\end{proof}

Unless otherwise stated, we fix a $\lambda\in \Lambda_r^+(n-2f)$.
Motivated by Ariki, Mathas and Rui's work on cyclotomic
Nazarov-Wenzl algebras in \cite{AMR}, we introduce the following
rational functions in an indeterminate $y$. Such  functions will
play a key role in the construction of seminormal representations of
$\W_{r,n}$.

\begin{Defn}\label{rational-W}
Suppose that $\s\in \UPD_n(\lambda)$.  For $1\le k\le n$, define
rational functions $W_k(y, \s)$ by
$$ W_k(y,\s)= \frac{y^2}{y^2-1}-\delta^{-1}
\varrho+(\delta^{-1}\varrho \prod_{i=1}^r
u_i+\frac{y\gamma_r(y)}{y^2-1}) \prod_{i=1}^r u_i \prod_{\alpha}
\frac{y-c^{-1}(\alpha)}{y-c(\alpha)}, $$ where $\alpha$ runs over
the addable and removable nodes of the $r$-partition $\s_{k-1}$.
\end{Defn}

\begin{Lemma} Suppose $\lambda$ is an $r$-partition. Then
\begin{equation}\label{contidentity}\prod_{\alpha} c(\alpha)=\prod_{i=1}^r
u_i\end{equation} where $\alpha$ runs over all addable nodes and
removable nodes of  $\lambda$.\end{Lemma}

\begin{proof}  It is known that  the number of addable nodes of
a partition, say $\mu$,  is equal to the number of the removable
nodes of $\mu$  plus $1$. We arrange the removable nodes (resp.
addable nodes) of $\mu$ from top to bottom.
 Therefore,  we assume that  $(p, r_i,
s_i), 1\le i\le k$ (resp. $(p, a_i, b_i), 1\le i\le k+1$) are all
removable (resp. addable) nodes of $\lambda^{(p)}$, the $p$-th
component of $\lambda$. We have $a_1=1$, $b_{k+1}=1$,
$a_i=r_{i-1}+1$ and $s_j=b_j-1$ for $2\le i\le k+1$ and  $1\le j\le
k$. By (\ref{content1}),
$$ \prod_{\alpha} c(\alpha)=u_p q^{2(\sum_{i=1}^{k+1}
(b_i-a_i)+\sum_{i=1}^k (r_i-s_i))}=u_p.$$ Multiplying the previous
equality for all positive integers $p\le  r$ yields
(\ref{contidentity}).
\end{proof}

\begin{Lemma}\label{w-residue}
    Suppose that $\bu$ is generic and $o(q^2)>2n$. Let
    $\s\in \UPD_n(\lambda)$ and $1\le k\le n$. Then
    $$ \frac{W_k(y, \s)}{y}=\sum_\alpha
         \Big(\Res_{y=c(\alpha)}\frac{W_k(y,\s)}{y}\Big)\cdot\frac 1{y-c(\alpha)},$$
    where $\alpha$ runs over the addable and removable nodes of
    $\s_{k-1}$.
\end{Lemma}

\begin{proof}
Since $\bu$ is generic and $o(q^2)>2n$,  $c(\alpha)$ are pairwise
distinct for different addable and removable nodes  $\alpha$ of
$\lambda$. Further, we have $c(\alpha)\not\in \{0, \pm 1\}$.
Therefore, we can write
$${W_k(y,\s)\over y}
    =a+\frac by +\frac c{y-1}+\frac d {y+1} +\sum_{ \alpha} \Big(\Res_{y=c(\alpha)}\frac{W_k(y,\s)}y\Big)
             \cdot {1\over y-c(\alpha)},$$
for some $a,b, c, d\in R$, where $\alpha$ runs over the addable and
removable nodes of $\s_{k-1}$. In order to prove the result, we need
verify $a=b=c=d=0$. In fact, $a=\lim_{y\rightarrow \infty}
{W_k(y,\s)\over y}=0$. We have   $b=\Res_{y=0}\tfrac{W_k(y,
\s)}{y}=-\delta^{-1} \varrho +\delta^{-1} \varrho\prod_{i=1}^r u_i^2
\prod_{\alpha} c(\alpha)^{-2}\overset
{(\ref{contidentity})}=-\delta^{-1}\varrho+\delta^{-1}\varrho =0$.
One can verify $c=d=0$ similarly.
\end{proof}

The following definition is the same as those for cyclotomic
Nazarov-Wenzl algebras if we use our rational functions $W_k(y, \s)
$ instead of those for  cyclotomic Nazarov-Wenzl algebras in
\cite{AMR}.

\begin{Defn}\label{e}
    Let $\lambda\in \Lambda_r^+(n-2f)$ for some non-negative integer $f\le \lfloor\frac n2\rfloor$.
     Assume that  $k$ is a positive  integer with
    $ k\le n$. If
    $\s, \t\in \UPD_n(\lambda)$
    with  $\s_{k-1}=\s_{k+1}$, then   we define the scalars
    $E_{\s\t}(k)\in R$ by
$$\label{ew}
E_{\s\t}(k)=\begin{cases}\displaystyle
    \Res_{y=c_\s(k)} \frac{W_k(y,\s)}{y},&\text{if }\s=\t,\\[10pt]
    \sqrt{E_{\s\s}(k)}\sqrt{E_{\t\t}(k)},&
              \text{if }\s\ne\t \text{ and }\t\simk \s,\\[10pt]
    0,&\text{otherwise}.
\end{cases}
$$
We remark that we have to  fix the choice of square roots
$\sqrt{E_{\s\s}(k)}$, for $\s\in\UPD_n(\lambda)$ and $1\le k\le n$,
which  we will illustrate late.
\end{Defn}

%Note that when $c_\s(k)\ne0$ then $\displaystyle
%E_{\s\s}(k)=\Res_{y=c_\s(k)}\frac{W_k(y,\s)+y-\frac12}{y}.$

In \cite{AMR}, there is no definition for  $E_{\s\s}(k)$ under the
assumption  $\s_{k-1}\ne\s_{k+1}$. In the current paper, we do not
need such a definition either.

If   $r$ is odd, then  $\rho^{-1}\in \{ u_1\cdots u_r, -u_1\cdots
u_r\}$. It follows from Definition~\ref{rational-W} that

\begin{equation}\label{ekodd}
    E_{\s\s}(k)=\begin{cases} \frac{1} {\varrho c_\s(k)}
 \big(\frac{c_\s(k)-c_\s(k)^{-1}}{\delta} +1\big)
                 \prod_{\alpha}\frac{c_\s(k)-c(\alpha)^{-1}}{c_\s(k)-c(\alpha)},
             &\text{if $\varrho^{-1}=u_1\cdots u_r$},\\
\frac{1} {\varrho c_\s(k)}
 \big(\frac{c_\s(k)-c_\s(k)^{-1}}{\delta} -1\big)
                 \prod_{\alpha}\frac{c_\s(k)-c(\alpha)^{-1}}{c_\s(k)-c(\alpha)},
             &\text{if $\varrho^{-1}=- u_1\cdots u_r$,}\\
\end{cases}
\end{equation} where $\alpha$ runs over all addable and removable
nodes of $\s_{k-1}$ with $\alpha\neq \s_k\setminus \s_{k-1}$.

If $r$ is even, then $\varrho^{-1}\in \{ q^{-1}u_1\cdots u_r,
-qu_1\cdots u_r\}$. So, \begin{equation}\label{ekeven}
    E_{\s\s}(k)=\begin{cases} \frac{1}{\varrho\delta}
    \big(1-\frac{q^2}{c_\s(k)^2})\big)
             \prod_{\alpha}\frac{c_\s(k)-c(\alpha)^{-1}}{c_\s(k)-c(\alpha)},
             &\text{if $\varrho^{-1}=q^{-1}\prod_{i=1}^r u_i$,}\\
\frac{1}{\varrho\delta}
    \big(1-\frac{1}{q^2c_\s(k)^2}\big)
             \prod_{\alpha}\frac{c_\s(k)-c(\alpha)^{-1}}{c_\s(k)-c(\alpha)},
             &\text{if $\varrho^{-1}=-q\prod_{i=1}^r u_i$, }\\
\end{cases}
\end{equation} where $\alpha$ runs over all addable and removable
nodes of $\s_{k-1}$ with $\alpha\neq \s_k\setminus \s_{k-1}$.

It follows from   Lemma~\ref{generic u} and
(\ref{ekodd})--(\ref{ekeven}) that
\begin{equation}\label{estne0}
E_{\s\t}(k)\ne0, \text{ if $\s\simk\t$}.\end{equation} Rewriting
Lemma~\ref{w-residue} yields the following equality:
\begin{equation}\label{w-res}
    \frac{W_k(y,\s)}{y}=\sum_{\t\simk\s}\frac{E_{\t\t}(k)}{y-c_\t(k)}.
\end{equation}

Given two partitions $\s$ and $\t$ write $\s\ominus\t=\alpha$ if
either $\t\subset\s$ and $\s\setminus\t=\alpha$, or $\s\subset\t$
and $\t\setminus\s=\alpha$.

Let $\mathfrak S_n$ be the symmetric group in $n$ letters. As an
Coxeter group, $\mathfrak S_n$  is generated by  $s_i:=(i, i+1)$
subject to the relations  $$\begin{cases} s_i^2=1, & \text{if $1\le i\le n-1$,}\\
s_is_j=s_js_i & \text{if $|i-j|>1$}\\
s_is_js_i=s_js_is_j, &\text{if $|i-j|=1$.}\\
\end{cases}
$$

Let $\s\in\UPD_n(\lambda)$ with  $\s_{k-1}\ne\s_{k+1}$, for some
$k$, $1\le k<n$.  Suppose that $\s_k\ominus\s_{k-1}$ and
$\s_{k+1}\ominus\s_k$ are in different rows and in different
columns. It is defined in \cite{AMR} that
$$s_k\s=(\s_1, \cdots,
\s_{k-1}, \t_{k}, \s_{k+1}, \cdots, \s_{n})\in \UPD_n(\lambda)$$
where $\t_{k}$ is the $r$-partition which is uniquely determined by
the conditions $\t_k\ominus\s_{k+1}=\s_{k-1}\ominus \s_{k}$ and
$\s_{k-1}\ominus\t_{k}=\s_{k}\ominus \s_{k+1}$. If the nodes
$\s_k\ominus\s_{k-1}$ and $\s_{k+1}\ominus\s_k$ are either  in the
same row, or  in the same column, then $s_k\s$ is not defined.

\begin{Defn}\label{ab}
    Let $\s\in\UPD_n(\lambda)$ with  $\s_{k-1}\ne\s_{k+1}$, for some $k$, $1\le k<n$.
 We define
$$
a_{\s}(k)=\frac{\delta c_{\s}(k+1)}{c_\s(k+1)-c_\s(k)}\quad\text{
and }\quad
        b_{\s}(k)={\sqrt{1-a_{\s}(k)^2+\delta a_{\s}(k)}}.
$$
We will fix the choice of square root for $b_\s(k)$ in (\ref{be}).
Since $\bfu$ is generic,   by Lemma~\ref{generic u}(b),
$c_\s(k+1)-c_\s(k)\ne0$. So, the formula for $a_s(k)$ makes sense.
\end{Defn}

As in \cite{AMR}, we do not define $a_\s(k)$ and $b_\s(k)$
 when $\s_{k-1}=\s_{k+1}$.  The following result can be verified
 easily.

\begin{Lemma}\label {x-equal}
    Suppose that $\s\in \UPD_n(\lambda)$ and $1\le k<n$. Then:
    \begin{enumerate}
    \item If $s_k\s$ is defined then $c_\s(k)=c_{s_k\s}(k+1)$
        and $c_\s(k+1)=c_{s_k\s}(k)$; consequently, $a_{s_k\s}(k)=\delta-a_\s(k)$.
    \item If $s_k\s$ is not defined then $a_{\s}(k)\in\{q,-q^{-1}\}$ and $b_{\s}(k)=0$.
    \end{enumerate}
\end{Lemma}

Finally, if $\s_{k-1}=\s_{k+1}$ and $\t\simk\s$, where $1\le k<n$,
we set
$$T_{\s\t}(k)=\delta \frac{E_{\s\t}(k)-\delta_{\s\t}}{c_\s(k)c_\t(k)-1}.$$
Note that $c_\s(k)c_\t(k)\ne1$ by Lemma~\ref{generic u}.

We will assume that we have chosen the square roots in the
definitions of $b_\s(k)$ and $E_{\s\t}(k)$ so that the following
equalities hold.

\begin{Assumption}[Root conditions]\label{be}
We assume that the ring $R$ is large enough so that
$\sqrt{E_{\s\s}(k)}\in R$ and $b_\s(k)=\sqrt{1-a_\s(k)^2+\delta
a_{\s}(k)}\in R$, for all $\s,\t\in\UPD_n(\lambda)$ and $1\le k<n$,
and that the following equalities hold:
\begin{enumerate}
  \item If $\s_{k-1}\ne\s_{k+1}$ and $s_k\s$ is defined
        then $b_{s_k\s}(k)=b_\s(k)$.
  \item If $\s_{k-1}\ne\s_{k+1}$ and $\s\siml\t$, where $|k-l|>1$,
      then $b_\s(k)=b_\t(k)$.
  \item If $\s_{k-1}\ne\s_{k+1}$, $\s_k\ne\s_{k+2}$ and $s_k\s$
    and $s_{k+1}\s$ are both defined then $b_{s_{k+1}\s}(k)=b_{s_k\s}(k+1)$.
  \item If $\s_{k-1}=\s_{k+1}$ and $\s_k=\s_{k+2}$ then
      $\sqrt{E_{\s\s}(k)}\sqrt{E_{\s\s}(k+1)}=1$.
  \item If $\s_{k-1}=\s_{k+1}$, $\t_{k-1}=\t_{k+1}$ and
      $E_{\s\s}(k)=E_{\t\t}(k)$ then
      $\sqrt{E_{\s\s}(k)}=\sqrt{E_{\t\t}(k)}$.
  \item If $\s_{k-1}=\s_{k+1}$, $\s_k=\s_{k+2}$ and
      $\t\simkk\s$, $\u\simk\s$ with $s_k\t$ and $s_{k+1}\u$ both
      defined and $s_k\t=s_{k+1}\u$ then
      $b_\t(k)\sqrt{E_{\t\t}(k+1)}=b_\u(k+1)\sqrt{E_{\u\u}(k)}$.
\end{enumerate}
\end{Assumption}

The following is the main result of this section.
\begin{Theorem} \label{seminormal} Let $\W_{r, n}$ be over a field $R$ such that the Assumption~4.1 holds.
  Let
$\Delta(\lambda)$ be the~$R$--vector space with basis
$\set{v_\s|\s\in\UPD_n(\lambda)}$. Then $\Delta(\lambda)$ becomes a
$\W_{r,n}$-module via
\begin{align*}
\bullet\quad
   T_k v_{\s}&=\begin{cases}
         \displaystyle \sum_{\t\simk\s} T_{\s\t}(k)v_\t,&
               \text{if }\s_{k-1}= \s_{k+1},\\[15pt]
     a_{\s}(k) v_{\s}+b_{\s}(k)v_{s_k\s}, &\text{if }\s_{k-1}\neq \s_{k+1},
      \end{cases}\\
\bullet\quad
    E_k v_{\s}&=\begin{cases}
        \displaystyle\sum_{\t\simk\s} E_{\s\t}(k)v_\t,&\text{if }\s_{k-1}=\s_{k+1}\\[15pt]
    0, & \text{if }\s_{k-1}\neq \s_{k+1},\\
      \end{cases}\\
\bullet\quad
     X_i v_\s&=c_\s(i)v_\s,
 \end{align*}
 for $1\le k<n$ and $1\le i\le n$.  When   $s_k\s$ is not defined, we set $v_{s_k\s}=0$.
\end{Theorem}

\begin{Defn}\label{seminormal rep}
We call $\Delta(\lambda)$ the \textsf{ seminormal representation} of
$\W_{r,n}(\bu)$ with respect to $\lambda$ for $\lambda\in
\Lambda_r^+(n-2f)$, and $0\le f\le \lfloor \frac n2\rfloor$.
\end{Defn}

In the remainder of this section, we will prove
Theorem~\ref{seminormal}. The rational functions $W_k(y, \s)$ play
the key role. As in \cite{AMR}, we  will work with formal (infinite)
linear combinations of elements of $\Delta(\lambda)$ and $\W_{r,n}$.
Let  $Z(A)$ be the center of an algebra $A$.

\begin{Lemma}\label{tilde W} Suppose that $R$ is a commutative ring which contains $1$ and invertible elements
$q,  \delta, u_1, \dots, u_r$ such that $\Omega\cup\{\varrho\}$ is
$\bu$-admissible.
   Given two integers  $k\ge1$ and  $a\ge0$.
     Then there is a
    $\omega^{(a)}_k\in Z(\W_{r,k-1})\cap R[X_1^{\pm 1},\dots,X_{k-1}^{\pm 1}]$
    such that
    $$E_kX_k^aE_k=\omega^{(a)}_kE_k.$$
Moreover, the generating series
    $\widetilde W_k(y)=\sum_{a=0}^\infty \omega_k^{(a)}y^{-a}$ satisfies

    \begin{equation} \label{omegaa}
    \frac{ \widetilde W_{k+1}(y)
    +\delta^{-1}\varrho-\frac{y^2}{y^2-1}}{
        \widetilde W_k(y)+\delta^{-1}\varrho-\frac{y^2}{y^2-1}}=
        \frac{(y-X_k)^2}{(y-X_k^{-1})^2}
         \cdot\frac{y-q^{-2}X_k^{-1}}{y-q^{-2}X_k}\cdot\frac{y-q^{2}X_k^{-1}}{y-q^{2}X_k}.
         \end{equation}
\end{Lemma}

\begin{proof} We prove the result by induction on $k$. When $k=1$,
the result follows from Definition~\ref{Waff relations}(g). In this
case,   $\tilde W_{1}(y)=\tilde w_{1, +}(y)$. Suppose that we have
already proved the result for all positive integers which are less
than $k+1$. In order to prove the result for $k+1$, we start with
the identity {\small \begin{equation}\label{omegaformula}
T_k\frac1{y-X_k}=\frac1{y-X_{k+1}}T_k+\delta\frac1{yX_k-1}E_k\frac1{y-X_k}
                         -\frac{\delta
                         X_{k+1}}{(y-X_k)(y-X_{k+1})},
                         \end{equation}}
We  multiply $(y-X_{k+1})(yX_k-1)$ (resp. $(y-X_k)$) on the left
(resp. right) of (\ref{omegaformula}). Then we use
Definition~\ref{Waff relations}(f),(j) to get the identity.
Multiplying (\ref{omegaformula}) on the left by $E_k$ and
    replacing $y$ by $y^{-1}$ yields
  {\small  \begin{equation}\label{omegaformula1} E_k\frac1{y-X_k}T_k=
  E_k(\frac{\varrho X_{k}}{yX_k-1}-\frac{\delta}{(y-X_k)(yX_{k}-1)})+\delta\frac{\widetilde W_k(y)}{y}E_k\frac1{yX_k-1}
                         \end{equation}}

Multiplying $T_k$ on the right of (\ref{omegaformula}) and  using
(\ref{omegaformula})--(\ref{omegaformula1}), we have
$$\begin{aligned} T_k\frac1{y-X_k}T_k&=\frac{1}{y-X_{k+1}}+\delta
                    T_k\frac{1}{y-X_k}+\frac{\delta^2 X_{k+1}}{(y-X_k)(y-X_{k+1})}\\
                   & -\frac{\delta^2}{yX_k-1}E_k\frac{1}{y-X_k}
                   -\frac{\varrho\delta X_k}{yX_k-1}E_k-\frac{\delta^2X_{k+1}^2}{(y-X_k)^2(y-X_{k+1})}
                  \\ &  +\frac{\delta^2}{X_k(y-X_k)(yX_k-1)}E_k\frac{1}{y-X_k}
                   -\frac{\delta}{y-X_k}T_k\frac{X_k}{y-X_k}\\ &-\frac{\delta^2 X_{k+1}}{(y-X_k)^2}
                     +\frac{\delta^2}{X_k(y-X_k)}E_k\frac{1}{y-X_k}
                   +\frac{\varrho\delta}{yX_k-1}E_k\frac{X_k}{yX_k-1}
                   \\ &-\frac{\delta^2}{yX_k-1}E_k\frac{1}{(y-X_k)(yX_k-1)}
                   +\frac{\delta^2 \widetilde
                   W_k(y)}{y(yX_k-1)}E_k\frac1{yX_k-1}\\
\end{aligned}$$
Note that $$\begin{aligned} & \frac{\delta^2
X_{k+1}}{(y-X_k)(y-X_{k+1})}-\frac{\delta^2X_{k+1}^2}{(y-X_k)^2(y-X_{k+1})}-\frac{\delta^2
X_{k+1}}{(y-X_k)^2}\\
= & \frac{\delta^2 X_k }{(y-X_k)^2} -\frac{\delta^2X_k}{(y-X_k)^2}
\frac{y}{y-X_{k+1}}.\\
\end{aligned}$$ By our induction assumption,
$$\begin{aligned} E_{k+1} T_{k} \frac 1{y-X_k} T_k E_{k+1}&
=(\frac 1 y-\frac{\delta^2X_k}{(y-X_k)^2}) E_{k+1}
\frac{y}{y-X_{k+1}} E_{k+1}\\ & +(
                    \frac{\varrho^{-1}\delta}{y-X_k} +\frac{\delta^2\omega_0 X_{k}}{(y-X_k)^2}
-\frac{\delta^2}{(yX_k-1)(y-X_k)} \\ &
 -\frac{\varrho\delta X_k}{yX_k-1}+\frac{\delta^2}{X_k(y-X_k)^2}
                    +\frac{\delta^2}{X_k(y-X_k)^2(yX_{k}-1)}\\ &
                   -\frac{\varrho^{-1}\delta X_k}{(y-X_k)^2}
                    +\frac{\varrho\delta X_k}{(yX_k-1)^2}
                    -\frac{\delta^2}{(yX_k-1)^2(y-X_k)}\\
               &      +\frac{\delta^2\widetilde
                   W_k(y)}{y(yX_k-1)^2})  E_{k+1}\\
\end{aligned}$$
    On the other hand,
\begin{align*}
& E_{k+1}T_k\frac1{y-X_k}T_kE_{k+1}=E_{k+1} E_k T_{k+1}^{-1}
\frac{1}{y-X_{k}} T_k E_{k+1}\\ =& E_{k+1}E_k \frac 1{y-X_k} E_k
E_{k+1}-\delta E_{k+1} E_k \frac 1
{y-X_k} T_k E_{k+1}+\delta \varrho^{-1} E_{k+1} \frac1{y-X_k} \\
\end{align*}
We use (\ref{omegaformula1}) to compute the second term in the right
hand side of the above equality.  Thus,  $$\begin{aligned}
E_{k+1}T_k\frac1{y-X_k}T_kE_{k+1}=&  (\frac{\widetilde
W_k(y)}{y}+\frac{\varrho^{-1}\delta}{y-X_k}
                         -\frac{\varrho\delta X_k}{yX_k-1}\\
    & +\frac{\delta^2}{(yX_k-1)(y-X_k)}-\frac{\delta^2 \widetilde
                      W_k(y)}{y(yX_k-1)})E_{k+1}\\
                      \end{aligned}
$$
Comparing the first and third  expressions of
$E_{k+1}T_k\frac1{y-X_k}T_kE_{k+1}$ yields
$$\begin{aligned} &( \frac{1}{y}  -\frac{\delta^2
X_{k}}{(y-X_k)^2})E_{k+1}\frac{y}{y-X_{k+1}}E_{k+1}+
                    (\frac{\varrho^{-1}\delta}{y-X_k}
                    +\frac{\delta^2\omega_0 X_{k}}{(y-X_k)^2}\\
                   &-\frac{\delta^2}{(yX_k-1)(y-X_k)}
                   -\frac{\varrho\delta
                   X_k}{yX_k-1}+\frac{\delta^2}{X_k(y-X_k)^2}\\
                   & +\frac{\delta^2}{X_k(y-X_k)^2(yX_{k}-1)}
                   -\frac{\varrho^{-1}\delta X_k}{(y-X_k)^2}
                    +\frac{\varrho\delta X_k}{(yX_k-1)^2}\\
                    &-\frac{\delta^2}{(yX_k-1)^2(y-X_k)}
                   +\frac{\delta^2\widetilde
                   W_k(y)}{y(yX_k-1)^2}) E_{k+1}\\
                   &=(\frac{\widetilde W_k(y)}{y}+\frac{\varrho^{-1}\delta}{y-X_k}
                         -\frac{\varrho\delta X_k}{yX_k-1}
     +\frac{\delta^2}{(yX_k-1)(y-X_k)}-\frac{\delta^2 \widetilde
                      W_k(y)}{y(yX_k-1)})E_{k+1}
\end{aligned}$$
So,
$$\begin{aligned} &( \frac{1}{y}  -\frac{\delta^2
X_{k}}{(y-X_k)^2})E_{k+1}\frac{y}{y-X_{k+1}}E_{k+1}\\
                   = & (-\frac{\varrho^{-1}\delta}{y-X_k}
                    -\frac{\delta^2\omega_0 X_{k}}{(y-X_k)^2}\\
                   &+\frac{\delta^2}{(yX_k-1)(y-X_k)}
                   +\frac{\varrho\delta
                   X_k}{yX_k-1}-\frac{\delta^2}{X_k(y-X_k)^2}\\
                   & -\frac{\delta^2}{X_k(y-X_k)^2(yX_{k}-1)}
                   +\frac{\varrho^{-1}\delta X_k}{(y-X_k)^2}
                    -\frac{\varrho\delta X_k}{(yX_k-1)^2}\\
                    &+\frac{\delta^2}{(yX_k-1)^2(y-X_k)}
                   -\frac{\delta^2\widetilde
                   W_k(y)}{y(yX_k-1)^2}\\
                   &+\frac{\widetilde W_k(y)}{y}+\frac{\varrho^{-1}\delta}{y-X_k}
                         -\frac{\varrho\delta X_k}{yX_k-1}
     +\frac{\delta^2}{(yX_k-1)(y-X_k)}-\frac{\delta^2 \widetilde
                      W_k(y)}{y(yX_k-1)})E_{k+1}\\
                      =&  ( \frac{1}{y}  -\frac{\delta^2
                      X_{k}}{(yX_k-1)^2}) \widetilde W_k(y)E_{k+1}
                      +\varrho\delta
                      X_k(\frac{1}{(y-X_k)^2}-\frac{1}{(yX_k-1)^2})E_{k+1}\\
                      &+\frac{\delta^2X_ky^2(1-X_{k}^2)}{(yX_k-1)^2(y-X_k)^2}E_{k+1}
\end{aligned}$$
Multiplying  $(yX_k-1)^2(y-X_k)^2$ on both sides of above equation
yields
$$\begin{aligned} &(yX_k-1)^2( \frac{(y-X_k)^2}{y}  -\delta^2
X_{k})E_{k+1}\frac{y}{y-X_{k+1}}E_{k+1}\\
                    = &(y-X_k)^2( \frac{(yX_k-1)^2}{y}  -\delta^2
                      X_{k}) \widetilde W_k(y)E_{k+1}
                      +\varrho\delta
                      X_k((yX_k-1)^2\\ & -(y-X_k)^2)E_{k+1}
                      +\delta^2X_ky^2(1-X_{k}^2)E_{k+1}.
\end{aligned}$$
So,
$$\begin{aligned} &(y-X_k^{-1})^2( (y-X_k)^2  -\delta^2
X_{k}y)E_{k+1}\frac{y}{y-X_{k+1}}E_{k+1}\\
                    = & (y-X_k)^2((y-X_k^{-1})^2  -\delta^2
                      X_{k}^{-1}y) \widetilde W_k(y)E_{k+1}
                      +\delta^2X_k^{-1}y^3(1-X_{k}^2)E_{k+1}\\
                      &+ \varrho\delta y
                      (X_k(y-X_k^{-1})^2-X_k^{-1}(y-X_k)^2)E_{k+1}.\\
\end{aligned}$$
Therefore,
$$\begin{aligned}
 &(y-X_k^{-1})^2(y-q^{-2}X_k)(y-q^{2}X_k)
 E_{k+1}\frac{y}{y-X_{k+1}}E_{k+1}\\ = &  (y-X_k)^2(y-q^{-2}X_k^{-1})(y-q^{2}X_k^{-1})\widetilde
 W_k(y)E_{K+1}\\
    &-(\delta^{-1}\varrho-\frac{y^2}{y^2-1})((y-X_k^{-1})^2(y-q^{-2}X_k)(y-q^{2}X_k)\\ & -(y-X_k)^2(y-q^{-2}X_k^{-1})(y-q^{2}X_k^{-1}))
\end{aligned}$$
Thus,
$$E_{k+1}\frac{y}{y-X_{k+1}}E_{k+1}=\widetilde W_{k+1}(y)E_{k+1},$$
where  $\widetilde W_{k+1}(y)$ satisfies the recursive relation
(\ref{omegaa}).

 Note that $X_k T_j=T_j X_k$ and $X_k E_j=E_j X_k$
for all positive integers $j\le k-2$. By induction assumption and
(\ref{omegaa}), $\omega_{k+1}^{(a)}$ commutates with
       $E_1,\dots,E_{k-2}$ and
      $T_1,\dots,T_{k-2}$. In order to verify \begin{equation}\label{Y} \omega_{k+1}^{(a)} Y =Y \omega_{k+1}^{(a)},
\text{ for $Y\in \{T_{k-1}, E_{k-1}\}$,}\end{equation}   we use the
following series
$$\sum_{m\ge0}a_mz^m=\frac{(1+X_{k-1}z)(1+X_kz)}{(1+X_{k-1}^{-1}z)(1+X_k^{-1}z)},$$
     where $z\in \{-y^{-1}, -q^{\pm2}y\}$,  $$\begin{aligned} a_0&=1,
     a_1=X_{k-1}+X_k-X_{k-1}^{-1}-X_{k}^{-1},\\
      a_2& =X_{k}X_{k-1}-X_{k}^{-1}X_{k-1}^{-1}-(X_{k}^{-1}+X_{k-1}^{-1})(X_{k}+X_{k-1}-X_{k}^{-1}-X_{k-1}^{-1})\\
     a_m& =-({X_{k-1}^{-1}+X_k^{-1}}) a_{m-1}-{X_{k-1}^{-1}X_k^{-1}}a_{m-2},\text{ for $ m\ge3 $.}\\
     \end{aligned}$$
   By induction, we have $E_{k-1}a_{j}=a_j E_{k-1}=0$ for all non-negative integers $j$.

 We  verify
  $ T_{k-1}a_m=a_mT_{k-1}$ for all $m\ge 0$ by induction on $m$.
There is nothing to be proved when $m=0$ since $a_0=1$.  By
Definition~\ref{Waff relations}(f),
$$
\begin{cases} & T_{k-1}(X_{k-1}+X_k)=(X_{k-1}+X_k)T_{k-1}+\delta(X_k E_{k-1} - E_{k-1}X_k)
 \\
& T_{k-1}(X_{k-1}^{-1}+X_k^{-1})
=(X_{k-1}^{-1}+X_k^{-1})T_{k-1}+\delta(X_{k-1}^{-1} E_{k-1} -
E_{k-1}X_{k-1}^{-1})\\
 \end{cases}$$
So, $T_{k-1}a_1=a_1T_{k-1}$.
 When  $m\ge 2$, since $T_{k-1}$ commute with
 $X_{k}X_{k-1}$ and $X_{k}^{-1}X_{k-1}^{-1}$, we have, by induction
 assumption that
 $$\begin{aligned}
 T_{k-1}a_m-a_mT_{k-1}= & (X_{k-1}^{-1}+X_k^{-1}) a_{m-1} T_{k-1}  -T_{k-1}(X_{k-1}^{-1}+X_k^{-1})a_{m-1}
\\
= &(X_{k-1}^{-1}+X_k^{-1}) a_{m-1}
T_{k-1}-(X_{k-1}^{-1}+X_k^{-1})T_{k-1}a_{m-1} \\ & -
\delta(X_{k-1}^{-1}
E_{k-1} - E_{k-1}X_{k-1}^{-1})a_{m-1}  = 0.\\
 \end{aligned}$$

Using (\ref{omegaa}) twice, we have {\small
\begin{equation}\label{w0}\begin{aligned} \frac{ \widetilde
W_{k+1}(y)
    +\delta^{-1}\varrho-\frac{y^2}{y^2-1}}{
        \widetilde W_{k-1}(y)+\delta^{-1}\varrho-\frac{y^2}{y^2-1}}&= \frac{(y-X_k)^2}{(y-X_k^{-1})^2}
         \cdot\frac{y-q^{-2}X_k^{-1}}{y-q^{-2}X_k}\cdot\frac{y-q^{2}X_k^{-1}}{y-q^{2}X_k}\\
       &\quad\frac{(y-X_{k-1})^2}{(y-X_{k-1}^{-1})^2}
         \cdot\frac{y-q^{-2}X_{k-1}^{-1}}{y-q^{-2}X_{k-1}}\cdot\frac{y-q^{2}X_{k-1}^{-1}}{y-q^{2}X_{k-1}}.
\end{aligned} \end{equation}}
Thus, $T_{k-1}, E_{k-1}$ commute with the right hand side of
(\ref{w0}). By induction assumption,  $T_{k-1}, E_{k-1}$ commutes
with $\tilde W_{k-1}(y)$, it has to commute with $\tilde
W_{k+1}(y)$. Thus, $\omega_{k+1}^{(a)}$ commutes with $T_{k-1}$ and
$E_{k-1}$.

Now, we prove  $\omega_k^{(a)}\in R[X_1^{\pm1}, \cdots, X_{k-1}^{\pm
1}]$.  Let $g_k(X_k)=(y-X_k^{-1})^2(y-q^{-2}X_k)(y-q^{2}X_k)$. By
(\ref{omegaa}), {\small  $$
 \widetilde W_{k+1}(y) g_k(X_k) =  \widetilde
W_{k}(y) g_k(X_k^{-1})
+(X_k-X_k^{-1})\delta^2y((\delta^{-1}\rho-1)y^2-\delta^{-1}\rho).$$
} Comparing the coefficients of $y^j$ for  $j\leq 4$ on both sides
of the last equation, we have the following results:
\begin{enumerate}
\item   $j=4$:  $\omega_{k+1}^{(0)}=\omega_{k}^{(0)}$. By induction
assumption, $\omega_{k+1}^{(0)}=\omega_0$.

\item   $j=3$:
$ \omega_{k+1}^{(1)} =\omega_{k}^{(1)}+\delta\rho^{-1}
(X_k-X_k^{-1}) $.

\item  $j=2$:
{\small $$ \omega_{k+1}^{(2)}=\omega_{k}^{(2)}
-\omega_{k}^{(1)}(2X_k+(q^2+q^{-2})X_k^{-1})
+\omega_{k+1}^{(1)}(2X_k^{-1}+(q^2+q^{-2})X_k).
$$
}
\item  $j=1$:
$$
\begin{aligned}
&\omega_{k+1}^{(3)}-\omega_{k}^{(3)}+(\omega_{k+1}^{(1)}-\omega_k^{(1)})(X_k^2+X_k^{-2}+2(q^2+q^{-2}))
\\
=&(\omega_{k+1}^{(2)}-\omega_{k}^{(0)})(2X_k^{-1}+(q^2+q^{-2})X_k)-\delta\rho(X_k-X_k^{-1})\\
& +(\omega_{k+1}^{(0)}-\omega_{k}^{(2)})
(2X_k+(q^2+q^{-2})X_k^{-1}).
\\
\end{aligned}
$$
\item $j\leq 0$:
$$
\begin{aligned}
&\omega_{k+1}^{(-j+4)}+\omega_{k+1}^{(-j)}+(\omega_{k+1}^{(-j+2)}-\omega_k^{(-j+2)})(X_k^2+X_k^{-2}+2(q^2+q^{-2}))\\
=&\omega_{k}^{(-j)}+\omega_{k}^{(-j+4)}+
(\omega_{k+1}^{(-j+1)}-\omega_{k}^{(-j+3)})(2X_k+(q^2+q^{-2})X_k^{-1})\\ & +(\omega_{k+1}^{(-j+3)}-\omega_{k}^{(-j+1)})(2X_K^{-1}+(q^2+q^{-2})X_k).\\
\end{aligned}
$$\end{enumerate}
By induction assumptions on $k$ and $a$ together with the formulae
in (a)-(e), we have that $\omega_k^{(a)}\in R[X_1^{\pm1}, \cdots,
X_{k-1}^{\pm 1}]$.
\end{proof}

\begin{Cor} \label{relation1} Suppose that $R$ is a commutative ring which contains $1$ and invertible elements
$q, \delta, u_1, \dots, u_r$ such that $\Omega\cup\{\varrho\}$ is
$\bu$-admissible.
 Given a  positive integer $k\le n-1$.
If $a\in \mathbb Z$, then
\begin{enumerate}
\item
$E_kX_k^aT_{k-1}^{\pm 1}E_k=\sum_{i=-a}^{a}f_iX_{k-1}^iE_k$;
\item
$E_kX_k^aT_{k+1}^{\pm 1}
E_k=\sum_{i=-a}^{a}g_iX_{k+2}^iE_k$,\end{enumerate} where $f_i,
g_i\in Z(\W_{r,k-1})\cap R[X_1^{\pm 1},\dots,X_{k-1}^{\pm 1}]$ for
$-a\le i\le a$.
\end{Cor}
\begin{proof} Both (a) an (b)  follow from Lemma~\ref{relations2}(1)(2)(4)(5) and Lemma~\ref{tilde W} and Definition~\ref{Waff relations}.\end{proof}

By Lemma~\ref{tilde W}, we have
\begin{equation}\label{wk}\frac{
    \widetilde W_k(y)+\delta^{-1}\varrho-\frac{y^2}{y^2-1}}{
        \widetilde W_1(y)+\delta^{-1}\varrho-\frac{y^2}{y^2-1}}=
         \prod_{i=1}^{k-1} \frac{(y-X_i)^2}{(y-X_i^{-1})^2}
         \cdot\frac{y-q^{-2}X_i^{-1}}{y-q^{-2}X_i}\cdot\frac{y-q^{2}X_i^{-1}}{y-q^{2}X_i}.
\end{equation}
Since $\omega_k^{(a)}\in R[X_1^{\pm1}, \cdots, X_{k-1}^{\pm 1}]$, we
can define $\widetilde W_k(y,\s)\in R((y^{-1}))$ by $$\widetilde
W_k(y)v_\s=\widetilde W_k(y,\s)v_\s.$$

\begin{Prop}\label{W=tilde W}
Given an $\s\in \UPD_n(\lambda)$ and a positive integer  $ k\le n$,
we have $W_k(y,\s)=\widetilde W_k(y,\s)$.
\end{Prop}

\begin{proof}
As $\Omega\cup\{\varrho\}$ is $\bu$-admissible, by
Lemma~\ref{uadmissequi} and (\ref{wk}),  we have $$\begin{aligned} &
\widetilde W_k(y, \s)+\delta^{-1}\varrho-\frac{y^2}{y^2-1}
       \\ = &  A\cdot\prod_{\ell =1}^r \frac{(y-u_\ell^{-1})}{(y-u_\ell)}
       \times    \prod_{i=1}^{k-1} \frac{(y-c_\s (i))^2}{(y-c_\s (i)^{-1})^2}
         \cdot\frac{y-q^{-2}c_\s (i)^{-1}}{y-q^{-2}c_\s (i)} \cdot\frac{y-q^{2}c_\s (i)^{-1}}{y-q^{2}c_\s
         (i)}.\\
         \end{aligned}$$
where
$$A=
(\varrho\delta^{-1}\prod_{\ell=1}^ru_\ell+\dfrac{y\gamma_r(y)}{y^{2}-1})
\prod\limits_{\ell=1}^ru_\ell.
$$
By  arguments similar to those for \cite[4.17]{AMR}  we can verify
that
$$ \widetilde W_k(y,\s)+\delta^{-1}\varrho-\frac{y^2}{y^2-1}
         =A\prod_{\alpha}\frac{y-c(\alpha)^{-1}}{y-c(\alpha)},$$
where $\alpha$ runs over the addable and removable nodes
of~$\s_{k-1}$. This proves $W_{k}(y, \s)=\tilde W_{k} (y, \s)$.
\end{proof}
One can verify the following result by similar  arguments to those
for  \cite[4.18]{AMR}.

\begin{Cor}\label{exe=oe}
    Suppose that $\s\in \UPD_n(\lambda)$ and that $1\le k<n$ and $a\ge0$. Then
    $E_kX_k^aE_kv_\s=\omega_k^{(a)} E_k v_\s$.
\end{Cor}

\begin{Lemma}\label{ees} Suppose that $\s\in \UPD_n(\lambda)$ with
$\s_{k-1}=\s_{k+1}$ and $\s_k=\s_{k+2}$. Then
$E_{\s\s}(k)E_{\s\s}(k+1)=1$.
\end{Lemma}

\begin{proof} By (\ref{omegaa})
and  Proposition \ref{W=tilde W},
$$
\begin{aligned}\label{abcd}
 W_{k+1}(y,\s)+&\delta^{-1}\varrho-\frac{y^2}{y^2-1}=
(W_k(y,\s)+\delta^{-1}\varrho-\frac{y^2}{y^2-1})\\
& \times {(y-c_{\s}(k))^2\over
(y-c_{\s}(k)^{-1})^2}{(y-q^{-2}c_{\s}(k)^{-1})\over
(y-q^{-2}c_{\s}(k))}{(y-q^{2}c_{\s}(k)^{-1})\over
(y-q^{2}c_{\s}(k))},
\end{aligned}$$
where $W_k(y,\s)$ is given by Definition~\ref{rational-W}. Note that
$c_\s(k) c_\s(k+1)=1$ and   $$E_{\s\s}(k+1)= Res_{y=c_\s(k+1)}
\frac{W_{k+1}(y,\s) }{y}= Res_{y=c_\s(k)^{-1}} \frac{W_{k+1}(y,\s)
}{y}.$$  There are four cases we need to discuss:

\Case {1. $2\nmid r$ and $\varrho^{-1}=u_1u_2\cdots u_r$}
\begin{align*}
E_{\s\s}(k+1)=&
\Big(\delta^{-1}c_{\s}(k)+\dfrac{c_{\s}(k)^2}{1-c_{\s}(k)^2}
\Big)\varrho^{-1}\prod\limits_{\t\simk\s,\t\neq \s}
\dfrac{c_{\s}(k)^{-1} - c_{\t}(k)^{-1}}{c_{\s}(k)^{-1}
-c_{\t}(k)}\\
&\Big(c_{\s}(k)^{-1}-c_{\s}(k)\Big)
\dfrac{c_{\s}(k)^{-1}-q^{-2}c_{\s}(k)^{-1}}{c_{\s}(k)^{-1}-q^{-2}c_{\s}(k)}
\dfrac{c_{\s}(k)^{-1}-q^2c_{\s}(k)^{-1}}{c_{\s}(k)^{-1}-q^2c_{\s}(k)}\\
&=\dfrac{\varrho\delta c_{\s}(k)^2}{c_{\s}(k)^2+\delta
c_{\s}(k)-1}\prod\limits_{\t\simk\s,\t\neq \s} \dfrac{c_{\s}(k) -
c_{\t}(k)}{c_{\s}(k) -c_{\t}(k)^{-1}}=\frac{1}{E_{\s\s}(k)}
\end{align*}

\Case{2. $2\nmid r$  and $\varrho^{-1}=-\prod_{l=1}^r u_i$}
\begin{align*}
E_{\s\s}(k+1)=&
\Big(-\delta^{-1}c_{\s}(k)+\dfrac{c_{\s}(k)^2}{1-c_{\s}(k)^2}
\Big)\Big(-\varrho^{-1}\Big)\prod\limits_{\t\simk\s,\t\neq \s}
\dfrac{c_{\s}(k)^{-1} - c_{\t}(k)^{-1}}{c_{\s}(k)^{-1}
-c_{\t}(k)}\\
&\Big(c_{\s}(k)^{-1}-c_{\s}(k)\Big)
\dfrac{c_{\s}(k)^{-1}-q^{-2}c_{\s}(k)^{-1}}{c_{\s}(k)^{-1}-q^{-2}c_{\s}(k)}
\dfrac{c_{\s}(k)^{-1}-q^2c_{\s}(k)^{-1}}{c_{\s}(k)^{-1}-q^2c_{\s}(k)}\\
=&\dfrac{\varrho\delta c_{\s}(k)^2}{c_{\s}(k)^2-\delta
c_{\s}(k)-1}\prod\limits_{\t\simk\s,\t\neq \s} \dfrac{c_{\s}(k) -
c_{\t}(k)}{c_{\s}(k) -c_{\t}(k)^{-1}}=\frac{1}{E_{\s\s}(k)}
\end{align*}

\Case{3. $2\mid r$ and $\varrho^{-1}=q^{-1} \prod_{l=1}^r u_i$}
\begin{align*}
E_{\s\s}(k+1)= &
\Big(q\delta^{-1}c_{\s}(k)-\dfrac{c_{\s}(k)}{1-c_{\s}(k)^2}
\Big)q\varrho^{-1}\prod\limits_{\t\simk\s,\t\neq \s}
\dfrac{c_{\s}(k)^{-1} - c_{\t}(k)^{-1}}{c_{\s}(k)^{-1}
-c_{\t}(k)}\\
&\Big(c_{\s}(k)^{-1}-c_{\s}(k)\Big)
\dfrac{c_{\s}(k)^{-1}-q^{-2}c_{\s}(k)^{-1}}{c_{\s}(k)^{-1}-q^{-2}c_{\s}(k)}
\dfrac{c_{\s}(k)^{-1}-q^2c_{\s}(k)^{-1}}{c_{\s}(k)^{-1}-q^2c_{\s}(k)}\\
= &\dfrac{\varrho\delta
c_{\s}(k)^2}{c_{\s}(k)^2-q^2}\prod\limits_{\t\simk\s,\t\neq \s}
\dfrac{c_{\s}(k) - c_{\t}(k)}{c_{\s}(k)
-c_{\t}(k)^{-1}}=\frac{1}{E_{\s\s}(k)}
\end{align*}
\Case {4. $2\mid r$ and $\varrho^{-1}=-q\prod_{l=1}^r u_i$}
\begin{align*}
E_{\s\s}(k+1)=&
\Big(-q^{-1}\delta^{-1}c_{\s}(k)-\dfrac{c_{\s}(k)}{1-c_{\s}(k)^2}
\Big)\Big(-q^{-1}\varrho^{-1}\Big)\prod\limits_{\t\simk\s,\t\neq \s}
\dfrac{c_{\s}(k)^{-1} - c_{\t}(k)^{-1}}{c_{\s}(k)^{-1}
-c_{\t}(k)}\\
&\Big(c_{\s}(k)^{-1}-c_{\s}(k)\Big)
\dfrac{c_{\s}(k)^{-1}-q^{-2}c_{\s}(k)^{-1}}{c_{\s}(k)^{-1}-q^{-2}c_{\s}(k)}
\dfrac{c_{\s}(k)^{-1}-q^2c_{\s}(k)^{-1}}{c_{\s}(k)^{-1}-q^2c_{\s}(k)}\\
=& \dfrac{\varrho\delta
c_{\s}(k)^2}{c_{\s}(k)^2-q^{-2}}\prod\limits_{\t\simk\s,\t\neq \s}
\dfrac{c_{\s}(k) - c_{\t}(k)}{c_{\s}(k)
-c_{\t}(k)^{-1}}=\frac{1}{E_{\s\s}(k)}.
\end{align*}
We remark that we use Lemma~\ref{simk}, (\ref{contidentity}) and
(\ref{ekodd})--(\ref{ekeven}) when we verify the equalities in
cases~1--4.
\end{proof}

The following result can be proved by  similar  arguments to those
for \cite[4.20]{AMR}. The only difference is that we use our
rational functions $W_k(y, \s)$ instead of those for cyclotomic
Nazarov-Wenzl algebras.

\begin{Lemma}\label{be equality}
Fix an integer $k$ with $1\le k<n-1$ and suppose that
$\s,\t,\u\in\UPD_n(\lambda)$ such that $\s_{k-1}=\s_{k+1}$,
$\s_k=\s_{k+2}$, $\t\simkk\s$, $\u\simk\s$ and that $s_k\t$ and
$s_{k+1}\u$ are both defined with $s_k\t=s_{k+1}\u$. Then
$$b_\t(k)^2E_{\t\t}(k+1)=b_\u(k+1)^2E_{\u\u}(k).$$
\end{Lemma}

The following combinatorial identities will be used in the proof of
Theorem~\ref{seminormal}.

\begin{Prop}\label{identity}  Suppose that $\s, \t'\in \UPD_n(\lambda)$ with
$\s_{k-1}=\s_{k+1}$, $\s_k\neq \s_{k+2}$,  $ \t'\simk \s$ and
$\t'\neq \s$, where $1\le k<n-1$. Then the following identities
hold:
\begin{enumerate}
    \item$\displaystyle\sum_{\t\simk\s} {E_{\t\t}(k)\over
c_\s(k)c_\t(k)-1} =\delta^{-1}\varrho+{\frac1{ c_\s(k)^2-1}},$
\item
{\small
$\displaystyle\sum_{\t\simk\s}\frac{E_{\t\t}(k)}{(c_\s(k)c_\t(k)-1)^2}
         =\frac{c_{\s}(k)^2+1}{(c_{\s}(k)^2-1)^2}-\delta^{-1}\varrho
+(\frac{1} {\delta^2}
         -\frac{c_{\s}(k)^2}{(c_{\s}(k)^2-1)^2})
         \frac{1} {E_{\s\s}(k)}$}
\item
$\displaystyle\sum_{\t\simk\s}
  {E_{\t\t}(k)\over (c_\s(k)c_\t(k)-1) (c_\t(k)c_{\t'}(k)-1)}=\frac{c_{\s}(k)c_{\t'}(k)+1}{(c_{\s}(k)^2-1)(c_{\t'}(k)^2-1)}-\delta^{-1}\varrho,$
\end{enumerate}
\end{Prop}

\begin{proof}
Evaluating both sides of (\ref{w-res})  at $y=c_\s(k)^{-1}$ and
using (\ref{rational-W}) gives (a). By Proposition~\ref{W=tilde W}
and Corollary~\ref{exe=oe} we have
$$ E_k {1\over (y-X_k) (v-X_k)}E_kv_{\s}
       ={1\over v-y}\Big( {W_k(y,\s)\over y}-{W_k(v,\s)\over v}\Big) E_k
v_{\s}.$$ Comparing the coefficients of $v_{\s}$ on both sides of
this equation yields
$$
  \sum_{\t\simk \s}\frac{E_{\t\t}(k)}{(y-c_\t(k))(v-c_\t(k))}
    =\frac1{v-y}\Big\{ \frac{W_k(y,\s)}{y}-\frac{W_k(v,\s)}{v} \Big\}.
$$
Let $y=c_\s(k)^{-1}$. We use  (a) to rewrite the above equality and
obtain the following equality:
\begin{equation}\label{b1}
\begin{aligned} & \sum_{\t\simk
\s}\frac{E_{\t\t}(k)}{(1-c_\s(k)c_\t(k))(v-c_\t(k))}
   \\ =&\frac{W_k(v,\s)+\delta^{-1}\varrho-\frac{v^2}{v^2-1}}
   {v(1-vc_\s(k))}-v^{-1}\delta^{-1}\varrho
   +\frac{v+c_{\s}(k)}{(v^2-1)(1-c_{\s}(k)^2)}  \\
\end{aligned}\end{equation}
Setting $v=c_{\t'}(k)^{-1}$ gives (c). Now we set $v=c_\s(k)^{-1}$.
There are four cases we need to  discuss.

When  $2\nmid r$ and $\prod_{l=1}^r u_i=\varrho^{-1}$, it follows
from
  (\ref{b1}) that
\begin{align*}
& \sum_{\t\simk\s}\frac{E_{\t\t}(k)}{(1-c_\s(k)c_\t(k))^2}
         +\delta^{-1}\varrho-\frac{1+c_{\s}(k)^2}{(1-c_{\s}(k)^2)^2}\\
         =& \varrho^{-1}\Big(\delta^{-1}c_{\s}(k)+\frac{c_{\s}(k)^2}{1-c_{\s}(k)^2}\Big)
          \prod_{\substack{\t\simk\s\\\t\neq\s}}\frac{c_\s(k)^{-1}-c_\t(k)^{-1}}{c_\s(k)^{-1}-c_\t(k)}\frac{-c_\s(k)^{-2}}{c_{\s}(k)^{-1}-c_{\s}(k)}
\\=& \frac{c_\s(k)^2-\delta c_\s(k)-1}{\delta (c_\s(k)^2-1)^2}
\varrho^{-1} c_{\s}(k)^2 \prod_{\t\simk \s} c_{\t}(k)^{-2}
\prod_{\t\neq \s} \frac{c_\t(k)-c_\s(k)}{c_\t(k)^{-1}-c_\s(k)}\\
=&  \frac{c_\s(k)^2-\delta c_\s(k)-1}{\delta (c_\s(k)^2-1)^2}
\varrho c_{\s}(k)^2 \prod_{\t\neq \s}
\frac{c_\t(k)-c_\s(k)}{c_\t(k)^{-1}-c_\s(k)}\text{ by (\ref{contidentity})} \\
=&(\delta^{-2}-\frac{c_{\s}(k)^2}{(1-c_{\s}(k)^2)^2})\frac{1}{E_{\s\s}(k)} \text{ by (\ref{ekodd}).}\\
\end{align*}

This proves (b) under the assumption  $2\nmid r$ and $\prod_{l=1}^r
u_i=\varrho^{-1}$. One can verify (b) in other cases. \end{proof}

We are going to check that the action of $\W_{r,n}(\bu)$ on
$\Delta(\lambda)$ respects the defining relations of
$\W_{r,n}(\bu)$.

\begin{Lemma}\label{start relations}
Suppose  $\s\in\UPD_n(\lambda)$. Then
\begin{enumerate}
    \item $E_i^2v_\s=\omega_0 E_iv_\s$, for $1\le i<n$.
    \item $E_1X_1^kE_1v_\s=\omega_kE_1v_\s$, for $k\in \mathbb Z$.
    \item $(X_1-u_1)(X_1-u_2)\cdots (X_1-u_r)v_{\s}=0$.
    \item $X_iX_j v_{\s}=X_jX_i v_{\s}$ for $1\le i, j\le n$.
    \item $E_iX_iX_{i+1} v_{\s}=X_iX_{i+1} E_i v_{\s}=E_iv_{\s}$,
           $1\le i\le n-1$.
    \item $(T_iX_i-X_{i+1}T_i)v_{\s}=\delta X_{i+1}(E_i-1)v_{\s}$, for $1\le i\le n-1$
    \item $(X_iT_i-T_{i+1}X_i)v_{\s}=\delta(E_i-1)X_{i+1}v_{\s}$, for $1\le i\le n-1$
    \item $T_kX_lv_{\s}=X_lT_kv_{\s}$ if  $l\neq k, k+1$.
    \item $E_kE_{k\pm1} E_kv_{\s}=E_k v_{\s}$.
    \item  $E_kT_k v_{\s}=\varrho E_k v_{\s}=T_kE_k v_{\s}$.
    \item $T_iT_j v_\s=T_jT_i v_\s$ if $|i-j|>1$.
    \item $X_iX_i^{-1}=X_i^{-1} X_i=1$ for $1\le i\le n$.
\end{enumerate}\end{Lemma}

\begin{proof} We have already proved (a) and (b) for  $k>0$ in
Corollary~\ref{exe=oe}. (c)-(h) and (l) can be verified easily. By
(c), we have (b) for all $k\in \mathbb Z$ with $k<0$. (i)-(k) can be
proved by
    arguments similar to those in \cite[4.23,4.25, 4.27a]{AMR}. When we
    prove (j) we need to use Proposition~\ref{identity}(a) instead of
    \cite[4.21a]{AMR}.
\end{proof}

It remains to check the defining relations (b), (c)(ii), (h)(ii)  in
Definition~\ref{Waff relations}.

\begin{Lemma}\label{s-square} Suppose that $\s\in \UPD_n(\lambda)$.
Then $(T_k^2-\delta T_k+\delta\varrho E_k) v_{\s}=v_{\s}$.
\end{Lemma}

\begin{proof} We prove the result by computing the coefficient of
$v_\t$ in the expression of $(T_k^2-\delta T_k+\delta\varrho E_k)
v_{\s}$. There are two cases we  have to discuss as follows.

\Case{1. $\s_{k-1}\neq \s_{k+1}$} Then $E_kv_\s=0$. If $s_k\s$ is
not defined then $a_{\s}(k)\in\{q, -q^{-1}\}$ and $b_{\s}(k)=0$ (see
Lemma~\ref{x-equal}). So, $(T_k^2-\delta T_k+\delta\varrho
E_k)v_{\s}=v_{\s}$. If $s_k\s\in\UPD_n(\lambda)$ then by the choice
of the square roots in (\ref{be})(a) we have $$\begin{aligned}
(T_k^2-\delta T_k+\delta\varrho E_k)v_{\s}
    =&(T_k-\delta)\Big(a_{\s}(k)v_{\s}+b_{\s}(k) v_{s_k\s}\Big)\\
    =&a_{\s}(k)\Big(a_{\s}(k)v_{\s}+b_{\s}(k) v_{s_k\s}\Big)-\delta\Big(a_{\s}(k)v_{\s}+b_{\s}(k) v_{s_k\s}\Big)
      \\ & +b_{\s}(k)\Big(a_{s_k\s}(k)v_{s_k\s}+b_{s_k\s}(k) v_{\s}\Big)\\
   = & v_{\s}\quad {\small(\text{by
   Lemma~\ref {x-equal}.})}
\end{aligned}$$
\Case{2.  $\s_{k-1}=\s_{k+1}$} We have $$\begin{aligned} &
(T_k^2-\delta T_k+\delta\varrho E_k)v_\s\\ =& \sum_{\t\simk\s}
        \Big(\sum_{\v\simk\s}T_{\s\v}(k)T_{\v\t}(k)\Big)v_\t-\delta\sum_{\t\simk\s}T_{\s\t}(k)v_{\t}+\varrho\delta\sum_{\t\simk\s}E_{\s\t}(k)v_{\t}.\\
\end{aligned}$$
The coefficient of $v_\s$ in~$(T_k^2-\delta T_k+\delta\varrho
E_k)v_\s$ is equal to $1$ since $$\begin{aligned} &\quad
\sum_{\v\simk \s} T_{\s\v}(k) T_{\v\s}(k)-\delta
  T_{\s\s}(k)+\varrho\delta E_{\s\s}(k)\\
=&\sum_{\v\simk\s}\frac{\delta^2
  E_{\s\s}(k)E_{\v\v}(k)}{(c_\s(k)c_\v(k)-1)^2}
              -\frac{\delta^2
              (E_{\s\s}(k)-1)}{c_\s(k)^2-1}+\delta\varrho
              E_{\s\s}(k)+\frac{\delta^2(1-2E_{\s\s}(k))}{(c_\s(k)^2-1)^2}\\
=&\delta^2
E_{\s\s}(k)\Big(\frac{1+c_{\s}(k)^2}{(1-c_{\s}(k)^2)^2}+\frac{1}{\delta^2
E_{\s\s}(k)}-\frac{c_{\s}(k)^2}{(1-c_{\s}(k)^2)^2E_{\s\s}(k)}\Big)\\
&+\frac{\delta^2(1-2E_{\s\s}(k))}{(1-c_{\s}(k)^2)^2}+\frac{\delta^2(E_{\s\s}(k)-1)}{1-c_{\s}(k)^2}\quad
\text{(by Proposition~\ref{identity}(b)) } \\ =&1.\\
\end{aligned}$$

 If $\t\simk\s$ and $\t\ne\s$  then
the coefficient of $v_\t$ in $(T_k^2-\delta T_k+\delta\varrho
E_k)v_\s$ is zero since {\small \begin{align*} &\sum_{\v\simk \s}
T_{\s\v}(k)T_{\v\t}(k)-\delta T_{\s\t}(k)+\delta\varrho E_{\s\t}(k)\\
   =&\sum_{\substack{\v\sim \s\\\s\ne\v\ne\t}}
      \frac{\delta^2E_{\s\v}(k) E_{\v\t}(k)}
           {(c_\s(k)c_\v(k)-1)(c_\v(k)c_{\t}(k)-1)}
      +\frac{\delta(E_{\s\s}(k)-1)}{c_{\s}(k)^2-1}\frac{\delta
      E_{\s\t}(k)}{c_{\s}(k)c_{\t}(k)-1}\\
    &+\frac{\delta E_{\s\t}(k)}{c_{\s}(k)c_{\t}(k)-1}\frac{\delta(E_{\t\t}(k)-1)}{c_{\t}(k)^2-1}
    -\frac{\delta^2 E_{\s\t}(k)}{c_{\s}(k)c_{\t}(k)-1}+\delta\varrho
    E_{\s\t}(k)\\
   =& \delta^2E_{\s\t}(k)\Big(\sum_{\substack{\v\sim \s}}
      \frac{E_{\v\v}(k)}
           {(c_\s(k)c_\v(k)-1)(c_\v(k)c_{\t}(k)-1)}-\frac{1}{(c_{\s}(k)^2-1)(c_{\s}(k)c_{\t}(k)-1)}\\
    &-\frac{1}{(c_{\t}(k)^2-1)(c_{\s}(k)c_{\t}(k)-1)}
    -\frac{1}{c_{\s}(k)c_{\t}(k)-1}+\delta^{-1}\varrho
        \Big)\\
   =& \delta^2E_{\s\t}(k)\Big(\sum_{\substack{\v\sim \s}}
      \frac{E_{\v\v}(k)}
           {(c_\s(k)c_\v(k)-1)(c_\v(k)c_{\t}(k)-1)}-\frac{c_{\s}(k)c_{\t}(k)+1}{(c_{\s}(k)^2-1)(c_{\t}(k)^2-1)}
           +\delta^{-1}\varrho\Big)\\
   =&0 \quad (\text{by Proposition~\ref{identity}(c)}).
\end{align*}}
Therefore, $(T_k^2-\delta T_k+\delta\varrho E_k)v_\s=v_\s$.
\end{proof}

\begin{Prop}\label{see} Suppose that $\s\in \UPD_n(\lambda)$. Then
\begin{enumerate}
 \item $E_{k+1}E_{k} v_{\s}=T_kT_{k+1} E_kv_{\s}$.
 \item$E_{k+1}E_kv_{\s} =E_{k+1} T_{k}T_{k+1}v_{\s}$.
\end{enumerate}
\end{Prop}

\begin{proof}
(a) We assume that $\s_{k-1}=\s_{k+1}$ since otherwise $E_{k+1}E_{k}
v_{\s}=T_kT_{k+1} E_kv_{\s}=0$. Let $\ts\in\UPD_n(\lambda)$ be  such
that $\ts\simk\s$ and $\ts_k=\s_{k+2}$. Note that $E_k v_\s=0$ for
any $\s\in \UPD_n(\lambda)$ with  $\s_{k-1}\neq \s_{k+1}$. So,
\begin{align*}
&T_kE_{k+1}E_k v_{\s}-\delta E_{k+1}E_kv_{\s}+\delta E_k v_{\s}\\
    =&(T_k-\delta)E_{k+1}E_{\s,\ts}(k)v_{\ts}+\delta
            \sum_{\t\simk\s}E_{\s\t}(k) v_{\t}\\
    =&(T_k-\delta)E_{\s,\ts}(k)\sum_{\t\simkk\ts}E_{\ts\t}(k+1) v_{\t}
            +\delta\sum_{\t\simk\s}E_{\s\t}(k) v_{\t}\\
    =&E_{\s,\ts}(k)E_{\ts,\ts}(k+1)\sum_{\t\simk\ts}T_{\ts\t}(k) v_{\t}
                 +\sum_{\t\simkk\ts,\t\ne\ts}E_{\s,\ts}(k)E_{\ts,\t}(k+1)(a_{\t}(k)
               v_{\t} + b_{\t}(k) v_{s_k\t})\\
    &-\delta E_{\s,\ts}(k)\sum_{\t\simkk\ts}E_{\ts\t}(k+1) v_{\t}
                  +\delta\sum_{\t\simk\s}E_{\s\t}(k) v_{\t}.
\end{align*}
If $s_k\t$ is defined, for $\t$ in the second sum, then
$(s_k\t)_k\ne \s_{k+2}$ and $\u=s_{k+1}s_k\t$ is also defined.
Further, we have $\u\simk\ts$ and $\u\ne\s$. Similarly,
$$
T_{k+1}E_kv_{\s}= E_{\s\ts}(k)\sum_{\t\simkk\ts}
       T_{\ts\t}(k+1)v_{\t}+\sum_{\t\simk\ts, \t\ne\ts} E_{\s\t}(k)
      \Big(a_\t(k+1)v_{\t}+b_\t(k+1) v_{s_{k+1} \t}\Big)
$$
We are going to  compare the coefficients of $v_{\t}$ in both
$T_kE_{k+1}E_k v_{\s}-\delta E_{k+1}E_kv_{\s}+\delta E_k v_{\s}$ and
$T_{k+1} E_k v_\s$. Note that $E_{\ts\ts}(k)E_{\ts\ts}(k+1)=1$ by
Lemma~\ref{ees}.

\Case{1. $\t=\ts$} Since $c_{\ts}(k)c_{\ts}(k+1)=1$, the definitions
and the remarks above show that the coefficient of $v_\t$ in
$T_kE_{k+1}E_k v_{\s}-\delta E_{k+1}E_kv_{\s}+\delta E_k v_{\s}$  is
equal to
\begin{align*}
 &E_{\s\ts}(k) E_{\ts\ts}(k+1)(T_{\ts\ts}(k)-\delta)+\delta
     E_{\s,\ts}(k)\\
 &=E_{\s\ts}(k)\Big(\delta
   E_{\ts\ts}(k+1)\frac{E_{\ts\ts}(k)-1}{c_{\ts}(k)^2-1}-\delta
   E_{\ts\ts}(k+1)+\delta \Big)\\
 &=E_{\s\ts}(k)\frac{\delta(E_{\ts\ts}(k+1)-1)}{c_{\ts}(k+1)^2-1}\\
 &=E_{\s\ts}(k)T_{\ts\ts}(k+1)
\end{align*}
which is the coefficient of $v_\t$ in $T_{k+1} E_kv_{\s}$.

\Case{2.  $\t\simk\ts$ and $\t\ne\ts$} Now, $c_{\ts}(k)=c_\t(k+2)$
and $c_\t(k+1)=c_\t(k)^{-1}$, so the coefficient of $v_\t$ in
$T_kE_{k+1}E_k v_{\s}-\delta E_{k+1}E_kv_{\s}+\delta E_k v_{\s}$ is
\begin{align*}
&E_{\s\ts}(k) E_{\ts\ts}(k+1)T_{\ts\t}(k)+\delta
     E_{\s\t}(k)\\
&=E_{\s\ts}(k) E_{\ts\ts}(k+1)\frac{\delta
E_{\ts\t}(k)}{c_{\ts}(k)c_\t(k)-1}+\delta
     E_{\s\t}(k)\\
& =\delta E_{\s\t}(k)\frac{1}{1-c_{\ts}(k)^{-1}c_\t(k)^{-1}}\\
&=\delta E_{\s\t}(k)\frac{c_\t(k+2)}{c_\t(k+2)-c_\t(k+1)}\\
&=E_{\s\t}(k)a_\t(k+1).
\end{align*}
which is the coefficient of $v_\t$ in $T_{k+1} E_kv_{\s}$.

\Case{3. $\t\simkk\ts$ and $\t\ne\ts$} Since
$c_\t(k)c_{\ts}(k+1)=1$, the coefficient of $v_\t$ in $T_kE_{k+1}E_k
v_{\s}-\delta E_{k+1}E_kv_{\s}+\delta E_k v_{\s}$
 is
\begin{align*}
&(a_{\t}(k)-\delta)E_{\ts\t}(k+1)E_{\s\ts}(k)\\
=&(\frac{\delta
c_\t(k+1)}{c_\t(k+1)-c_\t(k)}-\delta)E_{\ts\t}(k+1)E_{\s\ts}(k)\\
=&E_{\s\ts}(k)\frac{\delta E_{\ts\t}(k+1)}{c_\t(k+1)c_{\ts}(k+1)-1}\\
=&E_{\s\ts}(k)T_{\ts\t}(k+1)
\end{align*}
which is the coefficient of $v_\t$ in $T_{k+1} E_kv_{\s}$.

Now suppose that $s_k\t$ is defined and let $\u=s_{k+1}s_k\t$ be as
above. Then the coefficient of $v_{s_k\t}$ in $T_kE_{k+1}E_k
v_{\s}-\delta E_{k+1}E_kv_{\s}+\delta E_k v_{\s}$ is
\begin{align*}
E_{\s\ts}(k)E_{\ts\t}(k+1)b_\t(k)
  &=\sqrt{E_{\s\s}(k)}\sqrt{E_{\t\t}(k+1)}b_\t(k)\\
  &=\sqrt{E_{\s\s}(k)}\sqrt{E_{\u\u}(k)}b_\u(k+1)\\
  &=E_{\s\u}(k)b_\u(k+1),
\end{align*}
where the second equality comes from (\ref{be})(f). As
$s_k\t=s_{k+1}\u$ this is the coefficient of $v_{s_k\t}$ in $T_{k+1}
E_kv_{\s}$.

In summary, we have proved that $(T_kE_{k+1}E_k -\delta
E_{k+1}E_k+\delta E_k) v_{\s}= T_{k+1} E_kv_{\s}$. By
Lemma~\ref{s-square} and Lemma~\ref{start relations}(j),
$$\begin{aligned} E_{k+1}E_k v_\s& =(T_k^2-\delta T_k+\delta\varrho
E_k) E_{k+1}E_k v_\s\\ &= T_k(T_kE_{k+1}E_k v_{\s}-\delta
E_{k+1}E_kv_{\s}+\delta E_k v_{\s})\\
&=T_k( T_{k+1} E_kv_{\s})=T_kT_{k+1}E_k v_\s, \\
\end{aligned}$$ and (a) follows.

 In order to prove (b), we need to consider four cases as
follows.

\Case{1. $\s_k=\s_{k+2}$ and $\s_{k-1}=\s_{k+1}$} We have
$$\begin{aligned}& E_{k+1}E_{k} T_{k+1}
v_{\s}-\delta E_{k+1}E_{k} v_{\s}+\delta E_{k+1} v_{\s}\\
=&E_{k+1}E_k T_{\s\s}(k+1)v_{\s}-\delta E_{k+1}E_{k} v_{\s}+\delta E_{k+1} v_{\s}\\
=&(T_{\s\s}(k+1)-\delta)E_{k+1}E_{\s\s}(k)v_{\s}+\delta E_{k+1} v_{\s}\\
=&((T_{\s\s}(k+1)-\delta)E_{\s\s}(k)+\delta) E_{k+1} v_{\s}\\
=&\frac{\delta c_\s(k+1)^2}{c_\s(k+1)^2-1}(1-E_{\s\s}(k)) E_{k+1} v_{\s}\\
=&\frac{\delta(E_{\s\s}(k)-1)}{c_\s(k)^2-1} E_{k+1} v_{\s}\\
=&T_{\s\s}(k) E_{k+1}v_{\s}=E_{k+1}T_{k} v_{\s}.
\end{aligned}
$$

\Case{2. $\s_k\neq \s_{k+2}$ and $\s_{k-1}=\s_{k+1}$} Define
$\ts\in\UPD_n(\lambda)$ to be the unique updown tableau such that
$\ts\simk\s$ and $\ts_k=\s_{k+2}$. Then $\ts\ne\s$ and
\begin{align*}
&E_{k+1}E_{k} T_{k+1}
v_{\s}-\delta E_{k+1}E_{k} v_{\s}+\delta E_{k+1} v_{\s}\\
=&E_{k+1}E_{k}( a_\s(k+1)v_{\s}+b_\s(k+1)v_{s_{k+1}{\s}})-\delta
E_{k+1}E_{k} v_{\s}\\
=&(a_\s(k+1)-\delta)E_{k+1}E_{\s\ts}(k)v_{\ts}\\
=&\frac{\delta E_{\s\ts}(k)c_\s(k+1)}{c_\s(k+2)-c_\s(k+1)}E_{k+1}v_{\ts}\\
=&\frac{\delta E_{\s\ts}(k)}{c_\s(k)c_{\ts}(k)-1}E_{k+1}v_{\ts}\\
=&T_{\s\ts}(k)E_{k+1}v_{\ts}=E_{k+1}T_kv_{\ts}
\end{align*}
where the last second equality uses the facts that
$c_\s(k+1)c_\s(k)=1$, $c_\s(k+2)=c_{\ts}(k)$ and
$(s_{k+1}\s)_{k-1}\ne(s_{k+1}\s)_{k+1}$.

\Case{3.  $\s_k=\s_{k+2}$ and $\s_{k-1}\neq\s_{k+1}$} Let
$\ts\in\UPD_n(\lambda)$ such that $\ts\simkk\s$ and
$\ts_{k+1}=\s_{k-1}$. Then
\begin{align*}
   &E_{k+1}E_{k} T_{k+1}
v_{\s}-\delta E_{k+1}E_{k} v_{\s}+\delta E_{k+1} v_{\s}\\
=&E_{k+1}E_{k} T_{\s\ts}(k+1)v_{\ts}+\delta E_{k+1} v_{\s}\\
=&T_{\s\ts}(k+1)E_{\ts\ts}(k)E_{k+1}v_{\ts}+\delta E_{k+1} v_{\s}\\
=&T_{\s\ts}(k+1)E_{\ts\ts}(k)\sum_{\t\simkk\ts}
E_{\ts\t}(k+1)v_\t+\delta\sum_{\t\simkk\s} E_{\s\t}(k+1)v_\t
\end{align*}
and
$E_{k+1}T_kv_\s=E_{k+1}(a_\s(k)v_\s+b_\s(k)v_{s_k\s})=a_\s(k)\sum_{\t\simkk\s}
E_{\s\t}(k+1)v_\t$ since $(s_k\s)_{k}\neq(s_k\s)_{k+2}$ However,
since $c_{\ts}(k+1)=c_\s(k)^{-1}$, we have
\begin{align*}
& \quad T_{\s\ts}(k+1)E_{\ts\ts}(k)E_{\ts\t}(k+1)+\delta E_{\s\t}(k+1)\\
&=\frac{\delta\sqrt{E_{\s\s}(k+1)}\sqrt{E_{\t\t}(k+1)}}{c_\s(k+1)c_{{\tilde\s}}(k+1)-1}+\delta
E_{\s\t}(k+1)\\
&=\frac{\delta
c_\s(k+1)}{c_\s(k+1)-c_\s(k)}E_{\s\t}(k+1)=a_\s(k)E_{\s\t}(k+1).
\end{align*}
So, $E_{k+1}E_{k} T_{k+1} v_{\s}-\delta E_{k+1}E_{k} v_{\s}+\delta
E_{k+1} v_{\s}=E_{k+1}T_kv_\s$.

\Case{4.  $\s_k\neq \s_{k+2}$ and $\s_{k-1}\neq \s_{k+1}$} Under our
assumptions,  we have $E_{k+1}E_{k} T_{k+1} v_{\s}-\delta
E_{k+1}E_{k} v_{\s}+\delta E_{k+1}
v_{\s}=b_\s(k+1)E_{k+1}E_kv_{s_{k+1}\s}$ and
$E_{k+1}T_kv_\s=b_\s(k)E_{k+1}v_{s_k\s}$. If
$(s_{k+1}\s)_{k-1}\ne(s_{k+1}\s)_{k+1}$ then
$(s_{k}\s)_{k}\ne(s_{k}\s)_{k+2}$. So, $E_{k+1}E_{k} T_{k+1}
v_{\s}-\delta E_{k+1}E_{k} v_{\s}+\delta E_{k+1}
v_{\s}=0=E_{k+1}T_{k}v_{\s}$.

Suppose now that $(s_{k+1}\s)_{k-1}=(s_{k+1}\s)_{k+1}$ and let
$\ts\in\UPD_n(\lambda)$ be the unique updown tableau such that
$\ts\simk s_{k+1}\s$ and $\ts_{k}=\s_{k+2}$. Set $\t=s_k\s$ and
$\u=s_{k+1}\s$ and observe that the assumptions of (\ref{be})(f)
hold, so that
$b_\t(k)\sqrt{E_{\t\t}(k+1)}=b_\u(k+1)\sqrt{E_{\u\u}(k)}$. As
$b_\s(k)=b_\t(k)$ and $b_\s(k+1)=b_\u(k+1)$. By   (\ref{be})(d),
together with the fact that $\t'\simkk\ts$ if and only if $\t'\simkk
s_k\s$, we have
\begin{align*}
& E_{k+1}E_{k} T_{k+1} v_{\s}-\delta E_{k+1}E_{k} v_{\s}+\delta
E_{k+1} v_{\s}\\
  =&b_{\s}(k+1)E_{k+1}\sum_{\t'\simk\ts}
      E_{\u\t'}(k)v_{\t'}\\
  =&b_{\s}(k+1)E_{\u\ts}(k)E_{k+1}v_{\ts}\\
  =&b_{\s}(k+1)E_{\u\ts}(k)\sum_{\t'\simkk\ts}E_{\ts\t'}(k+1)v_{\t'}\\
  =&b_\s(k)\sum_{\t'\simkk\ts}E_{\t\t'}(k+1)v_{\t'}\\
  =&b_\s(k)E_{k+1}v_{s_k\s}=E_{k+1}T_kv_\s
\end{align*}
In summary, we have proved that $E_{k+1}E_{k} T_{k+1} v_{\s}-\delta
E_{k+1}E_{k} v_{\s}+\delta E_{k+1} v_{\s}=E_{k+1}T_kv_\s$ for any
$\s\in \UPD_n(\lambda)$. So, $E_{k+1}T_k (T_{k+1}v_\s)=
(E_{k+1}E_{k} T_{k+1} -\delta E_{k+1}E_{k} +\delta E_{k+1}) (T_{k+1}
v_{\s})$. Now, (b) follows from  Lemma~\ref{s-square} and
Lemma~\ref{start relations}(i)(j), immediately.
\end{proof}

\begin{Lemma}\label{sts} Suppose that $\s\in \UPD_n(\lambda)$ with
$\s_{k-1}\neq \s_{k+1}$ and $\s_{k}\neq \s_{k+2}$, where $1\le
k<n-1$. Then $T_kT_{k+1}T_kv_{\s}=T_{k+1}T_kT_{k+1} v_{\s}$.
\end{Lemma}

\begin{proof} One can verify the result   without difficult if he
uses the arguments in the proof of Lemma~\cite[4.28]{AMR}. We only
give an example to illustrate it and leave the others to the reader.

Suppose that either $s_k\s$ is not defined, or $s_k\s$ is defined
and $(s_k\s)_{k}\neq (s_k\s)_{k+2}$. In this case, the formulae for
$T_kT_{k+1}T_k v_{\s}$ and $T_{k+1}T_kT_{k+1} v_{\s}$ are exactly
the same as those given in the proof of \cite[4.28]{AMR} up to the
definitions of $a_\t(k), b_t(k) $ etc. We can verify $T_kT_{k+1}T_k
v_{\s}=T_{k+1}T_kT_{k+1} v_{\s}$ by comparing the coefficients of
$v_\u$ on both sides of the above equality. For example, we need to
show \begin{equation}\begin{aligned}\label{vs} &
a_{\s}(k)^2a_{\s}(k+1)+a_{s_k\s}(k+1)(1-a_{\s}(k)^2+\delta
a_{\s}(k))\\
=&a_{\s}(k)a_{\s}(k+1)^2+a_{s_{k+1}\s}(k)(1-a_{\s}(k+1)^2+\delta
a_{\s}(k+1))\\
\end{aligned}
\end{equation}
when we prove that the coefficients of $v_\s$ in  $T_kT_{k+1}T_k
v_{\s}=T_{k+1}T_kT_{k+1} v_{\s}$  are equal. The reader should
compare the (\ref{vs}) with that given in line 4,  in
\cite[p93]{AMR}. In our case, $a_\s(k)=\delta
c(\beta)(c(\beta)-c(\alpha))^{-1}$, $a_\s(k+1)=\delta
c(\gamma)(c(\gamma)-c(\beta))^{-1}$,
$a_{s_{k+1}\s}(k)=a_{s_{k}\s}(k+1)=\delta
c(\gamma)(c(\gamma)-c(\alpha))^{-1}$ if we write
 $\s_{k}\ominus\s_{k-1}=\alpha,
\s_{k+1}\ominus\s_{k}=\beta,\s_{k+2}\ominus\s_{k+1}=\gamma$. By
direct computation, we can verify (\ref{vs}) easily.
\end{proof}

\begin{Lemma}\label{sts rel2} Suppose that $\s\in \UPD_{n}(\lambda)$
and that either $\s_{k-1}=\s_{k+1}$ and $\s_k\neq \s_{k+2}$, or
$\s_{k-1}\neq \s_{k+1}$ and $\s_k= \s_{k+2}$, for $1\le k<n-1$. Then
$T_kT_{k+1}T_k v_{\s}=T_{k+1}T_kT_{k+1} v_{\s}$.
\end{Lemma}

\begin{proof} The result can be proved by arguments given in the proof of
\cite[4.29]{AMR}.  Since it does not involve huge computation, we
include a proof here.

 \Case{1. $s_{k+1}\s$ is defined} Suppose first
that $\s_{k-1}=\s_{k+1}$ and $\s_k\neq \s_{k+2}$. Then
$\t=s_{k+1}\s\in \UPD(\lambda)$ is well-defined. Furthermore,
$\t_k\neq \t_{k+2}$ and $\t_{k-1}\neq \t_{k+1}$, so $T_kT_{k+1}T_k
v_{\t}=T_{k+1}T_kT_{k+1} v_{\t}$ by Lemma~\ref{sts}. Now,
$T_{k+1}v_\t=a_\t(k+1)v_\t+b_\t(k+1)v_\s$ and $b_\t(k+1)\ne0$.
Therefore $$\begin{aligned} T_kT_{k+1}T_kv_\s
   &=\frac1{b_\t(k+1)}T_kT_{k+1}T_k\Big(T_{k+1}v_\t-a_\t(k+1)v_\t\Big)\\
   &=\frac1{b_\t(k+1)}\Big(T_k(T_{k+1}T_kT_{k+1})v_\t-a_\t(k+1)(T_kT_{k+1}T_k)v_\t\Big)\\
   &=\frac1{b_\t(k+1)}\Big(T_k(T_kT_{k+1}T_k)v_\t
                  -a_\t(k+1)(T_{k+1}T_kT_{k+1})v_\t\Big)\\
                  \end{aligned}
                  $$
by Lemma~\ref{sts}. Hence, using Lemma~\ref{s-square} twice, {\small
$$\begin{aligned} T_kT_{k+1}T_kv_\s
   &=\frac1{b_\t(k+1)}\Big((1+\delta T_k-\varrho\delta E_k)T_{k+1}T_kv_\t-a_\t(k+1)(T_{k+1}T_kT_{k+1})v_\t\Big)\\
      &=\frac1{b_\t(k+1)}\Big(T_{k+1}T_k(1+\delta T_{k+1}-\varrho\delta E_{k+1})v_\t
            -a_\t(k+1)(T_{k+1}T_kT_{k+1})v_\t\Big)\\
   &=\frac1{b_\t(k+1)}(T_{k+1}T_kT_{k+1})\Big(T_{k+1}v_\t
                  -a_\t(k+1)v_\t\Big)\\
   &=(T_{k+1}T_kT_{k+1})v_\s
\end{aligned}$$}
as required.

The case when $\s_{k-1}\neq \s_{k+1}$ and $\s_k= \s_{k+2}$ can be
proved similarly.

\Case{2. $s_{k+1}\s$ is not defined} This is  equivalent to saying
that the two nodes $\s_{k+2}\ominus\s_{k+1}$ and $\s_{k+1}\ominus
\s_k$ are either in the same row or in the same column. Therefore,
either $\s_k\subset \s_{k+1}\subset \s_{k+2}$ or $\s_k\supset
\s_{k+1}\supset \s_{k+2}$. Note that in either case
$\s_{k-1}=\s_{k+1}$, so we have
$$ E_k v_\s=\sum_{\substack{\t\simk\s\\\t\neq \s}}
        E_{\s\t}(k) v_\t +E_{\s\s}(k) v_\s.$$
Using  Lemma~\ref{see} and Lemma~\ref{start relations}(j) twice, we
have $T_{k}T_{k+1}T_kE_k v_\s=\varrho T_kT_{k+1}E_kv_\s=\varrho
E_{k+1}E_k v_\s=T_{k+1}E_{k+1}E_kv_\s=T_{k+1}T_{k}T_{k+1}E_kv_\s$.

Suppose that $\t\simk\s$ and $\t\ne\s$. Since  the two boxes
$\s_{k+2}\ominus\s_{k+1}$ and $\s_{k+1}\ominus \t_k$ belong to
different rows and columns,  $s_{k+1}\t$ is well-defined and
$\t_{k-1}=\t_{k+1}$. By Case~1, $T_{k+1}T_{k}T_{k+1}
v_\t=T_kT_{k+1}T_{k} v_\t$. Consequently,
$T_{k+1}T_{k}T_{k+1}E_{\s\s}(k)v_\s=T_k
T_{k+1}T_{k}E_{\s\s}(k)v_\s$. Canceling the non-zero factor
$E_{\s\s}(k)$ shows that $T_{k}T_{k+1}T_k v_\s=T_{k+1}T_kT_{k+1}
v_\s$.
\end{proof}

\begin{Prop}\label{end relations} Suppose that $1\le k<n-1$
    and $\s\in \UPD_{n}(\lambda)$. Then
    $T_kT_{k+1}T_k v_{\s}=T_{k+1}T_kT_{k+1} v_{\s}$.
\end{Prop}

\begin{proof} By Lemma~\ref{sts} and Lemma~\ref{sts rel2}, we need to
 consider the case when $\s_{k-1}=\s_{k+1}$ and $\s_k= \s_{k+2}$.
By Proposition~\ref{see}(a) and Proposition~\ref{start
relations}(j),
$$T_{k+1}T_kT_{k+1}E_kv_\s= T_{k+1}E_{k+1}E_kv_\s= \varrho E_{k+1}E_kv_\s
=\varrho T_kT_{k+1}E_kv_\s = T_kT_{k+1}T_kE_kv_\s
$$
Therefore,
$$
\(T_{k+1}T_kT_{k+1}-T_kT_{k+1}T_k\)
   \Big(E_{\s\s}(k)v_\s+\sum_{\t\simk\s, \t\ne\s} E_{\s\t}(k)v_\t\Big)=0.$$
Now, if $\t\simk \s$ and $\t\neq \s$ then $T_kT_{k+1}T_k
v_{\t}=T_{k+1}T_k T_{k+1}v_{\t}$ by Lemma~\ref{sts rel2}.
Consequently, $T_kT_{k+1}T_k v_{\s}=T_{k+1}T_k T_{k+1}v_{\s}$ since
$E_{\s\s}(k)\ne0$. This completes the proof.
\end{proof}

\begin{proof}[Proof of Theorem~\ref{seminormal}]
  We have already checked the defining relations for $\W_{r, n}$ on $\Delta(\lambda)$. So,
   $\Delta(\lambda)$ is a
    $\W_{r,n}(\bu)$--module, as we wanted to show.
\end{proof}

The following result shows that we can chose $u_i, q\in \R$ such
that the root conditions can be satisfied in $\R$.

\begin{Lemma}\label{be real}
Suppose that $R=\R$.  We choose $q,u_i\in R^+$  in such a way that
$|log_{q^2}u_1|>\cdots>|log_{q^2}u_r|\ge n$ and
$|log_{q^2}u_i|-|log_{q^2}u_{i+1}|\ge 2n$, and one of the following
conditions holds:
\begin{enumerate}
\item $q>1$  if either $2\nmid r$  and $
\varrho^{-1}=\prod\limits_{l=1}^ru_l$ or $2\mid r$  and $
\varrho^{-1}=q^{-1}\prod\limits_{l=1}^ru_l $. Furthermore, we assume
that $ log_{q^2}u_i<0$   if $2\mid i$ and $log_{q^2}u_i>0$  if
$2\nmid i$.
\item  $ 0<q<1$
if either  $2\nmid r$ and $\varrho^{-1}=-\prod\limits_{l=1}^ru_l$ or
$2\mid r$  and $ \varrho^{-1}=-q\prod\limits_{l=1}^ru_l$.
Furthermore, we assume that $log_{q^2}u_i>0$ if $2\mid i$ and
$log_{q^2}u_i<0$  if $2\nmid i$.\end{enumerate}

\noindent%
Suppose that $\s\in\UPD_n(\lambda)$ and $1\le k<n$. Then
$1-a_{\s}(k)^2+\delta a_{\s}(k)\geq 0$, if $\s_{k-1}\ne\s_{k+1}$,
and $E_{\s\s}(k)>0$, if $\s_{k-1}=\s_{k+1}$. In particular, the Root
Condition (\ref{be}) holds if we choose positive square roots
$\sqrt{b_{\s}(k)}\ge0$ and $\sqrt{E_{\s\s}(k)}>0$.
\end{Lemma}

\begin{proof}

We start with the case  $\s_{k-1}\ne\s_{k+1}$. Let
$\alpha=\s_k\ominus\s_{k-1}$ and $\beta=\s_{k+1}\ominus\s_k$. Define
$S=\{a\in \mathbb R^+||log_{q^2}a|\geq 1\}$. By the definitions of
$c(\alpha)$ and $c(\beta)$,
${c(\beta)}{c(\alpha)^{-1}}=u_i^{\pm1}u_j^{\pm1}q^{2(\pm k\pm l)}$
for some integers $i, j, k$ and $ l$. We want to prove $c(\beta)
c(\alpha)^{-1}\in S$. There are two cases we need to discuss:

\Case {1.  $u_i^{\pm1}u_j^{\pm1}=1$} In this case,   $\alpha$ and
$\beta$ are in the same component of $\lambda$. Also, both $\alpha$
and $\beta$ are  either  removable nodes  or addable nodes of
$\lambda$. By Lemma~\ref{generic u} ${c(\beta)}{c(\alpha)^{-1}}\neq
1$. Therefore, ${c(\beta)}{c(\alpha)^{-1}}=q^{2(\pm k\pm l)}\in S$.

\Case {2:  $u_i^{\pm1}u_j^{\pm1}\neq 1$} We have
$$\begin{aligned} &|log_{q^2}(u_i^{\pm1}u_j^{\pm1}q^{2(\pm k\pm
l)})|=|\pm
log_{q^2}u_i \pm log_{q^2}u_j \pm k\pm l|\\
\geq & | log_{q^2}u_i \pm log_{q^2}u_j|- |k\pm l|\geq 2n - |k\pm
l|\geq 1.\\
\end{aligned} $$
Hence, ${c(\beta)}{c(\alpha)^{-1}}\in S$. So, $$\begin{aligned}
1-a_{\s}(k)^2+\delta
a_{\s}(k)&=\frac{(c(\beta)-q^{-2}c(\alpha))(c(\beta)-q^{2}c(\alpha))}{(c(\beta)-c(\alpha))^2}\\
         &=\frac{(\frac{c(\beta)}{c(\alpha)}-q^{-2})(\frac{c(\beta)}{c(\alpha)}-q^{2})}{(\frac{c(\beta)}{c(\alpha)}-1)^2}>0
\end{aligned}$$

Now, we prove  $E_{\s\s}(k)>0$.  Since we are assuming that
$|log_{q^2}u_i|-|log_{q^2}u_{i+1}|\ge 2n$,  $|log_{q^2}u_t\pm
log_{q^2}u_{t'}|\ge 2n$ if $t'\neq t$. Therefore, the signs of
$log_{q^2}u_t^{\pm1}u_{t'}^{\pm1}q^{2(\pm c\pm d)}$ and
 $\pm log_{q^2}u_t\pm log_{q^2}u_{t'}$ are the same. In other words,
 \begin{equation}\label{signs} \frac{u_t^{\pm1}u_{t'}^{\pm1}q^{2(\pm c\pm d)
}-1}{u_t^{\pm1}u_{t'}^{\pm1}-1}>0.\end{equation}

Similarly, we can verify  \begin{equation}\label{signs1}
\frac{u_t^{\pm 2}q^{2(\pm c\pm d) }-1}{u_t^{\pm
2}-1}>0.\end{equation}

Next we consider the case $\s_{k-1}=\s_{k+1}$. Let
$\alpha=\s_k\ominus\s_{k-1}$ and $\lambda=\s_{k-1}$. Write
$\alpha=(i,j,t)$.

Let $u_tq^{2c_i}$, for $1\le i\le l+1$, be the contents of the
addable nodes of $\lambda^{(t)}$ and let $u_t^{-1}q^{-2d_j}$, for
$1\le j\le l$, be the contents of the removable nodes of
$\lambda^{(t)}$. We may assume that
$$c_1>d_1>\cdots>c_l>d_l>c_{l+1}.$$
Let $\epsilon_t$ be the sign of the product of
$\frac{c(\alpha)c(\beta)-1}{c(\alpha)-c(\beta)}$, where $\beta$ runs
over all of the addable and removable nodes of $\lambda^{(t)}$ such
that $\beta\ne\alpha$. First we consider  $\epsilon_{t'}$ where
$t'\neq t$. By (\ref{signs}),  $\epsilon_{t'}$ is equal to either
the sign of
$$
\frac{(u_t ^{-1}u_{t'}^{-1}-1)^l}{(u_t^{-1}-u_{t'}^{-1})^l}
\frac{(u_t^{-1}u_{t'}-1)^{l+1}}{(u_t^{-1}-u_{t'})^{l+1}}
=\frac{u_{t'}-u_{t}}{1-u_tu_{t'}}.
$$
or the sign of
$$
\frac{(u_t u_{t'}^{-1}-1)^l}{(u_t-u_{t'}^{-1})^l}
\frac{(u_tu_{t'}-1)^{l+1}}{(u_t-u_{t'})^{l+1}}
=\frac{1-u_tu_{t'}}{u_{t'}-u_{t}}.
$$
It is not difficult to see that the signs of the last two equations
are the same. Suppose $t'<t$ and $q>1$. There are four cases we have
to discuss.
\begin{itemize}
\item Both $t'$ and  $t$ are odd. Then $u_{t'}>u_t$ and  $u_tu_{t'}>1$.
\item $t'$ is odd and $t$ is even. Then $u_{t'}>u_t$  and  $u_tu_{t'}>1$.
\item $t'$ is even and $t$ is odd. Then $u_{t'}<u_t$ and
$u_tu_{t'}<1$.
\item Both $t'$ and  $t$ are even. Then $u_{t'}<u_t$ and $u_tu_{t'}<1$.
\end{itemize}
 So, $\varepsilon_{t'}<0$ if $t'<t$. When $t'>t$, we switch the role
 between $t$ and $t'$. So, $\epsilon_{t'}>0$.
When $0<q<1$, we use $q^{-1}$ instead of the previous $q$. Since
$|log_{(q^{-1})^2}u_i|=|log_{q^2}u_i|$, we still have
$\varepsilon_{t'}<0$ if $t'<t$ and  $\epsilon_{t'}>0$ if $t'>t$.
Hence
$$\prod_{t'\ne t}\epsilon_{t'}=(-1)^{t-1}.$$

Suppose  $c(\alpha)=u_tq^{2c_i}$, for some $i$. $\epsilon_t$ is
equal to the sign of
$$\prod\limits_{k=1,\atop k\ne i}^{l+1}\frac{u_t^2q^{2(c_i+c_k)}-1}{u_t(q^{2c_i}-q^{2c_k})}
\prod_{k=1}^l\frac{q^{2(c_i-d_k)}-1}{u_tq^{2c_i}-u_t^{-1}q^{-2d_k}}.$$
By  (\ref{signs1}), it is   equal to the sign of
$$\prod\limits_{k=1,\atop k\ne i}^{l+1}\frac{1}{q^{(2c_i-2c_k)}-1}\prod_{k=1}^l(q^{2(c_i-d_k)}-1)$$
so $\epsilon_t=(-1)^{i-1}(-1)^{i-1}=1$ if $q>1$ and
$\epsilon_t=(-1)^{l-i+1}(-1)^{l-i+1}=1$ if $0<q<1$.

If $c(\alpha)=u_t^{-1}q^{-2d_j}$, for some $j$, then $\epsilon_t$ is
equal to the sign of
$$\prod\limits_{k=1,\atop k\ne j}^{l}\frac{u_t^{-2}q^{-2(d_j+d_k)}-1}{u_t^{-1}(q^{-2d_j}-q^{-2d_k})}
\prod_{k=1}^{l+1}\frac{q^{2(c_k-d_j)}-1}{u_t^{-1}q^{-2d_k}-u_tq^{2c_k}}.$$
By (\ref{signs1}), it  is equal to the sign of
$$\prod\limits_{k=1,\atop k\ne j}^{l}\frac{1}{q^{-2(d_j-d_k)}-1}\prod_{k=1}^{l+1}(q^{2(c_k-d_j)}-1).$$
So $\epsilon_t=(-1)^{l-j}(-1)^{l+1-j}=-1$ if $q>1$ and
$\epsilon_t=(-1)^{j-1}(-1)^{j}=-1$ if $0<q<1$. In summary, we have
proved
\begin{equation}\label{sign}\begin{cases}
     \prod_{1\le t'\le r}\epsilon_{t'}=(-1)^{t-1}, & \text{ if }
     c(\alpha)=u_tq^{2c_i} \text{ for some $i$ }, \\
     \prod_{1\le t'\le r}\epsilon_{t'}=(-1)^t, &\text{ if } c(\alpha)=u_t^{-1}q^{-2d_j}, \text{ for some
     $j$ }.
\end{cases}\end{equation}

We determine the  sign of $E_{\s\s}(k)$ as follows.

\Case {1. $q>1$, $r$ is odd and $\varrho^{-1}=\prod_{l=1}^r u_i$}
$$\begin{aligned}
E_{\s\s}(k)&= \frac{1} {\varrho c(\alpha)}
 (\frac{c(\alpha)-c(\alpha)^{-1}}{\delta} +1)
                 \prod_{\beta\neq
                 \alpha}\frac{c(\alpha)-c(\beta)^{-1}}{c(\alpha)-c(\beta)}\\
&=\frac{1} {\varrho \delta c(\alpha)^2}(c(\alpha)^2+\delta c(\alpha)
- 1)\frac{c(\alpha)}{ \prod_{l=1}^ru_l} \prod_{\beta\neq
                 \alpha}\frac{c(\alpha)c(\beta)-1}{c(\alpha)-c(\beta)}\\
&=\frac{1} {\delta c(\alpha)}(c(\alpha) - q^{-1})(c(\alpha) +
q)\prod_{\beta\neq
                 \alpha}\frac{c(\alpha)c(\beta)-1}{c(\alpha)-c(\beta)}
\end{aligned}$$
On the other hand, under our assumption, we have
\begin{itemize}\item $c(\alpha)<q^{-1}$  if either $2\mid t$ and
     $c(\alpha)=u_tq^{2c_i}$ or $2\nmid t$ and $c(\alpha)=u_t^{-1}q^{-2d_i}$ for some
     $i$, \item     $c(\alpha)>q^{-1}$ if either $2\nmid t$ and
     $c(\alpha)=u_tq^{2c_i}$ or $2\mid t$  and $c(\alpha)=u_t^{-1}q^{-2d_i}$  for some
     $i$.\end{itemize}
     Since we are assuming that $q>1$, $\delta>0$. By (\ref{sign}), $E_{\s\s}(k)>0$ as required.

 \Case{2. $0<q<1$, $r$ is odd and $\varrho^{-1}=-\prod_{l=1}^r
u_i$}
$$\begin{aligned}
E_{\s\s}(k)= \frac{1} {\varrho c(\alpha)}
 &(\frac{c(\alpha)-c(\alpha)^{-1}}{\delta} -1)
                 \prod_{\beta\neq
                 \alpha}\frac{c(\alpha)-c(\beta)^{-1}}{c(\alpha)-c(\beta)}\\
&=\frac{1} {\varrho \delta c(\alpha)^2}(c(\alpha)^2-\delta c(\alpha)
- 1)\frac{c(\alpha)}{ \prod_{l=1}^ru_l} \prod_{\beta\neq
                 \alpha}\frac{c(\alpha)c(\beta)-1}{c(\alpha)-c(\beta)}\\
&=\frac{-1} {\delta c(\alpha)}(c(\alpha) + q^{-1})(c(\alpha) -
q)\prod_{\beta\neq
                 \alpha}\frac{c(\alpha)c(\beta)-1}{c(\alpha)-c(\beta)}
\end{aligned}$$
On the other hand, under our assumption, we have
\begin{itemize}\item $c(\alpha)<q$  if either $2\mid t$ and
     $c(\alpha)=u_tq^{2c_i}$ or $2\nmid t$ and $c(\alpha)=u_t^{-1}q^{-2d_i}$ for some
     $i$, \item     $c(\alpha)>q$ if either $2\nmid t$ and
     $c(\alpha)=u_tq^{2c_i}$ or $2\mid t$  and $c(\alpha)=u_t^{-1}q^{-2d_i}$  for some
     $i$.\end{itemize}
     Since we are assuming that $0<q<1$, $\delta<0$. By (\ref{sign}), $E_{\s\s}(k)>0$ as required.

 \Case{3. $q>1$, $r$ is even and
$\varrho^{-1}=q^{-1} \prod_{l=1}^r u_i$}
$$\begin{aligned}
E_{\s\s}(k)= \frac{1} {\varrho \delta}
 &(1 - \frac{q^2}{c(\alpha)^2} )
                 \prod_{\beta\neq
                 \alpha}\frac{c(\alpha)-c(\beta)^{-1}}{c(\alpha)-c(\beta)}\\
&=\frac{1} {\varrho \delta c(\alpha)^2}(c(\alpha)^2 -
q^2)\frac{c(\alpha)}{ \prod_{l=1}^ru_l} \prod_{\beta\neq
                 \alpha}\frac{c(\alpha)c(\beta)-1}{c(\alpha)-c(\beta)}\\
&=\frac{1} {\delta q c(\alpha)}(c(\alpha) + q)(c(\alpha) -
q)\prod_{\beta\neq
                 \alpha}\frac{c(\alpha)c(\beta)-1}{c(\alpha)-c(\beta)}
\end{aligned}$$

On the other hand, under our assumption, we have
\begin{itemize}\item $c(\alpha)<q$  if either $2\mid t$ and
     $c(\alpha)=u_tq^{2c_i}$ or $2\nmid t$ and $c(\alpha)=u_t^{-1}q^{-2d_i}$ for some
     $i$, \item     $c(\alpha)>q$ if either $2\nmid t$ and
     $c(\alpha)=u_tq^{2c_i}$ or $2\mid t$  and $c(\alpha)=u_t^{-1}q^{-2d_i}$  for some
     $i$.\end{itemize}
     Since we are assuming that $q>1$, $\delta>0$. By (\ref{sign}), $E_{\s\s}(k)>0$ as required.

\Case {4. $0<q<1$, $r$ is even and $\varrho^{-1}=-q\prod_{l=1}^r
u_i$}

$$\begin{aligned}
E_{\s\s}(k)= \frac{1} {\varrho \delta}
 &(1 - \frac{1}{q^2c(\alpha)^2} )
                 \prod_{\beta\neq
                 \alpha}\frac{c(\alpha)-c(\beta)^{-1}}{c(\alpha)-c(\beta)}\\
&=\frac{1} {\varrho \delta c(\alpha)^2}(c(\alpha)^2 -
q^{-2})\frac{c(\alpha)}{ \prod_{l=1}^ru_l} \prod_{\beta\neq
                 \alpha}\frac{c(\alpha)c(\beta)-1}{c(\alpha)-c(\beta)}\\
&=\frac{-q} {\delta  c(\alpha)}(c(\alpha) + q^{-1})(c(\alpha) -
q^{-1})\prod_{\beta\neq
                 \alpha}\frac{c(\alpha)c(\beta)-1}{c(\alpha)-c(\beta)}
\end{aligned}$$

On the other hand, under our assumption, we have
\begin{itemize}\item $c(\alpha)<q^{-1}$  if either $2\mid t$ and
     $c(\alpha)=u_tq^{2c_i}$ or $2\nmid t$ and $c(\alpha)=u_t^{-1}q^{-2d_i}$ for some
     $i$, \item     $c(\alpha)>q^{-1}$ if either $2\nmid t$ and
     $c(\alpha)=u_tq^{2c_i}$ or $2\mid t$  and $c(\alpha)=u_t^{-1}q^{-2d_i}$  for some
     $i$.\end{itemize}
     Since we are assuming that $0<q<1$, $\delta<0$. By (\ref{sign}), $E_{\s\s}(k)>0$ as required.
\end{proof}

\section{A cellular basis of $\W_{r,n}(\bu)$ with odd $r$}

Throughout this section, unless otherwise stated,  we always keep
the following assumption:

\begin{Assumption}\label{assumption5} Let  $R$ be a commutative ring
containing invertible elements $q, q-q^{-1}, $ and $u_i$, $ 1\le
i\le r$. We also assume that  $\Omega\cup\{\varrho\}\subseteq R$ is
$\bu$--admissible.\end{Assumption}

The main purpose of this section is to construct a cellular basis
for $\W_{r,n}$. We remark that we  assume that $r$ is odd. In other
words, $r=2p+1$ for some non-negative integer $p$.

In \cite{AK}, Ariki and Koike have proved that Ariki-Koike algebra
$\H_{r, n}$ is free over $R$.  Given a non-negative integer
$f\le\floor{n2}$.  Then  $\H_{r, n-2f}$ can be identified  with the
subalgebra of $\H_{r,n}$. Since we have not proved that $\W_{r, n}$
is free over $R$, we could not say $\W_{r, n_1}$ is a subalgebra of
$\W_{r, n_2}$ if $n_1< n_2$. However, there is an algebraic
homomorphism from $\W_{r, n_1}$ to $\W_{r, n_2}$. So, we define
$\W_{r, n_1}'$ to  be the image of $\W_{r, n_1}$ in $\W_{r, n_2}$.

\begin{Prop}\label{epsilon f} Given a positive integer
 $n\ge 2$.  Let $\Ef_n=\W_{r,n}E_1\W_{r,n}$ be the
two-sided ideal of $\W_{r, n}$ generated by~$E_1$. Then there is a
unique $R$--algebra isomorphism $\proj_{
n}:\H_{r,n}\cong\W_{r,n}/\Ef_n$ such that
$$\proj_{n}(g_i)=T_i+\Ef_n \text{ and } \proj_{n}(y_j)=X_j+\Ef_n,$$
for $1\le i<n$ and $1\le j\le n$.
\end{Prop}

\begin{proof} Let $S=\{X_j+\Ef_n, T_i+\Ef_n\mid 1\le i\le
n-1, 1\le j\le n\}$. By Definition~\ref{Waff relations}, $S$
generates $\W_{r,n}/\Ef_n$. Therefore,  $\proj_{n}$ is an algebraic
epimorphism. We claim that $ \W_{r,n}/\Ef_n$ is free over $R$ with
rank $r^n n!$. In fact, we consider $ \W_{r,n}/\Ef_n$ over
$R_0:=\mathbb Z[\bu^{\pm 1} , q^{\pm 1}, \delta^{\pm 1} ]$ where
$\delta=q-q^{-1}$. Further, we assume that $\bu, q$ are
indeterminates. We have constructed the seminormal representations
for $\W_{r, n}$ with respect to all $\lambda\in \Lambda_r^+(n-2f)$,
$0\le f\le \floor{n2}$ under the conditions in Lemma~\ref{generic u}
and (\ref{be}). In particular, we have seminormal representations of
$\W_{r, n}$ over $\mathbb R$. As $\mathbb R$ is not finitely
generated over $\mathbb Q$, we can take $r+1$ algebraically
independent transcedental real numbers $v_i\in \mathbb R$ and
$\mathbf q$. We define $R_1=\mathbb Z[v_1, v_2 \dots, v_r, \mathbf
q^{\pm 1}, \delta^{\pm}]$ Also, we assume $\Omega\cup\{\varrho\}$ is
$\mathbf v$-admissible. By Lemma~\ref{be real},
 $\Delta(\lambda)$ are $
\W_{r,n}/\Ef_n$-modules for all $\lambda\in \Lambda_r^+(n)$ over the
field $\mathbb R$.
 By Wedderburn-Artin theorem for
semisimple finite dimensional algebra,
$$\dim_{\mathbb R} \W_{r, n}/\langle E_1\rangle \ge r^n n!.$$  So,
 the image of an  $\mathbb R$-basis  of $\H_{r, n}$ has to be $\mathbb R$-linear
independent, and hence $R_1$-linear independent. Therefore,
$\W_{r,n}/\Ef_n$ is free over $R_1$ with rank $r^n n!$.

Note that $R_1\cong R_0$ as rings. So,  $\W_{r, n}(\u)$ over $R_0$
is isomorphic to $\W_{r, n}(\mathbf v)$ over $R_1$ as $R_0$-modules.
The corresponding isomorphism sends $u_i$ (resp. $q$) to $v_i$
(resp. $\mathbf q$). So, $\W_{r,n}/\Ef_n$ is free over $R_0$ with
rank $r^n n!$. By base change, it is free over an arbitrary
commutative ring $R$. So, $\proj_{n}$ is an isomorphism.
\end{proof}

\begin{Defn}\label{Ef} Given a non-negative integer $f\le
\floor{n2}$.  Let $E^f=E_{n-1}E_{n-3}\cdots E_{n-2f+1}$ and let
$\W_{r,n}^f=\W_{r,n}E^f\W_{r,n}$. If $f=\floor{ n2}\in \mathbb Z$,
then we set  $\W_{r,n}^{f+1}=0$.
\end{Defn}
By Definition~\ref{Ef}, there is  a filtration of two-sided ideals
of  $\W_{r,n}$ as follows:
$$\W_{r,n}=\W_{r,n}^0\supset \W_{r,n}^1\supset\dotsi
     \supset \W_{r,n}^{\floor{n2}}\supset \W_{r,n}^{\floor{n2}+1}=0.$$

For $0\le f\le\floor{n2}$ let $\pi_{f,
n}\map{\W_{r,n}^f}\W_{r,n}^f/\W_{r,n}^{f+1}$ be the corresponding
projection map of $\W_{r,n}$--bimodules.

Since we are assuming that $r=2p+1$ for some non-negative integer
$p$,  we set $\N_r=\{-p\dots,-1, 0,1,\dots,p\}$ and define $\Nrf$ to
be the set of $n$--tuples $\kappa=(k_1,\dots,k_n)$ such that
$k_i\in\N_r$ and $k_i\ne0$ only for $i=n-2j+1$, $1\le j\le f$. Thus,
$X^\kappa=X_{n-1}^{k_{n-1}}X_{n-3}^{k_{n-3}}\dots
X_{n-2f+1}^{k_{n-2f+1}}$.

\begin{Lemma}\label{if} Suppose that
$\kappa\in \Nrf$ with  $0 \le f\le \floor{n2}$. Then $E^f
X^\kappa\Ef_{n-2f}'\subset B_{r,n}^{f+1}$ where $\Ef'_{ n-2f}$ is
the image of $\Ef_{n-2f}$ in $\W_{r, n}$.
\end{Lemma}

\begin{proof} The result follows since $E^f X^\kappa$ commutes with  $\Ef_{n-2f}'$.
\end{proof}

By Proposition~\ref{epsilon f} and Lemma~\ref{if}, there is a
well-defined $R$-module homomorphism $\sigma_f\map{\H_{r,
n-2f}}\W^f_{r,n}/\W_{r,n}^{f+1}$, for each non-negative integer $
f\le\floor{n2}$, such that
$$\sigma_f(h)=E^f\proj_{n-2f}(h)'+\W_{r,n}^{f+1}, \text{
for $h\in \H_{r, n-2f}$, }$$ where $\proj_{ n-2f}(h)'$ is the image
of $\proj_{ n-2f}(h)$ in $\W_{r, n}$.

We recall the following definition in \cite{AMR}. Suppose that $f$
is a non-negative integer with $f\le \floor{n2}$. Let $\Bcal_f$ be
the subgroup of $\Sym_n$ generated by $\{s_{n-1}\}\cup\{s_{n-2i+2}
         s_{n-2i+1}s_{n-2i+3} s_{n-2i+2}\mid 2\le i\le f\}$.
Let $\tau=((n-2f),(2^f)\)$ and define
$$\Dcal_{f,n}=\Set[80]d\in \Sym_n|
      $\t^{\tau} d=(\t_1, \t_2)$ is a row standard $\tau$-tableau
     and the first column of $\t_2$ is increasing from top to bottom|.$$
For any positive integers $i, j$, write
$$
s_{i,j}=\begin{cases}
            s_{i-1}s_{i-2}\cdots s_j,  & \text{if  $ i>j$,} \\
            1,                    & \text{if $ i=j$,}\\
              s_{i}s_{i+1}\cdots s_{j-1} & \text{if  $ i<j$.} \\
        \end{cases} \quad
E_{i,j}=\begin{cases}
            E_{i-1}E_{i-2}\cdots E_j,  & \text{if  $ i>j$,} \\
            1,                    & \text{if $ i=j$,}\\
              E_{i}E_{i+1}\cdots E_{j-1} & \text{if  $ i<j$.} \\
        \end{cases}
       $$
\begin{Lemma}\label{cosets}
Suppose that $0\le f\le\floor{n2}$. Then {\small
\begin{equation}\label{co1}\Dcal_{f,
n}=\Set[40]s_{n-2f+1,i_f}s_{n-2f+2,j_f}\cdots s_{n-1,i_1}s_{n,j_1}|
      $1\leq i_f<\cdots<i_1\leq n ;\atop 1\leq i_k<j_k\leq n-2k+2;1\leq k\leq f$
      |.\end{equation}}
\end{Lemma}
\begin{proof} It has been proved in \cite{AMR} that  $\Dcal_{f,
n}$ is a complete set of right coset representatives for $\mathfrak
S_{n-2f}\times \Bcal_f$ in $\mathfrak S_{n}$. So, $$\#
\Dcal_{f,n}=\frac{|\mathfrak S_n|}{|\mathfrak S_{n-2f}||\Bcal_f|}=
\frac{n!}{(n-2f)!f!2^f}.$$   Let $\Dcal_{f,n}'$ be  the set given in
the right hand of (\ref{co1}).  Then $\# \Dcal_{1,n}'=\binom
{n}{2}$. In general, since we are assuming that $ 1\leq i_f<j_f\leq
n-2f+2$, $i_f\in \{1, 2,\cdots, n-2f+1\}$.  There are $n+2-2f-i_f$
choices for $j_f$. In this case, by induction assumption, there are
$\#\Dcal_{f-1,  n-i_f}$ choices for the sequences $i_1, j_1, \dots,
i_{f-1}, j_{f-1}$ which satisfy the inequalities in (\ref{co1}). So,
$$\# \Dcal_{f,n}'=(n-2f+1)\# \Dcal_{f-1,n-1}'+\cdots+1\cdot \#
\Dcal_{f-1,2f-1}'.$$ By induction assumption  on $\# \Dcal_{f-1,k}'$
for $2f-1\le k\le n-1$,
$$\# \Dcal_{f,n}'=\sum_{k=2f}^n  \frac{(k-2f+1)(k-1)!} {(k-2f+1)! (f-1)!
2^{f-1}}. $$ By induction on $n$, we have
 $\# \Dcal_{f,n}'=\frac{n!}{(n-2f)!f!2^f}=\# \Dcal_{f, n}$.
Since  $\Dcal_{f,n}\supseteq \Dcal_{f,n}'$, $\Dcal_{f,n}=
\Dcal_{f,n}'$.
\end{proof}

Given two standard $\lambda$--tableaux $\s, \t$ for some
$r$-partition $\lambda$. It has been defined in \cite{DJM:cyc} that
$$\m_{\s\t}=g_{d(\s)^{-1}}\cdot \prod_{s=2}^r
      \prod_{i=1}^{a_{s-1}}(y_i-u_s)
        \sum_{w\in\Sym_\lambda}g_w \cdot g_{d(\t)}, $$
where $a_{s-1}=|\lambda^{(1)}|+\dots+|\lambda^{(s-1)}|$.

For each $w\in \mathfrak S_n$, let $T_w=T_{i_1} \cdots T_{i_k}$ if
$s_{i_1} s_{i_2}\cdots s_{i_k}$ is a reduced expression of $w$ in
$\mathfrak S_n$.  By Matsumoto's theorem (see e.g.
\cite[1.2.2]{GP}), $T_w\in \W_{r, n}$ is independent of the reduced
expression of $w$.

\begin{Defn}\label{M_st} Suppose that
$\lambda\in\Lambda_r^+(n-2f)$, for some non-negative integer $f\le
\floor{n2}$. For each pair $(\s,\t)$ of standard $\lambda$--tableaux
define
$$M_{\s\t}=T_{d(\s)^{-1}}\cdot \prod_{s=2}^r
      \prod_{i=1}^{a_{s-1} }(X_i-u_s)
        \sum_{w\in\Sym_\lambda}T_w \cdot T_{d(\t)}.$$
\end{Defn}

\begin{Lemma}\label{M_st properties} Suppose that
$\lambda\in\Lambda_r^+(n-2f)$, for some non-negative integer $f\le
\floor{n2}$.  For any  $\s,\t\in\Std(\lambda)$,
\begin{enumerate}
\item $E^fM_{\s\t}=M_{\s\t}E^f\in \W_{r,n}^f$.
\item If $\kappa\in \Nrf$ then $M_{\s\t}X^\kappa=X^\kappa M_{\s\t}$.
\item If $w$ is a permutation on $\{n-2f+1,\dots,n\}$ then
$M_{\s\t}T_w=T_wM_{\s\t}$.
\item  $\sigma_f(\m_{\s\t})=\pi_{f, n} (E^fM_{\s\t})$.\end{enumerate}
\end{Lemma}

\begin{Defn}\label{Wlambda}
Given a $\lambda\in \Lambda_r^+(n-2f)$ and  $0\le f\le\floor{n2}$.
Define $\W_{r,n}^{\unrhd(f, \lambda)}$ to be the two--sided ideal of
$\W_{r,n}$ generated by $\W_{r,n}^{f+1}$ and the elements
$$\set{E^fM_{\s\t}|\s,\t\in \Std(\mu) \text{ and }
    \mu\in\Lambda_r^+(n-2f)\text{ with }\mu\trianglerighteq\lambda}.$$
We also define  $\Wlam =\sum_{\mu\gdom \lambda}\W_{r,n}^{\gedom(f,
\mu)}$, where in the sum $\mu\in\Lambda_r^+(n-2f)$.
\end{Defn}

\begin{Theorem}\label{cellular1} Suppose that $\s\in \Std(\lambda)$. Let $\Delta_\s(f,\lambda)$ be the $R$-submodule of
$\W_{r,n}^{\unrhd(f, \lambda)}/\W_{r,n}^{\rhd(f, \lambda)}$  spanned
by $\set{E^fM_{\s\t}X^\kappa
T_d+\Wlam|(\t,\kappa,d)\in\delta(f,\lambda)},$ where
$\delta(f,\lambda)
  =\set{(\t,\kappa,d)|\t\in\Std(\lambda),\kappa\in \Nrf \text{ and }
                        d\in \Dcal_{f,n}}$.
Then $\Delta_\s(f,\lambda)$ is a right $\W_{r, n}$-module.
                        \end{Theorem}

In order to prove Theorem~\ref{cellular1}, we need several Lemmas as
follows.

\begin{Lemma}\label{l1} Let $N_1$  be the $R$-submodule generated by $\W_{r,  n-2}' E_{n-1}X_{n-1}^{\alpha}T_d$
 where $d\in\Dcal_{1,n}$ and $-p\le \alpha \leq p$. Then $N_1$ is a right $\W_{r, n}$-module.
\end{Lemma}
\begin{proof} We need to verify
\begin{equation}\label{key1}  E_{n-1}X_{n-1}^{\alpha}T_d h\in  N_1\end{equation}   for any
$h\in \{T_i, E_i, X_1, 1\le i\le n-1\}$. Since
$X_{i}=T_{i-1}X_{i-1}T_{i-1}$,  $\{T_i, E_i, X_1, 1\le i\le n-1\}$
are algebraic generators of $\W_{r, n}$.  Since we are assuming that
$d\in \mathcal \Dcal_{1, n}$, by Lemma~\ref{cosets}, $d=s_{n-1, i}
s_{n, j}$ for some integers $i, j$ with $1\le i<j\le n$.

First, we verify (\ref{key1}) for  $h=E_k$, $1\le k\le n-1$. By
Definition~\ref{Waff relations}(b)(c), $E_{n-1}X_{n-1}^{\alpha}T_d
E_k= E_k E_{n-1}X_{n-1}^{\alpha}T_d\in N_1$ if $k\le i-2$. When
$k=i$ and $j=i+1$,
$$E_{n-1}X_{n-1}^{\alpha}T_{n-1, i} T_{n, j} E_k
=E_{n-1}X_{n-1}^{\alpha}E_{n, i}=\omega_{n-1}^{(\alpha)} E_{n, i},$$
where $\omega_{n-1}^{(\alpha)}\in Z(\W_{r, n-2})$ is given in
Lemma~\ref{tilde W}. For the simplification of notation, we use
$\omega_{n-1}^{(\alpha)}$ to express its image in $\W_{r, n-2}'$ in
the previous equation. Since $E_{n, i}=E_{n-1} T_{n-1, i} T_{n,
i+1}$, $\omega_{n-1}^{(\alpha)} E_{n, i}\in N_1$. If $j+1\leq k$,
then $$\begin{aligned}
   & E_{n-1}X_{n-1}^{\alpha}T_{n-1, i}T_{n, j}E_{k}\\
                                  = &E_{n-1}X_{n-1}^{\alpha}T_{n-1, i}T_{n, k-1}E_{k}T_{k-1, j}\\
                                  = &E_{n-1}X_{n-1}^{\alpha}T_{n-1, i}T_{n, k+1}E_{k-1}E_{k}T_{k-1, j}\\
                                  = & E_{n-1}X_{n-1}^{\alpha}T_{n-1, k-2}E_{k-1}T_{k-2, i}T_{n, k+1}E_{k}T_{k-1, j}\\
                                  = & E_{k-2}E_{n-1}X_{n-1}^{\alpha}T_{n-1, k}E_{k-1}T_{k-2, i}T_{n, k+1}E_{k}T_{k-1, j}\\
                                   = & E_{k-2}T_{k-2, i}T_{n-2, k-1}^{-1}E_{n-1}X_{n-1}^{\alpha}E_{n-1, k-1}T_{n, k+1}E_{k}T_{k-1, j}\\
                                     = & E_{k-2}T_{k-2, i}T_{n-2, k-1}^{-1}X_{n-2}^{-\alpha}T_{n-2, k-1}
                                     E_{n, k+1} E_{k}E_{k-1}E_{k}T_{k-1, j}\\
                                     = &E_{k-2}T_{k-2, i}T_{n-2, k-1}^{-1}X_{n-2}^{-\alpha}T_{n-2, j}E_{n, k}\in N_1\\
\end{aligned}$$

It is not difficult to check (\ref{key1}) if one of the conditions
holds:  (1) $k=i-1$, (2) $k=i, j>i+1$, (3) $i+1\le k\le j-2$, (4)
$k=j-1$ and $j>i+1$,  (5) $k=j$. We leave the details to the reader.
So, (\ref{key1}) holds for $h=E_k$ for all positive integers $k\le
n-1$.

Now, suppose $h=T_k$. Similarly, there are eight cases we have to
discuss.  We only check three cases and leave the remainder  cases
to the reader. Note that  the result for $ k\leq i-2$ follows
immediately from Definition~\ref{Waff relations}(b)(c).

If $k=i$ and $j>i+1$, then {\small $$ E_{n-1}X_{n-1}^{\alpha}T_{n-1,
i+1} T_{i}^2T_{n, j}
                             \overset{\ref{Waff relations}(b)}  {=}E_{n-1}X_{n-1}^{\alpha}T_{n-1, i+1}(1+\delta T_i-\varrho\delta E_i)T_{n, j}\in
                             N_1$$}
since we have already proved that (\ref{key1}) holds for $h=E_i$.
Similarly, when $k=j$,  $$ E_{n-1}X_{n-1}^{\alpha}T_{n-1, i}T_{n,
j+1}T_{j}^2=E_{n-1}X_{n-1}^{\alpha}T_{n-1, i}T_{n, j+1}(1+\delta
T_{j}-\varrho\delta E_{j})\in N_1. $$

Finally, we consider the case when  $h=X_1$. By Definition~\ref{Waff
relations}(b)-(c), (\ref{key1}) holds when $i>1$. If $i=1$, then $$
E_{n-1}X_{n-1}^{\alpha}T_{n-1, i} T_{n, j}X_1
=E_{n-1}X_{n-1}^{\alpha+1}T_{1, n-1}^{-1}T_{n, j}.$$ Since we have
verified (\ref{key1}) for  $h\in \{T_i, E_i\mid 1\le i\le n-1\}$, by
Definition~\ref{Waff relations}, (\ref{key1}) still holds when
$h=T_i^{-1}$ for all $1\le i\le n-1$. So, $
E_{n-1}X_{n-1}^{\alpha}T_{n-1, i} T_{n, j}X_1 \in N_1$ if
$E_{n-1}X_{n-1}^{\alpha+1}\in N_1$. By assumption, it is the case
when $-p\le \alpha< p$. So, we need consider the case $\alpha=p$. We
have $$\begin{aligned}
 E_{n-1}X_{n-1}^{p+1}&= E_{n-1}{X_{n-1}^{p}T_{n-2}}X_{n-2}T_{n-2}\\
                     &\overset{\ref{relations2}}{=} E_{n-1}(T_{n-2}X_{n-2}^p-\sum_{i=1}^p \delta X_{n-1}^i(E_{n-2}-1)X_{n-2}^{p-i})X_{n-2}T_{n-2}.\\
                   \end{aligned}$$

We claim that $E_{n-1}E_{n-2}X_{n-2}^{k}\in N_1$ for all positive
integers $k\le p+1$. If so, then  $E_{n-1}X_{n-1}^{p+1}\in N_1$
since $E_{n-1} X_{n-1}^i\in N_1$ for $1\le i\le p$, $
E_{n-1}T_{n-2}X_{n-2}^{p+1}=E_{n-1} E_{n-2}
                   X_{n-2}^{p+1} T_{n-1}^{-1}$ and $E_{n-1}  X_{n-1}^i
                   E_{n-2}=X_{n-2}^{-i} E_{n-1}E_{n-2}$. We remark
                   that, in this case,  we need to use (\ref{key1}) for $h\in \{T_i,
                   E_i\mid 1\le i\le n-1\}$, which has been proved.

In order to prove our claim, we use $E_{n-2}X_{n-2}^{p+1}$ instead
of $E_{n-1}X_{n-1}^{p+1}$, and repeat the previous arguments, we see
that our claim  holds if  $E_{n, i} X_i^k \in N_1$ for all positive
integers $2\le i\le n-1$, $0\le k\le p$  and $0\le k\le p+1$ when
$i=1$. Since $\prod_{i=1}^r(X_1-u_i)=0$ and $r=2p+1$,  we can use
$-p\le k\le p$ instead of $0\le k\le p+1$ when $i=1$.

Obviously, $E_{n, n-1} X_{n-1}^k \in N_1$. In general, $E_{n,
i}X_{i}^{k}=E_{n, i+1}T_{i}X_{i}^{k}T_{i+1}$. By
Lemma~\ref{relations2},
$$E_{n, i} X_i^k= E_{n, i+1}(X_{i+1}^{k} T_i +\sum_{j=1}^k
\delta X_{i+1}^j (E_i-1) X_i^{k-j})T_{i+1}.$$

Since we have already proved that $N_1$ is stable under the actions
of $E_j, T_j$, $1\le j\le n-1$, by our induction assumption on $i$
and $k$, $E_{n, i}X_{i}^{k}\in N_1$.

 When $i=1$, we have to discuss the case
when $-p\le k\le 0$.
 In fact,  $E_{n-1}X_{n-1}^{k}\in N_1$ for any $-p\le k\le 0$.
In general,
\begin{align*}
E_{n, i}X_{i}^{k}&=E_{n, i+1} T_{i}^{-1}X_{i}^{k}T_{i+1}^{-1}\\
                         & \overset{\ref{relations2}}{=}E_{n, i+1}(X_{i+1}^{k}T_{i}^{-1}-\sum_{j=1}^{-k} \delta
                         X_{i+1}^{-j}(E_{i}-1)X_{i}^{k+j})T_{i+1}^{-1}\\
\end{align*}
 By our  induction assumption on $i$ and $k$, we have  $E_{n, i}X_{i}^{k}\in N_1$. Setting $i=1$
 yields the result as required.
\end{proof}

The following results can be proved by arguments in the proof of
Lemma~\ref{l1}.
\begin{Cor} Let $d\in \Dcal_{1, n}$ and let $\alpha$ be integers with $-p\le \alpha\le p$.
\begin{enumerate} \label{et}
\item $E_{n-1}X_{n-1}^{\alpha}T_dE_{n-2}\in \W_{r,
n-2}'E_{n-1}E_{n-2}$. \item $E_{n-1}X_{n-1}^{\alpha}T_{d}E_{n-1}\in
\W_{r, n-2}' E_{n-1}$.\end{enumerate}\end{Cor}

\begin{Lemma}\label{l9} For any  $k\in \mathbb Z$,
$E_{n-1}X_{n-1}^{k}E_{n-3}\in N_3$ where $N_3$ is the $R$-submodule
of $\W_{r, n}$ generated by $\W_{r,
n-4}'E_{n-1}X_{n-1}^{\ell}E_{n-3}$ and $\W_{r,
n-4}'E_{n-3}X_{n-3}^{\ell}E_{n-1} T_w$ where $w=d_1 s_{n-1,
n-3}s_{n, n-2}$ for some $1\ne d_1\in \Dcal_{1,n-2}$ and  $\ell \in
\mathbb Z$ with $|\ell|\le p$.
\end{Lemma}
\begin{proof}
Note that  $E_{n-1}X_{n-1}^{k}E_{n-3}
=E_{n-3}X_{n-3}^{k}E_{n-1}E_{n-2}E_{n-3}$. Applying Lemma~\ref{l1}
on  $E_{n-3}X_{n-3}^{k}$, we can write
$E_{n-3}X_{n-3}^{k}E_{n-1}E_{n-2}E_{n-3}$ as an $R$-linear
combination of elements in $\W_{r, n-4}'
E_{n-3}X_{n-3}^{\ell}T_{d_1}E_{n-1}E_{n-2}E_{n-3}$ where $d_1\in
\Dcal_{1, n-2}$ and  $\ell\in \mathbb Z$ with $| \ell |\le p$.  Such
elements are in $N_3$ if $d_1\neq 1$ since $E_{n-1} T_{d_1}=T_{d_1}
E_{n-1}$ and  $E_{n-1}E_{n-2}E_{n-3}=E_{n-1}
T_{n-2}T_{n-1}T_{n-3}T_{n-2}$. When $d_1=1$,
$E_{n-3}X_{n-3}^{\ell}T_{d_1}E_{n-1}E_{n-2}E_{n-3}=E_{n-1}X_{n-1}^{\ell}E_{n-3}\in
N_3$.
\end{proof}

\begin{Lemma}\label{l11}
Given an integer $-p\le k\le p$. $E_{n-3}  \W_{r, n-2}'
E_{n-1}E_{n-2}X_{n-2}^{k}\in N_{4, k}$ where $N_{4, k}$ is the
$R$-submodule of $\W_{r, n}$ generated by $$\left\{\W_{r,
n-4}'E_{n-3}X_{n-3}^{\ell_1}E_{n-1}X_{n-1}^{k_1}T_{d_1}
T_{d_2}\Big|{ d_2\in\{1,s_{n-2},s_{n-2}s_{n-1}\},\atop d_1\in
\Dcal_{1,n-2}, |\ell_1|\leq p,  |k_1|\leq |k|}\right\}.$$
\end{Lemma}

\begin{proof} Applying  Lemma~\ref{relations2}(3)(6) (resp. Corollary~\ref{relation1}) on $
E_{n-1}X_{n-1}^{\ell}T_{n-1}$ (resp.
$E_{n-1}X_{n-1}^{\ell}T_{n-2}E_{n-1}$), we see that both of them can
be written as $R$-linear combinations of elements in
$\W_{r,n-2}'E_{n-1}X_{n-1}^{m}$ with $|m|\leq |\ell|$. Now, one can
verify $N_{4, k} T_{n-1}\subseteq N_{4, k}$ by Definition~\ref{Waff
relations} and Lemma~\ref{l1}  for $\W_{r,n-2}'$ without any
difficult.

Given a positive integer $ k\leq p$ and an element $h\in \W_{r,
n-2}'$. Then
\begin{equation}\label{l6}\begin{aligned} &E_{n-3}h
E_{n-1}E_{n-2}X_{n-2}^{k}
             =E_{n-3} h E_{n-1}T_{n-2}X_{n-2}^{k}T_{n-1} \\
              \overset{\ref{relations2}}{=} & E_{n-3}hE_{n-1}
                           (X_{n-1}^{k}T_{n-2}+\sum_{i=1}^{k}\delta
                           X_{n-1}^{i}(E_{n-2}-1)X_{n-2}^{k-i})T_{n-1}\\
\end{aligned}\end{equation}
By Lemma~\ref{l1} for $E_{n-3}h$, $E_{n-3}hE_{n-1}
X_{n-1}^{k}T_{n-2}T_{n-1}\in N_{4, k}$.
 Also, we
have  $$E_{n-3}hE_{n-1}
                           X_{n-1}^{i}X_{n-2}^{k-i}T_{n-1}
= E_{n-3} h X_{n-2}^{k-i}E_{n-1} X_{n-1}^{i}T_{n-1}\in N_{4, k}.$$
One can verify the above inclusion by using $N_{4, k} T_{n-1}\subset
N_{4, k}$ and Lemma~\ref{l1} for $E_{n-3} h X_{n-2}^{k-i}$.

Finally, by induction assumption on $k$ (when $k=0$, the result is
trivial since $E_{n-1}E_{n-2} =E_{n-1} T_{n-2} T_{n-1}$), we have
$$ h_1:= E_{n-3}hX_{n-2}^{-i}E_{n-1}
E_{n-2}X_{n-2}^{k-i}\in N_{4, k}$$ for all positive integers $i\le
k$. Since $N_{4, k} T_{n-1}\subset  N_{4, k}$, $ h_1  T_{n-1} \in
N_{4. k}$. So, $E_{n-3}h E_{n-1}E_{n-2}X_{n-2}^{k}\in N_{4,
k}$.

The case when  $-p\leq k\leq 0$ can be discussed similarly. The only
difference is that we have to use Lemma~\ref{relations2}(4) instead
of Lemma~\ref{relations2}(1). We leave the details to the reader.
\end{proof}

\begin{Lemma} \label{n4} Let $N_4$ be the $R$-submodule of $\W_{r, n}$ generated by $\W_{r,
n-4}'E_{n-3}X_{n-3}^{\ell}E_{n-1}X_{n-1}^{k}T_{d_1} T_{d_2}$ where
$d_2\in\{1,s_{n-2},s_{n-2}s_{n-1}\}$, $d_1\in \Dcal_{1,n-2}$, $\ell,
k\in \mathbb Z$ with  $-p\le k, \ell\le p$. Then $N_4 h\subset h$
for $h\in \{T_{n-1}, T_{n-2}, E_{n-1}, E_{n-2}\}$. \end{Lemma}
\begin{proof} By   Corollary~\ref{et} and Lemma~\ref{l1}, we have
$N_4 E_{i}\subset N_4$ for $i\in \{n-1, n-2\}$.  Using this result
together with  Lemma~\ref{relations2} and Definition~\ref{Waff
relations}(b), we have  $N_4 T_{i}\subset N_4$ for $i\in \{n-1,
n-2\}$.\end{proof}

\begin{Lemma} \label{ml}Fix an integer $\ell$ with $-p\le l\le p$.  Let $M_\ell\subset \W_{r, n}$ be the $R$-module generated by $\{\W_{r,
n-4}' E_{n-1} X_{n-1}^{k} E_{n-3}X_{n-3}^{i}|k\in \mathbb Z, |i|\le
|\ell|\}$. Let $N_4$ be the $R$-module defined in Lemma~\ref{n4}. If
$E_{n-1}X_{n-1}^{\ell'}E_{n-3}X_{n-3}^{m}T_{n-2}T_{n-3}\in N_4$ for
all integers  $|\ell'|< |\ell|$ and $|m|\le p$,  then
$M_\ell\subseteq N_4$.
\end{Lemma}

\begin{proof} First, we assume that $0\le i\le |\ell|$.
By Lemma~\ref{l9},  $M_\ell\subseteq N_4$ if

\begin{itemize}\item
[(1)]$E_{n-1} X_{n-1}^a E_{n-3} X_{n-3}^{i}\in N_4$ for all $ a\in
\mathbb Z$ and  $|a|\le p$,
\item [(2)] $A\in N_4$ where  $A=E_{n-3} X_{n-3}^a E_{n-1} T_{d_1}
T_{n-2}T_{n-3}T_{n-1}T_{n-2} X_{n-3}^{i}$ for all $a\in \mathbb Z$,
$|a|\le p$  and  $1\ne d_1\in \Dcal_{1, n-2}$.
\end{itemize}

By the definition of $N_4$,  (1) holds. We prove (2) by induction on
$i$.

Suppose  $i=0$. (2) holds since  $d_1
s_{n-2}s_{n-3}s_{n-1}s_{n-2}\in \Dcal_{2,n}$. When $i>0$, {\small
\begin{equation}\label{l13}
\begin{aligned} A=&
E_{n-3}X_{n-3}^{a}T_{d_1}E_{n-1}T_{n-2}X_{n-2}^{i}T_{n-3} T_{n-1}T_{n-2}  \\
& +  E_{n-3}X_{n-3}^{a}T_{d_1}E_{n-1}T_{n-2} \sum_{j=1}^{i}\delta
                           X_{n-2}^{j}(E_{n-3}-1)X_{n-3}^{i-j}T_{n-1}T_{n-2}\\
                           \end{aligned}
\end{equation}}
We consider the second term (up to a scalar in $R$)  of (\ref{l13})
as follows.
$$
\begin{aligned}
&E_{n-3}X_{n-3}^{a}T_{d_1}E_{n-1}T_{n-2}X_{n-2}^{j}E_{n-3}X_{n-3}^{i-j}T_{n-1}T_{n-2}\\
=&E_{n-3}X_{n-3}^{a}T_{d_1}X_{n-3}^{-j}E_{n-1}E_{n-2}E_{n-3}X_{n-3}^{i-j}T_{n-2}\\
=&E_{n-3}X_{n-3}^{a}T_{d_1}X_{n-3}^{-j}E_{n-1}T_{n-2}T_{n-3}T_{n-1}T_{n-2}X_{n-3}^{i-j}T_{n-2}\\
 \overset {\ref{key1}} {\in } &\sum_{\substack {-p\le a'\le p\\ d_1'\in \Dcal_{1, n-2}}}
 \W_{r, n-4}'E_{n-3}X_{n-3}^{a'}T_{d_1'}E_{n-1}T_{n-2}T_{n-3}T_{n-1}T_{n-2}X_{n-3}^{i-j}T_{n-2}\\
\end{aligned}
$$
If $d_1'\neq 1$, then
$E_{n-3}X_{n-3}^{a'}T_{d_1'}E_{n-1}T_{n-2}T_{n-3}T_{n-1}T_{n-2}X_{n-3}^{i-j}T_{n-2}\in
N_4$ by induction assumption on $i$ and Lemma~\ref{n4}.

If $d_1'= 1$, we still have
$E_{n-3}X_{n-3}^{a'}E_{n-1}T_{n-2}T_{n-3}T_{n-1}T_{n-2}X_{n-3}^{i-j}T_{n-2}
=E_{n-1}X_{n-1}^{a'}E_{n-3}X_{n-3}^{i-j}T_{n-2}\in N_4$, by
Lemma~\ref{n4}.

We consider the third term on the right hand side of (\ref{l13}). We
have
$$
\begin{aligned}
&E_{n-3}X_{n-3}^{a}T_{d_1}E_{n-1}T_{n-2}X_{n-2}^{j}X_{n-3}^{i-j}T_{n-1}T_{n-2}\\
=&E_{n-3}X_{n-3}^{a}T_{d_1}E_{n-1}E_{n-2}X_{n-2}^{j}X_{n-3}^{i-j}T_{n-2}\\
=&E_{n-3}X_{n-3}^{a}T_{d_1}X_{n-3}^{i-j}E_{n-1}E_{n-2}X_{n-2}^{j}T_{n-2}\in N_4\\
\end{aligned}
$$
We remark that we use Lemma~\ref{l11} for
$E_{n-3}X_{n-3}^{a}T_{d_1}X_{n-3}^{i-j}E_{n-1}E_{n-2}X_{n-2}^{j}$
and Lemma~\ref{n4} to get the above inclusion.

We use  Lemma~\ref{relations2} to express the first term on the
right hand side of (\ref{l13}) as follows:  {\small
\begin{equation} \label{l141}
E_{n-3}X_{n-3}^{a}T_{d_1}E_{n-1}(X_{n-1}^{i}T_{n-2}+\sum_{j=1}^{i}\delta
                           X_{n-1}^{j}(E_{n-2}-1)X_{n-2}^{i-j})T_{n-3}T_{n, n-2}.\end{equation}}

Since we are assuming that   $d_1\neq 1$, $d_1 s_{n-1, n-3}s_{n,
n-2}\in \Dcal_{2, n}$. So, the first term on the right hand side of
(\ref{l141}) is in $N_4$.

By Lemma~\ref{l1} for $E_{n-3} X_{n-3}^a T_{d_1} X_{n-2}^{i-j}
T_{n-3}$, and Lemma~\ref{n4},
$$E_{n-3} X_{n-3}^a T_{d_1} X_{n-2}^{i-j} T_{n-3} E_{n-1}
X_{n-1}^j T_{n-1} T_{n-2}\in N_4.$$ In other words, the third term
on the right hand side of (\ref{l141}) is in $N_4$.

In order to show that the second term on the right hand side of
(\ref{l141}) is in $N_4$, we need to show that
$$ B:= E_{n-3}X_{n-3}^{a}T_{d_1}E_{n-1}X_{n-1}^{j}E_{n-2}X_{n-2}^{i-j}T_{n-3}T_{n-1}T_{n-2}\in
N_4.$$ Note that
$$E_{n-3}X_{n-3}^{a}T_{d_1}E_{n-1}X_{n-1}^{j}E_{n-2}X_{n-2}^{i-j}
=E_{n-3}X_{n-3}^{a}T_{d_1}X_{n-2}^{-j}
E_{n-1}E_{n-2}X_{n-2}^{i-j}.$$ By Lemma~\ref{l11},  $B$ can be
written as linear combination of elements in $\W_{r, n-4}'
E_{n-3}X_{n-3}^uT_{d'
}E_{n-1}X_{n-1}^vT_{d_2}T_{n-3}T_{n-1}T_{n-2}$, where $u, v\in
\mathbb Z$ with  $|u|\leq p$, $|v|\leq i-j\leq i-1$, $d'\in
\Dcal_{1,n-2}$, $d_2\in\{1,s_{n-2},s_{n-2}s_{n-1}\}$. In order to
finish the proof, we need to show that
\begin{equation}\label{c} C:=E_{n-3}X_{n-3}^uT_{d'}E_{n-1}X_{n-1}^vT_{d_2}T_{n-3}T_{n-1}T_{n-2}\in
N_4.\end{equation}

There are four cases we have to discuss.

\begin{itemize}\item[(1)]   $d_2=1$. By Lemma~\ref{l1}
and Lemma~\ref{n4}, $C\in N_4$ since
$E_{n-3}X_{n-3}^uT_{d'}E_{n-1}X_{n-1}^vT_{d_2}T_{n-3}
=E_{n-3}X_{n-3}^uT_{d'} T_{n-3}E_{n-1}X_{n-1}^v$.

\item [(2)] $d_2=s_{n-2}$ and  $d'\neq 1$.  $C\in N_4$    since $d' s_{n-2}s_{n-3}s_{n-1}s_{n-2}\in \Dcal_{2,n}$.

\item [(3)] $d_2=s_{n-2}$ and  $d'=1$. By Lemma~\ref{n4} and our
assumption $C\in N_4$ since
$C=E_{n-3}X_{n-3}^uE_{n-1}X_{n-1}^vT_{n-2}T_{n-3}T_{n-1}T_{n-2}$ and
$|v|\le i-1\le \ell-1$.

\item [(4)] $d_2=s_{n-2}s_{n-1}$: (\ref{c}) follows from Lemma~\ref{n4}
and the result for $d_2=s_{n-2}$.
\end{itemize}
This completes the proof of the result under the assumption $0\le
i\le |\ell|$. When $-|\ell|\le i\le 0$, One  can verify the result
similarly by induction on $i$. Note that the result holds when
$i=0$. We remark that   we have to use Lemma~\ref{relations2}(5)
instead of Lemma~\ref{relations2}(1). We also need  $N_4
T_i^{-1}\subset N_4$ for $i=n-2, n-1$ which follows from
Lemma~\ref{n4} and Definition~\ref{Waff relations}(b), immediately.
We leave the details to the reader.
\end{proof}

\begin{Lemma}\label{01} Fix an integer $\ell$ with $-p\le \ell\le p$. Let   $N_4$ be  the
$R$-module defined in Lemma~\ref{n4}. If
$E_{n-1}X_{n-1}^{\ell'}E_{n-3}X_{n-3}^{m}T_{n-2}T_{n-3}\in N_4$ for
all integers  $|\ell'|< |\ell|$ and $|m|\le p$, then
$E_{n-3}X_{n-3}^{a}E_{n-1}E_{n-2}X_{n-2}^{b}T_{n-3}\in N_4$ with
$a\in\mathbb{Z}$ and $|b|\le|\ell|$.
\end{Lemma}
\begin{proof}
First, we assume that  $0\le b\le |\ell|$. Let
$h=E_{n-3}X_{n-3}^{a}E_{n-1}E_{n-2}X_{n-2}^{b}T_{n-3}$. By
Lemma~\ref{relations2}(1), we have
\begin{equation}\label{hh} h=E_{n-3}X_{n-3}^{a}E_{n-1}E_{n-2}(T_{n-3}X_{n-3}^{b} -
\sum_{i=1}^{b}\delta
                           X_{n-2}^{i}(E_{n-3}-1)X_{n-3}^{b-i}).\end{equation}
                           The first term on the right hand side of
                           (\ref{hh})   is equal to
\begin{equation}\label{hhh}E_{n-3}X_{n-3}^{a}E_{n-1}E_{n-2}E_{n-3}X_{n-3}^{b}T_{n-2}^{-1}
=E_{n-1}X_{n-1}^{a}E_{n-3}X_{n-3}^{b}T_{n-2}^{-1}\end{equation}
which is in $N_4$ by Lemmas~\ref{n4}--\ref{ml}.

The second term on the right hand side of (\ref{hh}) (up to a
scalar) is equal to $
E_{n-3}X_{n-3}^{a-i}E_{n-1}E_{n-2}E_{n-3}X_{n-3}^{b-i}$, which is in
$N_4$ by (\ref{hhh}). Finally, by Lemma~\ref{l11}, the last term is
in $N_4$.

When $b<0$, one can verify the result similarly. We remark that we
have to use Lemma~\ref{relations2}(4) instead of
Lemma~\ref{relations2}(1).\end{proof}

\begin{Prop}\label{l21}
$E_{n-3}X_{n-3}^{k}E_{n-1}X_{n-1}^{\ell}T_{n-2}T_{n-3}\in N_4$ for
all integers $k,\ell$ with $-p\le k,\ell\le p$.
\end{Prop}

\begin{proof}
Let $h=E_{n-3}X_{n-3}^{k}E_{n-1}X_{n-1}^{\ell}T_{n-2}T_{n-3}$. We
prove $h\in N_4$ by induction on $|\ell|$.

If $\ell = 0$, then
$$\begin{aligned} h = &E_{n-3}X_{n-3}^{k}E_{n-1}T_{n-2}T_{n-3}
=E_{n-3}X_{n-3}^{k}E_{n-1}E_{n-2}T_{n-1}^{-1}T_{n-3}\\
=&E_{n-1}X_{n-1}^{k}E_{n-3}E_{n-2}T_{n-3}T_{n-1}^{-1}
= E_{n-1}X_{n-1}^{k}E_{n-3}T_{n-2}^{-1}T_{n-1}^{-1}\in N_4\\
\end{aligned}$$ by Lemma~\ref{n4}. If $\ell> 0$, by Lemma~\ref{relations2}(1),
we have
\begin{equation}\label{02} h=E_{n-3}X_{n-3}^{k}E_{n-1}(T_{n-2}X_{n-2}^{\ell} - \sum_{i=1}^{\ell}\delta
                           X_{n-1}^{i}(E_{n-2}-1)X_{n-2}^{\ell-i})T_{n-3}
\end{equation}
The first term on the right hand side (\ref{02}) is equal to
$$E_{n-3}X_{n-3}^{k}E_{n-1}T_{n-2}X_{n-2}^{\ell}T_{n-3}
=E_{n-3}X_{n-3}^{k}E_{n-1}E_{n-2}X_{n-2}^{\ell}T_{n-3}T_{n-1}^{-1}$$
which is in $N_4$ by our induction assumption on for all integers
$\le \ell-1$, together with Lemma~\ref{01} and Lemma~\ref{n4}.

By induction and Lemma~\ref{01}, $E_{n-3}X_{n-3}^{k}E_{n-1}
X_{n-1}^{i}E_{n-2}X_{n-2}^{\ell-i}T_{n-3}=E_{n-3}X_{n-3}^{k+i}E_{n-1}
E_{n-2}X_{n-2}^{\ell-i}T_{n-3}\in N_4$. So, The second term on the
right hand side (\ref{02}) is in $N_4$.

Finally, we consider the third term on the right hand side
(\ref{02}). However, this term is equal to
$E_{n-3}X_{n-3}^{k}X_{n-2}^{\ell-i}T_{n-3}E_{n-1}X_{n-1}^{i}$, which
is in $N_4$ by Lemma~\ref{l1}.

When $\ell<0$, one can verify the result similarly. We remark that
we have to use Lemma~\ref{relations2}(4) instead of
Lemma~\ref{relations2}(1).
\end{proof}

\begin{Prop} \label{2line} Let $N_2$ be the $R$-submodule of $\W_{r,
n}$ generated by $\W_{r, n-4}'
E_{n-3}X_{n-3}^{k}E_{n-1}X_{n-1}^{\ell}T_d$ where $-p\leq k, \ell
\leq p$ and  $d\in \Dcal_{2,n}$. Then $N_2$ is a right
$\W_{r,n}$-module.
\end{Prop}
\begin{proof}  Applying  Lemma~\ref{l1} twice,
$E_{n-3}X_{n-3}^{k}E_{n-1}X_{n-1}^{\ell}T_d h$, for $h\in \W_{r,
n}$, can be written as an $R$-linear combination of elements
$$\W_{r, n}'E_{n-3}X_{n-3}^{k_1}T_{n-3, i_4}T_{n-2,
i_3}E_{n-1}X_{n-1}^{\ell_1}T_{n-1, i_2}T_{n, i_1}$$ where $i_1, i_2,
i_3, i_4, k_1, \ell_1\in \mathbb Z$  with $i_2<i_1$, $i_4<i_3$,
$-p\le k_1, \ell_1\le p$. In order to show that
$E_{n-3}X_{n-3}^{k}E_{n-1}X_{n-1}^{\ell}T_d h\in N_2$, we need to
show
\begin{equation}\label{l22}A:=E_{n-3}X_{n-3}^{k_1}T_{n-3, i_4}T_{n-2,
i_3}E_{n-1}X_{n-1}^{\ell_1}T_{n-1, i_2}T_{n, i_1}\in
N_2.\end{equation} We are going to prove (\ref{l22})  by induction
on $i_2$.

We assume that $i_4\ge i_2$. Otherwise, $A\in N_2$ and (\ref{l22})
follows. In particular, (\ref{l22}) holds for  $i_2\in \{n-1,
n-2\}$. By Lemma~\ref{l1} and induction on $i_2$,
\begin{equation}\label{indassump}
E_{n-3}  \W_{r, n-2}' E_{n-1} X_{n-1}^j T_{n-1, i_2'} T_{n, i_1'}\in
N_2\end{equation} for all integers $i, j, i_2', i_1'$ with
$i_2'<i_1'$, $i, j\in \mathbb Z$, $-p\le j\le p$ and $i_2'>i_2$.

Since we are assuming that $i_4\ge i_2$ and $i_2<n-2$,
$$ A=E_{n-3}X_{n-3}^{k_1}E_{n-1}X_{n-1}^{\ell_1}T_{n-1, i_2}T_{n-2,
i_4+1}T_{n-1, i_3+1}T_{n, i_1}.$$

By  Proposition~\ref{l21}, $A$ is in the $R$-submodule of $\W_{r,
n}$ generated by $$\W_{r, n-4}'
E_{n-3}X_{n-3}^{k_2}T_{d_1}E_{n-1}X_{n-1}^{\ell_2}T_{d_2}T_{n-3,
i_2}T_{n-2, i_4+1}T_{n-1, i_3+1}T_{n, i_1}$$ where $d_1\in
\Dcal_{1,n-2}$, $d_2\in \{s_{n-2},s_{n-2}s_{n-1},1\}$ and $-p\le
k_2, \ell_2\le p$. In order to prove (\ref{l22}), it suffices to
prove  \begin{equation}\label{l23} B:=
E_{n-3}X_{n-3}^{k_2}T_{d_1}E_{n-1}X_{n-1}^{\ell_2}T_{d_2}T_{n-3,
i_2}T_{n-2, i_4+1}T_{n-1, i_3+1}T_{n, i_1}\in N_2\end{equation}
There are three cases we have to discuss.

\Case{1. $d_2=1$} We have $B=E_{n-3}X_{n-3}^{k_2}T_{d_1} T_{n-3,
i_2}T_{n-2, i_4+1} E_{n-1}X_{n-1}^{\ell_2}T_{n-1, i_3+1}T_{n, i_1}$
which is in $N_2$ by (\ref{indassump}) if $i_3+1<i_1$. When
$i_3+1\geq i_1$,  $T_{n-1, i_3+1}T_{n, i_1}=T_{n, i_1}T_{n, i_3+2}$
and
\begin{equation}\label{b100} B= E_{n-3}X_{n-3}^{k_2}T_{d_1} T_{n-3, i_2}T_{n-2, i_4+1}
E_{n-1}X_{n-1}^{\ell_2}T_{n, i_1}T_{n, i_3+2}.\end{equation}   We
use Lemma~\ref{relations2}(3)(6) for $
E_{n-1}X_{n-1}^{\ell_2}T_{n-1}$ to write $B$ as a linear combination
of elements in $ E_{n-3} \W_{r, n-2}' E_{n-1}X_{n-1}^{j}T_{n-1,
i_1}T_{n, i_3+2}$ with $-p\le j\le p$. By (\ref{indassump}), $B\in
N_2$. This completes the proof for $d_2=1$.

\Case{2. $d_2=s_{n-2}$} We have  $B=E_{n-3}X_{n-3}^{k_2}T_{d_1}
T_{n-3, i_2} T_{n-2, i_3} E_{n-1}X_{n-1}^{\ell_2}T_{n-1, i_4+1}T_{n,
i_1}$ which is in $N_2$ by the result for $d_2=1$.

\Case{3. $d_2=s_{n-2}s_{n-1}$} If $i_3+1<i_1$, then
$$\begin{aligned}B&=E_{n-3}X_{n-3}^{k_2}T_{d_1} T_{n-3, i_2}
E_{n-1}X_{n-1}^{\ell_2}T_{n-1, i_4+1}T_{n, i_3+1}T_{n, i_1}\\
&=E_{n-3}X_{n-3}^{k_2}T_{d_1} T_{n-3, i_2} T_{n-2, i_1-2}
E_{n-1}X_{n-1}^{\ell_2}  T_{n-1, i_4+1}T_{n, i_3+1}\\
\end{aligned}$$
So, (\ref{l23}) follows from (\ref{indassump}). Finally, we assume
that  $i_3+1\geq i_1$. Then $$ B=E_{n-3}X_{n-3}^{k_2}T_{d_1} T_{n-3,
i_2} E_{n-1}X_{n-1}^{\ell_2}T_{n-1, i_4+1} T_{n-1}^2 T_{n-1,
i_1}T_{n, i_3+2}.$$  By Definition~\ref{Waff relations}(b),
\begin{equation}\label{l25}\begin{aligned}B = &E_{n-3}X_{n-3}^{k_2}T_{d_1} T_{n-3, i_2}
E_{n-1}X_{n-1}^{\ell_2}T_{n-1, i_4+1}\\ & \times (1+\delta
T_{n-1}-\delta\varrho E_{n-1})  T_{n-1, i_1}T_{n, i_3+2}.\\
\end{aligned}
\end{equation}

The second term on the right hand side of (\ref{l25}) is equal to
$$\delta E_{n-3}X_{n-3}^{k_2}T_{d_1} T_{n-3, i_2}T_{n-2, i_3}
E_{n-1}X_{n-1}^{\ell_2}T_{n-1, i_4+1}T_{n, i_1}.$$ By our result for
$d_2=1$, it is in $N_2$. Similarly, using Corollary~\ref{et}b,
Lemma~\ref{l1}, (\ref{indassump}),  we see that the third term on
the right hand side of (\ref{l25}) is in $N_2$. In order to prove
$B\in N_2$, it remains to prove that
\begin{equation}\label{l26} C:=E_{n-3}X_{n-3}^{k_2}T_{d_1} T_{n-3,
i_2} E_{n-1}X_{n-1}^{\ell_2}T_{n-1, i_4+1}  T_{n-1, i_1}T_{n,
i_3+2}\in M.\end{equation}  In fact, when  $i_4+1<i_1$,
$$C=E_{n-3}X_{n-3}^{k_2}T_{d_1} T_{n-3, i_2}T_{n-2, i_1-1} E_{n-1}X_{n-1}^{\ell_2}
 T_{n-1, i_4+1} T_{n, i_3+2}\overset{
(\ref{indassump})} \in N_2.$$ Suppose that  $i_4+1\geq i_1$. By
Definition~\ref{Waff relations}(b),
\begin{equation}\label{l27} \begin{aligned} C =& E_{n-3}X_{n-3}^{k_2}T_{d_1} T_{n-3, i_2}
E_{n-1}X_{n-1}^{\ell_2}\\ & \times (1+\delta T_{n-2}-\delta\varrho
E_{n-2}) T_{n-2, i_1}T_{n-1, i_4+2} T_{n, i_3+2}.\\ \end{aligned}
\end{equation}

Note that $T_{n-1, i_1} T_{n-1, i_4+2} T_{n, i_3+2}=T_{n-2, i_4+1}
T_{n-1, i_1} T_{n, i_3+2}$. By (\ref{indassump}), the first and the
second terms on the right of (\ref{l27}) are in $N_2$. The third
term on the right of (\ref{l27}) (up to a scalar) is equal to
$$\begin{aligned} &  E_{n-3}X_{n-3}^{k_2}T_{d_1} T_{n-3, i_2}
E_{n-1}X_{n-1}^{\ell_2}E_{n-2} T_{n-2, i_1}T_{n-1, i_4+2} T_{n,
i_3+2}\\ = & E_{n-3}X_{n-3}^{k_2}T_{d_1} T_{n-3, i_2}
X_{n-2}^{-\ell_2} E_{n-1}E_{n-2} T_{n-2, i_1}T_{n-1, i_4+2} T_{n,
i_3+2}\\
=& E_{n-3}X_{n-3}^{k_2}T_{d_1} T_{n-3, i_2}
X_{n-2}^{-\ell_2} E_{n-1} T_{n-1, i_1}T_{n, i_4+2} T_{n, i_3+2} \\
=&  E_{n-3}X_{n-3}^{k_2}T_{d_1} T_{n-3, i_2}
X_{n-2}^{-\ell_2}T_{n-2, i_3} E_{n-1}  T_{n-1, i_1}T_{n, i_4+2}\overset{\ref{indassump}}\in N_2.\\
\end{aligned}
$$  We have proved that (\ref{l26}) holds in any case.
So, $B\in N_2$ and hence (\ref{l22}) holds.  \end{proof}

\textbf{Proof of Theorem~\ref{cellular1}:}  We claim that $M$ is a
right $\W_{r, n}$-module where $M$ is  the $R$--module generated by
$\W_{r, n-2f}' E^f X^\kappa T_d$ with $\kappa\in  \Nrf$ and $d\in
\Dcal_{f, n}$. If so, $ E^fM_{\s\t} X^\kappa T_d h+\W_{r,n}^{\rhd(f,
\lambda)}$ can be written as an $R$-linear combination of elements
$M_{\s\t}\W_{r, n-2f}' E^f X^{\kappa'} T_{d'}+\W_{r,n}^{\rhd(f,
\lambda)}$ for $\kappa'\in \Nrf$ and $d'\in \Dcal_{f, n}$. Now, the
result follows immediately from Lemma~\ref{M_st properties}(d).

It remains to prove our claim. Let $d\in \Dcal_{f, n}$. By
Lemma~\ref{cosets},  $$d=  s_{n-2f+1,i_f}s_{n-2f+2,j_f}\cdots
s_{n-1,i_1}s_{n,j_1}$$ for some integers $i_1, \dots, i_f, j_1,
\dots, j_f$ such that $1\leq i_f<\cdots<i_1\leq n$, $1\leq
i_k<j_k\leq n-2k+2$ for $1\leq k\leq f$. Write $d=d_1
s_{n-1,i_1}s_{n,j_1}$. For any $h\in \W_{r, n}$, since $$  E^f
X^\kappa T_d h=\prod_{i=2}^f E_{n-2i+1}X_{n-2i+1}^{\kappa_{n-2i+1}}
T_{d_1} E_{n-1}X_{n-1}^{k_{n-1}}T_{n-1,i_1}T_{n,j_1} h, $$  by
Lemma~\ref{l1},
 $ E^f X^\kappa
T_d h \in N$ where $N$  is  the $R$-submodule of $\W_{r, n}$
generated by $$\prod_{i=2}^f E_{n-2i+1}X_{n-2i+1}^{\kappa_{n-2i+1}}
T_{d_1}\W_{r, n-2}'
E_{n-1}X_{n-1}^{\kappa_{n-1}'}T_{n-1,k_1}T_{n,\ell_1}$$ where
$\kappa_{n-1}' k_1, \ell_1\in \mathbb Z$  with $k_1<\ell_1$ and
$|\kappa_{n-1}'|\le p$. By induction assumption on $n-2$ for our
claim, $E^f X^\kappa T_d h $ is in the $R$-submodule of $\W_{r, n}$
generated by
$$\W_{r, n-2f}' \prod_{i=1}^f
E_{n-2i+1}X_{n-2i+1}^{\kappa_{n-2i+1}'} T_{w_1} T_{n-1, k_1} T_{n,
\ell_1}$$ with $w_1=s_{n-2f+1,k_f}s_{n-2f+2,\ell_f}\cdots
s_{n-3,k_2}s_{n-2,\ell_2}$ with $w_1\in \Dcal_{f-1, n-2}$ and
$|\kappa_{n-2i+1}'|\le p$ for $1\le i\le f$. If $w_1 s_{n-1, k_1}
s_{n, \ell_1}\in \Dcal_{f, n}$ then our claim follows. In
particular, our claim follows if $k_1\in \{n-2, n-1\}$. It remains
to prove
\begin{equation}\label{cellw12} A:=\prod_{i=1}^f
E_{n-2i+1}X_{n-2i+1}^{\kappa_{n-2i+1}'}T_{{w_1}}T_{n-1,k_1}T_{n,\ell_1}\in
M\end{equation} for   $k_2\ge k_1$. We prove it by induction on
$k_1$.   In general, we have
\begin{equation}\label{cellw13}
A=\prod_{i=3}^f E_{n-2i+1}X_{n-2i+1}^{\kappa_{n-2i+1}'}T_{{d_2}}
B\end{equation}  where
$B=E_{n-3}X_{n-3}^{\kappa_{n-3}'}E_{n-1}X_{n-1}^{\kappa_{n-1}'}
T_{n-3,k_2}T_{n-2,\ell_2}T_{n-1,k_1}T_{n,\ell_1}$ and $d_2=w_1
(s_{n-3, k_2} s_{n-2, \ell_2})^{-1}$.

When $k_1\in \{n-1,n-2\}$, there is nothing to be proved since $k_2
\le n-3< k_1$. Suppose $k_1\le n-3$. By  arguments in the proof of
(\ref{l22}) for $i_4\ge i_2$, we can write $B$ as an $R$-linear
combination of elements
$$C:=E_{n-3}\W_{r,n-2}'E_{n-1}X_{n-1}^{\kappa_{n-1}''}
T_{n-1,k_1'}T_{n,l_1'}, $$ where $k_1<k_1'<l_1'$ and $-p\le
\kappa_{n-1}^{''}\le p$. By induction assumption on $\W_{r, n-2}$
for our claim, we can write $E^fX^{\kappa}T_{d}h$ as an $R$-linear
combination of elements
$$\W_{r, n-2f}'\prod_{i=1}^f E_{n-2i+1}X_{n-2i+1}^{\alpha_{n-2i+1}}T_{{y_1}}T_{n-1,k_1'}T_{n,l_1'},$$
where $\alpha\in \Nrf$, $y_1\in \Dcal_{f-1, n-2}$ and $k_1'>k_1$. By
our  induction assumption on $k_1$, $C\in M$. This completes the
proof of our claim. \qed

\begin{Prop}\label{Wlam spanning} Let $\ast:\W_{r, n}\rightarrow \W_{r, n}$  be
the $R$-linear anti-involution in Lemma~\ref{antiinvolution}.
Suppose $0\le f\le \floor{n2}$ and $\lambda\in\Lambda_r^+(n-2f)$.
Then $\W_{r,n}^{\unrhd(f, \lambda)}/\Wlam$ is spanned by the
elements
\begin{equation}\label{genW} \set{T_e^*X^\rho E^fM_{\s\t}X^\kappa T_d+\W_{r,n}^{\rhd (f, \lambda)}|
            (\t,\kappa,d),(\s,\rho,e)\in\delta(f,\lambda)}.
\end{equation}
\end{Prop}

\begin{proof}Let $W$ be the $R$--submodule of
$\W_{r,n}^{\unrhd(f, \lambda)}/\Wlam$ spanned by the elements in
(\ref{genW}).  By Theorem~\ref{cellular1}, $W$ is both left and
right $\W_{r, n}$-submodule of $\W_{r,n}^{\unrhd(f, \lambda)
}/\Wlam$. As the generators $\{E^fM_{\s\t}+\Wlam\}$ of
$\W_{r,n}^{\unrhd(f, \lambda)}/\Wlam$ are contained in~$W$,
$W=\W_{r,n}^{\unrhd(f, \lambda)}/\Wlam$.
\end{proof}

\begin{Defn}\label{W cell basis} Let $\Lambda_{r, n}^+=\set{(f,\lambda)|0\le
f\le\floor{n2}\text{ and }
                     \lambda\in\Lambda_r^+(n-2f)}$.
If $(f,\lambda)\in\Lambda_{r, n}^+$ and
$(\s,\rho,e),(\t,\kappa,d)\in\delta(f,\lambda)$ then we define
$$C^{(f,\lambda)}_{(\s,\rho,e)(\t,\kappa,d)}
              =T_e^*X^\rho E^fM_{\s\t}X^\kappa T_d.$$\end{Defn}

We recall the definition of cellular algebra as follows.

\begin{Defn}\cite{GL}\label{GL}
    Let $R$ be a commutative ring and $A$ an $R$--algebra.
    Fix a partially ordered set $\Lambda=(\Lambda,\gedom)$ and for each
    $\lambda\in\Lambda$ let $T(\lambda)$ be a finite set. Finally,
    fix $C^\lambda_{\bfs\bft}\in A$ for all
    $\lambda\in\Lambda$ and $\bfs,\bft\in T(\lambda)$.

    Then the triple $(\Lambda,T,C)$ is a \textsf{cell datum} for $A$ if:
    \begin{enumerate}
    \item $\set{C^\lambda_{\bfs\bft}|\lambda\in\Lambda\text{ and }\bfs,\bft\in
        T(\lambda)}$ is an $R$--basis for $A$;
    \item the $R$--linear map $*\map AA$ determined by
        $(C^\lambda_{\bfs\bft})^*=C^\lambda_{\bft\bfs}$, for all
        $\lambda\in\Lambda$ and all $\bfs,\bft\in T(\lambda)$ is an
        anti--isomorphism of $A$;
    \item for all $\lambda\in\Lambda$, $\bfs\in T(\lambda)$ and $a\in A$
        there exist scalars $r_{\bft\bfu}(a)\in R$ such that
        $$C^\lambda_{\bfs\bft} a
            =\sum_{\bfu\in T(\lambda)}r_{\bft\bfu}(a)C^\lambda_{\bfs\bfu}
                     \pmod{A^{\gdom\lambda}},$$
            where
    $A^{\gdom\lambda}=R\text{--span}%
      \set{C^\mu_{\bfu\bfv}|\mu\gdom\lambda\text{ and }\bfu,\bfv\in T(\mu)}$.
    \end{enumerate}
    \noindent Furthermore, each scalar $r_{\bft\bfu}(a)$ is independent of $\bfs$.
     An algebra $A$ is a \textsf{cellular algebra} if it has
    a cell datum and in this case we call
    $\set{C^\lambda_{\bfs\bft}|\bfs,\bft\in T(\lambda), \lambda\in\Lambda}$
    a \textsf{cellular basis} of $A$.
\end{Defn}

We  recall the representation  theory of cellular algebras in
\cite{GL}. Every irreducible $A$--module arises in a unique way as
the simple head of some cell module. For each $\lambda\in\Lambda$
fix $\bfs \in T(\lambda)$ and let $C^{\lambda}_{\bft}
       =C^{\lambda}_{\bfs\bft}+ A^{\rhd \lambda}$.
       The cell modules of $A$ are the modules $\Delta(\lambda)$
which are the free $R$--modules with basis
$\set{C^{\lambda}_{\bft}|\bft\in T(\lambda)}$. The cell module
$\Delta(\lambda)$ comes equipped with a natural bilinear form
$\phi_{\lambda}$ which is determined by the equation
$$C^{\lambda}_{\bfs\bft}
           C^{\lambda}_{\bft'\bfs}
 \equiv\phi_{\lambda}\big(C^{\lambda}_{\bft},
               C^{\lambda}_{\bft'}\big)\cdot
        C^{\lambda}_{\bfs\bfs}\pmod{ A^{\rhd \lambda}}.$$
The form $\phi_{\lambda}$ is $A$--invariant in the sense that
$\phi_{\lambda}(xa,y)=\phi_{\lambda}(x,ya^*)$, for
$x,y\in\Delta(\lambda)$ and $a\in A$. Consequently,
$$\Rad\Delta(\lambda)
   =\set{x\in\Delta(\lambda)|\phi_{\lambda}(x,y)=0\text{ for all }
                          y\in\Delta(\lambda)}$$
is an $A$--submodule of $\Delta(\lambda)$ and
$D^\lambda=\Delta(\lambda)/\Rad\Delta(\lambda)$ is either zero or
absolutely irreducible. Graham and Lehrer have proved  that
$\{D^\lambda\mid D^\lambda\neq 0\}$  consists of a complete set of
pairwise non-isomorphic irreducible $A$-modules.

Now, we use the representation theory  of a cellular algebra to
prove Theorem~\ref{W cellular}, the main result of this section.
\begin{Theorem}\label{W cellular}
Let $R$ be a commutative ring which contains the invertible elements
$q, u_1, u_2, \dots, u_r$ and $q-q^{-1}$.  Suppose that
$\Omega\cup\{\varrho\}$ is $\bu$--admissible. Let $\W_{r, n}$ be the
cyclotomic BMW algebras over $R$ with $2\nmid r$.  Then $\W_{r, n}$
is free over $R$ with
$$\mathscr C=\set{C^{(f,\lambda)}_{(\s,\rho,e)(\t,\kappa,d)}|
            (\s,\rho,e),(\t,\kappa,d)\in\delta(f,\lambda),
              \text{ where }(f,\lambda)\in\Lambda_{r, n}^+}$$
as its an $R$-basis. Further, $\mathscr C$ is a   cellular basis of
$\W_{r,n}(\bu)$.
\end{Theorem}

\begin{proof} By Proposition~\ref{Wlam spanning}, $\W_{r, n}$ is
an $R$-module spanned by $\mathscr C$. First, we assume $R=R_0$
where $R_0=\mathbb Z[\bu^{\pm 1} , q^{\pm 1}, (q-q^{-1})^{-1}]$ and
$\bu, q$ are indeterminates over $\mathbb Z$. we prove that
$\mathscr C$ is $R_0$-linear independent.
 As
$\mathbb R$ is not finitely generated over $\mathbb Q$, we can take
$r+1$ algebraically independent transcedental real numbers $v_i\in
\mathbb R$ and $\mathbf q$. We define $R_1=\mathbb Z[v_1, v_2 \dots,
v_r, \mathbf q^{\pm 1}, \delta^{\pm}]$. Then $R_1\cong R_0$ as ring
isomorphism. Therefore, $\W_{r, n}$ over $R_0$ is isomorphic to
$\W_{r, n}$ over $R_1$ as $R_0$-algebra.

We have constructed the seminormal representations for $\W_{r, n}$
with respect to all $\lambda\in \Lambda_r^+(n-2f)$, $0\le f\le
\floor{n2}$ under the conditions in Lemma~\ref{generic u} and
(\ref{be}). In particular, by Lemma~\ref{be real}, we have
seminormal representations of $\W_{r, n}$ over $\mathbb R$. We
remark that we are assuming that $\Omega\cup{\varrho}$ is $\mathbf
v$-admissible. By arguments in the proof of Theorem~5.3 in
\cite{AMR}, we have that $\Delta( \lambda)$ are  irreducible $\W_{r,
n}$-modules for all $\lambda\in \Lambda_r(n-2f)$ and $0\le f\le
\floor{n2} $. Further, $\Delta(\lambda)\not\cong \Delta(\mu)$ if
$\lambda\neq \mu$. By Wedderburn--Artin theorem on semisimple finite
dimension algebras,
$$\dim_{\mathbb R} \W_{r, n}\ge \dim_{\mathbb R} \W_{r, n}/\Rad \W_{r, n}\ge \sum_{(f,
\lambda)\in \Lambda^+_{r, n}} \# \UPD_{n}(\lambda)^2 =r^n(2n-1)!!,$$
the last equality follows from classical branching rule for
cyclotomic Brauer algebras, which was  proved in Theorem~5.11 in
\cite{RuiYu}. It was also proved in \cite[5.2]{AMR}. Therefore,
$\dim_{\mathbb R} \W_{r, n}=r^n(2n-1)!!$ and $\mathscr C$ is
$R_1$-linear independent. So is over $R_0$. This shows that
$\mathscr C$ is an $R_0$ basis of $\W_{r, n}$.  By base change,
$\mathscr C$ is an $R$--basis for an arbitrary commutative ring.
Further, by Proposition~\ref{Wlam spanning}, $\mathscr C$ is a
cellular basis of $\W_{r, n}$ as required.
\end{proof}

In Theorem~\ref{W cellular}, we have assumed that $r$ is odd. We
remark that the only place we  need this assumption is that we use
Proposition~\ref{Wlam spanning} to prove that  $\W_{r, n}$ is an
$R$-module spanned by $\mathscr C$.

When $r=1$, $\W_{r, n}$ is the usual BMW algebra defined in
\cite{BirmanWenzl}. It has been proved in \cite{Xi} that BMW algebra
is  cellular. Late,  Enyang gave an another proof of this result in
\cite{Enyang}.

\section{Classification of the irreducible $\W_{r,n}(\bu)-modules$}
In this section we assume that $F$ is a field which contains
invertible elements $u_1, \dots, u_r$, $q$ and $q-q^{-1}$. We also
assume that $\Omega\cup\{\varrho\}$ is $\bu$-admissible. By
Theorem~\ref{W cellular}, $\W_{r, n}$ is a subalgebra of $\W_{r,
n_1}$ if $n\le n_1$. Therefore, we will identify $\W_{r, n}$ with
$\W_{r, n}'$ defined in the previous section.

We are going to classify the irreducible $\W_{r,n}$--modules over
$F$. We remark that we assume that $r$ is odd.

\textsf{All modules considered in this section  are right modules.}

\begin{Lemma}\label{class2}
Given a positive integer $f\le \lfloor \frac n2\rfloor$. We have
$E^f \W_{r, n} E^f =\W_{r, n-2f} E^f$.
\end{Lemma}

\begin{proof} First, we assume that $f=1$.  By Theorem~\ref{W cellular},
the result follows if we prove $E_{n-1} h E_{n-1} \in
\W_{r,n-2}E_{n-1}$ for each cellular basis element $h =T_e^*X^\rho
E^fM_{\s\t}X^\kappa T_d$, where $
(\s,\rho,e),(\t,\kappa,d)\in\delta(f,\lambda)$.

By Lemma~\ref{l1}, $E_{n-1}T_e^*X^\rho E^fM_{\s\t}X^\kappa T_d\in
N_1$ where  $N_1$ is the $R$-submodule of $\W_{r, n}$  generated by
$\W_{r,  n-2} E_{n-1}X_{n-1}^{k}T_d$
 where $d\in\Dcal_{1,n}$ and $-p\le k \leq p$.
Further, by Corollary~\ref{et}(b) for $E_{n-1}X_{n-1}^{k}T_d
E_{n-1}$, we have $$E_{n-1}T_e^*X^\rho E^fM_{\s\t}X^\kappa T_d
E_{n-1}\in \W_{r, n-2} E_{n-1}.$$ The inverse inclusion follows
since $\W_{r, n-2} E_{n-1}=E_{n-1} \W_{r, n-2} E_{n-2}  E_{n-1}
\subset  E_{n-1} \W_{r, n} E_{n-1}$.
 Using the result for $f=1$
repeatedly, we have $E^f \W_{r, n} E^f =\W_{r, n-2f} E^f$ for all
positive integers $f\le \lfloor \frac n2\rfloor$.
\end{proof}

 It is  proved in
\cite{DJM:cyc} that $\cup_{\lambda\in \Lambda_r^+(n)}
\{\m_{\s\t}\mid \s, \t\in \Std(\lambda)\}$ is a cellular basis for
$\H_{r, n}$. Let $\Delta(\lambda)$ be the cell module of $\H_{r, n}$
defined by this cellular basis. Let $\phi_\lambda$ be the
corresponding symmetric associative bilinear form. Let $\phi_{f,
\lambda}$ be the symmetric associative bilinear form on the cell
module $\Delta(f, \lambda)$ which is defined via the cellular basis
of $\W_{r, n}$ given in  Theorem~\ref{W cellular}.

\begin{Lemma}\label{key21} Assume that
$(f,\lambda)\in\Lambda_{r,n}^+$. \begin{enumerate}\item Let $f\neq
n/2$. Then $\phi_{f,\lambda}\neq 0$ if and only if $\phi_\lambda
\neq 0$.
\item Let $f=n/2$,  and assume that $\omega_a\neq 0$ for some non-negative
integer $a\le r-1$. Then $\phi_{f, 0}\neq 0$.
\item If $\omega_i=0$
for all non-negative integers $i\le r-1$, then $\phi_{f,0}=0$ for
$f=n/2$.
\end{enumerate}
\end{Lemma}

\begin{proof} (a) can be proved by arguments similar to those for
\cite[3.1]{RS}.  In order to prove (b), we assume that $\ell\in
\mathbb Z$ and $k\in \mathbb Z^{\ge 0}$. We have
$$\begin{aligned}
&\omega_{2k+1}^{(\ell)}E_{2k+1}E_{2k-1}\cdots E_1\\
=&E_{2k+1}X_{2k+1}^{\ell}E_{2k+1}E_{2k-1}\cdots E_1\\
=&E_{2k+1}X_{2k+1}^{\ell}E_{2k-1}E_{2k}E_{2k-1}E_{2k+1}E_{2k-3}\cdots
E_1\\
=&E_{2k+1}E_{2k-1}X_{2k-1}^{\ell}E_{2k}E_{2k-1}E_{2k+1}E_{2k-3}\cdots
E_1\\
=&E_{2k-1}X_{2k-1}^{\ell}E_{2k+1}E_{2k}E_{2k+1}E_{2k-1}E_{2k-3}\cdots
E_1\\
=&E_{2k-1}X_{2k-1}^{\ell}E_{2k+1}E_{2k-1}E_{2k-3}\cdots
E_1\\
=&E_{2k+1}E_{2k-1}X_{2k-1}^{\ell}E_{2k-1}E_{2k-3}\cdots
E_1\\
=&E_{2k+1}E_{2k-1}E_{2k-3}\cdots E_1X_1^\alpha E_1, \quad \text {by induction assumption}\\
=&\omega_\ell E_{2k+1}E_{2k-1}E_{2k-3}\cdots E_1.\\
\end{aligned}
$$
 Since we are assuming that
$\Omega\cup\{\varrho\}$ is $\bu$-admissible, by Theorem~\ref{W
cellular},   $\mathscr C$ is an $F$-basis of $\W_{r, n}$. Since
$E_{2k+1}E_{2k-1}E_{2k-3}\cdots E_1\in \mathscr C$,
$\omega_{2k+1}^{(\ell)}=\omega_\ell$. So,
$$\phi_{\frac{n}{2},0}(E^{\frac{n}{2}},E^{\frac{n}{2}}X_{n-1}^\ell\cdots
X_3^\ell X_1^\ell)=(\omega_\ell )^{\frac{n}{2}}\neq 0.$$ This proves
(b).

Suppose that  $\alpha, \beta\in \Nrf$ for $f=n/2$ . Using
Lemma~\ref{class2} repeatedly, we have, for any $w \in\mathfrak
{S}_n$
$$E^{\frac n2} X^\alpha \cdot T_w\cdot
X^\beta E^{\frac n2}=E^{\frac n2}h E_1$$ for some $h\in \W_{r, 2} $.
By direct computation, $E_1 h E_1=0$ for all $h\in \W_{r, 2} $,
forcing $E^{\frac n2}h E_1=0$. Therefore, $\phi_{\frac n2, 0}=0$.
This proves (c).
\end{proof}

Lemma~\ref{key21} sets up a relationship between the irreducible
$\W_{r, n}$-modules and the irreducible $\H_{r, n-2f}$-modules for
all non-negative integers $f\le \floor{n2}$. Note that we can keep
the assumption that $u_i=q^{k_i}, k_i\in \mathbb Z$ by using
Dipper-James-Mathas's Morita equivalent theorem for $\H_{r, n-2f}$.
In ``separate condition", such a result was proved in \cite{DuRui}.
By \cite{AM:simples}, \cite{Ariki:can} and \cite{A5}, irreducible
$\H_{r, n-2f}$-modules  are indexed by \textsf{$\bu$-Kleshchev
$r$-multipartitions of $n-2f$}.

\begin{Theorem}\label{simplem}
Suppose $F$ is a field which contains non-zero elements $q, u_1,
\dots, u_r$ and $q-q^{-1}$. Assume that $\Omega\cup\{\varrho\}$ is
$\bu$-admissible. Let $\W_{r,n}$, $2\nmid r $ be the cyclotomic BMW
algebra over $F$.
\begin{enumerate}
\item If $n$ is odd, then the set of all pair-wise non-isomorphic
irreducible $\W_{r,n}$-modules are indexed by $(f, \lambda)$ where
$0\le f\le \lfloor \frac n2\rfloor$ and  $\lambda$ are \textsf{
$\bu$-Kleshchev multipartitions} of $n-2f$.
\item Suppose that $n$ is an even number.
\begin{enumerate}
\item If $\omega_i\neq 0$ for some non-negative integers $i\le r-1$, then the set of all pair-wise non-isomorphic
irreducible $\W_{r,n}$-modules are indexed by   $(f, \lambda)$ where
$0\le f\le \frac n2$ and  $\lambda$ are \textsf{ $\bu$-Kleshchev
multipartitions} of $n-2f$.
\item If $\omega_i= 0$ for all non-negative integers  $i\le r-1$, then the set of all pair-wise non-isomorphic
irreducible $\W_{r,n}$-modules are  indexed by $(f, \lambda)$ where
$0\le f< \frac n2$ and  $\lambda$ are \textsf{ $\bu$-Kleshchev
multipartitions} of $n-2f$.
\end{enumerate}
\end{enumerate}
\end{Theorem}

\providecommand{\bysame}{\leavevmode ---\ } \providecommand{\og}{``}
\providecommand{\fg}{''} \providecommand{\smfandname}{and}
\providecommand{\smfedsname}{\'eds.}
\providecommand{\smfedname}{\'ed.}
\providecommand{\smfmastersthesisname}{M\'emoire}
\providecommand{\smfphdthesisname}{Th\`ese}


\begin{thebibliography}{DWH99}

\bibitem{AK}
{\scshape S.~Ariki {\normalfont \smfandname} K.~Koike}, {\og A
{H}ecke algebra
  of {$({\bf {Z}}/r{\bf {Z}})\wr{\mathfrak {S}}\sb n$} and construction of its
  irreducible representations\fg}, \emph{Adv. Math.} \textbf{106} (1994),
  216--243.

\bibitem{AM:simples}
{\scshape S.~Ariki {\normalfont \smfandname} A.~Mathas}, {\og {The
number of
  simple modules of the Hecke algebras of type $G(r,1,n)$}\fg}, \emph{Math. Z.}
  \textbf{233} (2000), 601--623.

\bibitem{AMR}
{\scshape S.~Ariki {\normalfont \smfandname} A.~Mathas, H. Rui},
{\og {Cyclotomic Nazarov-Wenzl algebras}\fg}, \emph{Nagoya Math. J.}
  \textbf{182} (2006), 47--134.


\bibitem{Ariki:can}
{\scshape S.~Ariki}, {\og On the decomposition numbers of the
{Hecke} algebra
  of {$G(m,1,n)$}\fg}, \emph{J.~Math. Kyoto Univ.} \textbf{36} (1996),
  789--808.

\bibitem{Ariki:class}
\bysame , {\og On the classification of simple modules for
cyclotomic {Hecke
  algebras of type $G(m,1,n)$ and Kleshchev} multipartitions\fg}, \emph{Osaka
  J.~Math.} \textbf{38} (2001), 827--837.

\bibitem{A5}\bysame, {\og  Proof of the modular branching rule for cyclotomic Hecke
algebras\fg}, \emph{  J.   Algebra}\textbf {306} (2006) 290¨C300.



\bibitem{BrauerAlg}
{\scshape R.~Brauer}, {\og On algebras which are connected with the
semisimple
  continuous groups\fg}, \emph{Ann. of Math.} \textbf{38} (1937), 857--872.


\bibitem{BirmanWenzl}
{\scshape J.~S. Birman {\normalfont \smfandname} H.~Wenzl}, {\og
Braids, link
  polynomials and a new algebra\fg}, \emph{Trans. Amer. Math. Soc.}
  \textbf{313} (1989), 249--273.

\bibitem{DJM:cyc}
{\scshape R.~Dipper, G.~James {\normalfont \smfandname} A.~Mathas},
{\og
  Cyclotomic $q$--{Schur} algebras\fg}, \emph{Math.~Z.} \textbf{229} (1999),
  385--416.

\bibitem{DuRui}
{\scshape J. Du and H. Rui}, {\og
 Ariki-Koike algebras with semisimple bottoms\fg}, \emph{Math.~Z.} \textbf {234}, (2000), 807-830.


\bibitem{Enyang}
{\scshape J.~Enyang}, {\og Cellular bases for the {Brauer} and
  {BMW} algebras\fg}, \emph{J. Algebra} \textbf{281}
  (2004), 413--449.



\bibitem{GP}
{\scshape M. Geck and G. Pfeiffer}, \emph {Characters of Finite
Coxeter groups and Iwahori--Hecke Algebras}, Clarendon Press.
Oxford, 2000.

\bibitem{GoodmanHauschild}
{\scshape F.~M. Goodman {\normalfont \smfandname} H.~M. Hauschild},
{\og
  {Affine Birman-Wenzl-Murakami Algebras and Tangles in the Solid Torus}\fg},
  arXiv:math.QA/0411155.

\bibitem{GH1}\bysame,
{\og
  {Cyclotomic  Birman-Wenzl-Murakami Algebras, I: freeness and realization as tangle algebras }\fg},
  arXiv:math.QA/0612064.

\bibitem{GH2}\bysame,{\og  Cyclotomic
Birman--Wenzl--Murakami algebras, II: Admissibility Relations and
Representation theory\fg} arXiv:math.QA/0612065.

\bibitem{GL}
{\scshape J.~J. Graham {\normalfont \smfandname} G.~I. Lehrer}, {\og
Cellular
  algebras\fg}, \emph{Invent. Math.} \textbf{123} (1996), 1--34.

\bibitem{HO:cycBMW}
{\scshape R.~H{\"a}ring-Oldenburg}, {\og Cyclotomic
{B}irman-{M}urakami-{W}enzl
  algebras\fg}, \emph{J. Pure Appl. Algebra} \textbf{161} (2001), 113--144.

\bibitem{Macdonald}
{\scshape I.~G. Macdonald}, \emph{Symmetric functions and {H}all
polynomials},
  second \smfedname, Oxford Mathematical Monographs, Clarendon Press, Oxford,
  1995.



\bibitem{Nazarov:brauer}
{\scshape M.~Nazarov}, {\og Young's orthogonal form for {B}rauer's
centralizer
  algebra\fg}, \emph{J. Algebra} \textbf{182} (1996), 664--693.

\bibitem{OrelRam}
{\scshape R.~Orellana {\normalfont \smfandname} A.~Ram}, {\og
{Affine braids,
  Markov traces and the category O}\fg},
  \emph {Proceedings of the International Colloquium on Algebraic Groups and Homogeneous Spaces}
   Mumbai (2004), V.B. Mehta ed., Tata Institute of Fundamental Research, Narosa Publishing House, Amer. Math. Soc. (2007) 423-473.



\bibitem{RS}
{\scshape H.~Rui {\normalfont \smfandname} M.~Si}, {\og On the
structure of cyclotomic Nazarov--Wenzl algebras\fg},\emph{J. Pure
Appl. Algebra}, to appear.

\bibitem{RuiYu}
{\scshape H.~Rui {\normalfont \smfandname} W.~Yu}, {\og On the
semisimplicity
  of cyclotomic {Brauer} algebras\fg}, \emph{J. Algebra} \textbf{277} (2004),
  187--221.

\bibitem{Yu1}
{\scshape Shona Yu }, {\og The cyclotomic Birman-Murakami-Wenzl
algebras\fg}, Ph.D thesis, Sydney University, 2007.

\bibitem{WY} {\scshape S. Wilcox and S. Yu}, {\og The cyclotomic BMW
algebra associated with two string type $B$ braid group, \fg},
preprint, 2006.

\bibitem{Xi}
{\scshape C.~Xi}, {\og On the quasi-heredity of {Birman--Wenzl}
algebras\fg},
  \emph{Adv. Math.} \textbf{154} (2000), 280--298.

\bibitem{Xu}
{\scshape J.~Xu}, {\og Cyclotomic Birman-Murakami-Wenzl algebras
$\W_{r, 2}$,\fg} thesis for Master degree, East China Normal
University, June, 2006. In chinese.
\end{thebibliography}
\end{document}